\documentclass[10pt]{amsart}
\usepackage{amscd, amsfonts, amsthm, amsgen, amsmath, amssymb ,verbatim, enumerate, textcomp}
\usepackage{hyperref} 
\usepackage[cmtip, all]{xy}
\usepackage{graphicx}

\newtheorem{theorem}{Theorem}[section]

\newtheorem{lemma}[theorem]{Lemma}

\newtheorem{proposition}[theorem]{Proposition}
\theoremstyle{definition}
\newtheorem{definition}[theorem]{Definition}

\newtheorem{example}[theorem]{Example}

\theoremstyle{remark}
\newtheorem{remark}[theorem]{Remark}

\theoremstyle{keyobservation}
\newtheorem{keyobservation}[theorem]{Key Observation}

\newtheorem{shypothesis}[theorem]{Standing Hypothesis}
\numberwithin{equation}{section}

\numberwithin{equation}{subsection}
\newcommand{\be}%
  {\protect\setcounter{equation}{\value{subsubsection}}}  
  \newcommand{\ee}%
   {\protect\setcounter{subsubsection}{\value{equation}}}

  {\protect\setcounter{subsubsection}{\value{equation}}}

\numberwithin{equation}{subsection}

\newcommand{\GL}{\operatorname{GL}}
\newcommand{\Hom}{\operatorname{Hom}}

\newcommand{\Spec}{\operatorname{Spec}}
\newcommand{\Perf}{\operatorname{Perf}}
\newcommand{\Pscoh}{\operatorname{Pscoh}}

\newcommand{\bG}{\mathbf G}

\newcommand{\C}{\rm C}

\newcommand{\codim}{{\rm codim}}

\newcommand{\bK}{\mathbf K}

\renewcommand{\L}{\mathcal L}

\newcommand{\Z}{\mathcal Z}

\newcommand{\GG}{\mathbb G}

\renewcommand{\setminus}{\smallsetminus}

\renewcommand{\O}{\mathcal{O}}

\def\displaytimes_#1{\mathrel{\mathop{\times}\limits_{#1}}}

\def\displayotimes_#1{\mathrel{\mathop{\bigotimes}\limits_{#1}}}

 \def\ari[#1]{\ar@{^(->}[#1]}
 \def\are[#1]{\ar[#1]^{\txt{\'et}}}
 \def\areh[#1]{\ar[#1]|{\txt{$H$-eq}}^{\txt{\'et}}}
 \def\ars[#1]{\ar@{->>}[#1]}
 \newcommand{\dplus}{\ar@{}[d]|{\mbox{$\oplus$}}}
 \newcommand{\dtimes}{\ar@{}[d]|{\mbox{$\times$}}}

\def \rmA{\rm A}
\def \rmB{\rm B}
\def \rmBG{\rm {BG}}
\def \C{\mathcal C}

\def \Cl{\mathbb C}

\def \colimm{\underset {m \rightarrow \infty}  {\hbox {lim}}}
\def \colimn{\underset {n \rightarrow \infty}  {\hbox {lim}}}

\def \colimalpha{\underset {\alpha}  {\hbox {colim}}}

\def \colimK.{\underset {\underset K^.  \rightarrow}  {\hbox {lim}}}

\def \colimU.{\underset {\underset U_.  \rightarrow}  {\hbox {lim}}}

\def \cosimp1{\stackrel{\rightrightarrows}{ \leftarrow}}

\def \compl{\, \, {\widehat {}}}
\def \rmE{\rm E}
\def \rmEG{\rm {EG}}
\def \rmEH{\rm {EH}}

\def \EG1{E{(G \times {\mathbb C}^*)}{\underset {G\times {\mathbb C}^*} \to \times}}

\def \EZ(s)1{E{(Z(s) \times {\mathbb C}^*)}{\underset {(Z(s)\times {\mathbb C}^*)} \to \times}}

\newcommand{\eps}{\boldsymbol\varepsilon}

\def \eps{\ \epsilon \ }
\def \EM(u){EM(u){\underset {M(u)} \to \times}}
\def \EM(us){EM(u,s){\underset {M(u, s)} \to \times}}

\def \bG{{\mathbf G}}

\def \rmG{\rm G}
\def\holimD{\mathop{\textrm{holim}}\limits_{\Delta }}
\def\holimDm{\mathop{\textrm{holim}}\limits_{\Delta_{\le m} }}
\def\hlimDone{\mathop{\textrm{holim}}\limits_{\Delta_{\le m_1} }}
\def\hlimD2{\mathop{\textrm{holim}}\limits_{\Delta_{\le m_2} }}
\def\hlimDn{\mathop{\textrm{holim}}\limits_{\Delta_{\le m_n} }}

\def \holimt {\underset {\infty \leftarrow t}  {\hbox {holim}}}
\def \holimr {\underset {\infty \leftarrow r}  {\hbox {holim}}}

\def \hocolimD{\underset \Delta  {\hbox {hocolim}}}

\def \holimn {\underset {\infty \leftarrow n}  {\hbox {holim}}}
\def \holimq {\underset {\infty \leftarrow q} {\hbox {holim}}}

\def \holimk{\underset {\infty \leftarrow k}  {\hbox {holim}}}

\def\holim{\mathop{\textrm{holim}}}

\def \H{\mathbb H}
\def \rmH{\rm H}

\def \Hom{\underline {Hom}}
\def \Hom{{\mathcal H}om}

\def \holimm {\underset {\infty \leftarrow m}  {\hbox {holim}}}
\def \invlim1{\underset {\infty \leftarrow q} \to {\hbox {lim}}^1}
\def \rmI{\rm I}

\def \rmK{\rm K}
\def \L3{\Lambda \times \Lambda \times \Lambda}
\def \L2{\Lambda \times \Lambda}

\def \limr{\underset {\infty \leftarrow r}  {\hbox {lim}}}

\def \limbeta{\underset { \beta }  {\hbox {lim}}}

\def \limm{\underset {\infty \leftarrow m}  {\hbox {lim}}}
\def \limn{\underset {\infty \leftarrow n}  {\hbox {lim}}}

\def \longright2arrow{{\overset \longrightarrow \to {\overset {} \to \longrightarrow}}}

\def \L{L\times \Cl ^*}

\def \M{\mathcal M}
\def \rmMap{\rm {Map}}
\def \Map{\underline {Map}}
\def \Map{{\mathcal M}ap}

\def \N{\mathbb N}

\def \O{{\mathcal O}}

\def \P{\mathbb P}
\def \ra{\rightarrow}
\def \Ra{\Rightarrow}
\def \RG^{R(G)^{\hat {}}\ }

\def \res{respectively}

\def \R{{\mathcal R}}
\def \rmR{\rm R}

\def \rmS{\rm S}
\def \S{\mathcal S}

\def\Spt{\rm {\bf Spt}}

\def \topGcoh*{^{top, *} _{G}}
\def \topGho*{ _{top,*} ^{G}}

\def \rmT{\rm T}

\def \rmU{\rm U}

\def \rmV{\rm V}
\def \wedgeK{\overset L {\underset {\bK (\rmS, G)} \wedge}}
\def \wedgeKH{\overset L {\underset {\bK (\rmS, H)} \wedge}}

\def \rmW{\rm W}
\def \rmX{\rm X}

\def \rmY{\rm Y}

\def \Z(s){Z(s) \times {\mathbb C}^*}
\def \Z{\mathbb Z}

\def \rmZ{\rm Z}
\begin{document}

\title{Atiyah-Segal Derived Completions for Equivariant Algebraic G-Theory and K-Theory}
\author{Gunnar Carlsson}
%  Address of record for the research reported here
%\nopagebreak
\address{Department of Mathematics, Stanford University, Building 380, Stanford,
California 94305}
\email{gunnar@math.stanford.edu}
\thanks{  }  
\author{Roy Joshua}
\address{Department of Mathematics, Ohio State University, Columbus, Ohio,
43210, USA.}
\email{joshua.1@math.osu.edu}
    %\thanks will become a 1st page footnote.
\thanks{ The second author was supported by a grant
from the NSF}

\maketitle

\begin{abstract}
In the mid 1980s, while working on establishing completion theorems for equivariant Algebraic K-Theory similar to the
well-known Atiyah-Segal completion theorem for equivariant topological K-theory, 
the late Robert Thomason found the
strong finiteness conditions that are required in such theorems to be too restrictive. Then he made a  conjecture on the existence of a completion theorem in the sense of
 Atiyah and Segal for
equivariant Algebraic G-theory, for actions of linear algebraic groups on schemes that holds without any of the
strong finiteness conditions that are required in such theorems proven by him, and also appearing in the
original Atiyah-Segal theorem.  {\it The main goal of the present paper is to provide a proof of this conjecture
in as broad a context as possible, making use of the technique of derived completion, and to consider several of the applications.} 
\vskip .2cm
Our solution is
broad enough to allow actions by all linear algebraic groups, irrespective of whether they are connected or not, and
acting on any quasi-projective scheme of finite type over a field, irrespective of whether they are regular or projective.
This allows us therefore to consider the Equivariant Algebraic G-Theory of large classes of varieties like all Toric varieties 
(for the action of a torus) and all Spherical varieties (for the action of a reductive group). Restricting to actions by split tori,
we are also able to consider actions on Algebraic Spaces. Moreover, the restriction that the base scheme be a field is
also not required often, but is put in mainly to simplify some of our exposition. These enable us to obtain a wide range of applications,
some of which are briefly sketched and which we plan to explore in detail in the future. A comparison of our results
with previously known results, none of which made use of derived completions, shows that without the use of derived completions, one can only obtain results which are indeed very restrictive.
\end{abstract} 
\setcounter{tocdepth}{1}
\tableofcontents\setcounter{tocdepth}{1}

\section{\bf Introduction}
\vskip .2cm
An important  problem in algebraic K-theory is to understand the behavior of equivariant K-theory.   
This is important not only for manifestly  equivariant problems, but  also for descent problems arising in 
non-equivariant problems over a non-algebraically closed base field, or more generally a scheme. 
An approach is suggested by powerful results in equivariant homotopy theory \cite{AS69}, \cite{C84}.  
For both of the referenced results, one finds that an 
equivariant theory (equivariant topological K-theory in the case of \cite{AS69}, 
equivariant stable homotopy theory in the case of \cite{C84}) satisfies an approximation result 
that asserts that the equivariant theory is equivalent to a homotopy fixed point set, after application 
of a suitable notion of completion. Homotopy fixed point sets can be considered homotopy theoretic information, 
since there are standard spectral sequences for computing them.  Informally, we will say that 
equivariant topological K-theory and equivariant stable homotopy theory are {\em homotopy computable}.  
The approach to analyzing equivariant algebraic K-theory is now to ask to what extent it too is homotopy computable.  The answer we provide
 is that Equivariant algebraic G-theory is also homotopy computable, but with a somewhat different notion of homotopy fixed
  point sets more suited to motivic frameworks, such as those in \cite{tot} or \cite{MV}, thereby also solving affirmatively the conjecture posed in \cite{Th86}.
 \vskip .2cm
Apart from \cite{C11}, derived completions have never been used in the context of the Atiyah-Segal framework for equivariant Algebraic K-theory. The focus in \cite{C11} is on actions
of profinite groups, such as the absolute Galois groups of fields, and therefore, is essentially disjoint from the context of actions by linear algebraic
 groups considered in this paper. However, work in progress by the authors suggest that the results of the present paper may
 also find application even when the primary focus is actions by profinite groups and also may have a bearing on the Norm-Residue
 Theorem: see \cite{HW}, \cite{V11}. We expect our results to have a number of applications, several of which are sketched later on (see section 6), such as
  equivariant forms of higher Riemann-Roch theorems, computation of the homotopy groups of the derived completion using the motivic Atiyah-Hirzebruch spectral sequences
   as well as to equivariant forms of the Weibel conjecture: see \cite{CHSW}. 
\vskip .2cm
But first, it seems worthwhile revisiting the results of Atiyah and Segal in their landmark paper in the 1960s: see \cite{AS69}.
If $\rmX$ is a compact topological space provided with the action of a compact Lie group,
we let ${\rm K}_0^{\rmG}(\rmX)$
(${\rm K}_0(\rmE \rmG\times _{\rmG} \rmX)$) denote the Grothendieck group of isomorphism classes of 
$\rmG$-equivariant topological vector bundles (the corresponding topological K-theory in degree $0$ 
 of the Borel construction ${\rm EG} \times_{\rm G} \rmX$, \res). 
The relationship between ${\rm K}_0^{\rmG}(\rmX)$ and ${\rm K}_0({\rm EG}\times _{\rmG} \rmX)$ was studied in this setting in \cite{AS69}.
 Their main result  was that, under strong assumptions on $\rmX$ which would imply
that ${\rm K}_0^{\rmG}(\rmX)$ is finite as a module over the representation ring ${\rm R}(\rmG)$, ${\rm K}_0({\rm EG}\times _{\rmG} \rmX)$ is the completion of ${\rm K}_0^{\rmG}(\rmX)$ with respect to the
augmentation ideal in ${\rm R}(\rmG)$. \footnote{The need for completion in this setting seems to be because of the failure of compactness of
${\rm EG}\times _{\rmG} \rmX$ in general. But this observation is of secondary importance for us, since the main goal of the paper
is to be able to relax the strong finiteness assumptions in the Atiyah-Segal theorem by making use of what is called derived completions.}
\vskip .2cm
The need for such strong assumptions may be understood by recalling completions for commutative rings. Let ${\rm A}$ denote a commutative ring with 
a multiplicative unit and let ${\rm I}$ denote an ideal in ${\rm A}$.
If ${\rm M}$ is an ${\rm A}$-module,
recall that the ${\rm I}$-adic completion of ${\rm M}$ along ${\rm I}$ is  ${\rm M}\, \widehat{} \,  =\limn {\rm M}/{\rm {I^nM}}$. Unfortunately, the functor 
${\rm M} \mapsto {\rm M}\, \widehat{} \, $ is
neither right-exact nor left-exact, in general : see \cite[Chapter 10, Proposition  10.12 ]{AM} for basic results on completion in the 
Noetherian 
case. Therefore,  one needs the higher derived functors of the above completion functor, which are
nontrivial in general, i.e. when ${\rm A}$ is no longer required to be Noetherian and/or ${\rm M}$ is no longer required to be finitely generated over ${\rm A}$.
\vskip .2cm
For the case considered in the second paragraph, the ring in question is ${\rm R}(\rmG)$ which is Noetherian, so that strong assumptions on $\rmX$ 
ensure that the ${\rm R}(\rmG)$-module, ${\rm K}_0^{\rmG}(\rmX)$ is finite, which then ensure that the completion is well-behaved. 
\vskip .2cm
The corresponding relationship 
for the action of algebraic groups on schemes was studied by Thomason in a series of papers in the
1980s: see \cite{Th83}, \cite{Th86} and \cite{Th88}. In his seminal paper \cite{Th86}, Thomason proves that with 
certain strong restrictions on an algebraic group $\rmG$ and scheme $\rmX$ with $\rmG$-action, a Bott periodic form of algebraic 
$G$-theory is indeed homotopy computable. The result is clearly very powerful, but it suffers from two deficiencies. 

\begin{enumerate}
\item{It requires that we deal with the Bott-element inverted algebraic $G$-theory with finite coefficients prime to the characteristic.
While it is true that there is a strong relationship between algebraic $G$-theory and its Bott-element inverted version, this relationship 
is not strong enough to permit exact calculations.  In fact the Bott-element inverted version of G-theory with finite coefficients
is essentially topological (or \'etale) G-theory, so that Thomason's results are only for this form of G-theory and
not for Algebraic G-theory itself. 
\vskip .2cm \noindent
We will in fact present an example in ~\ref{counter.eg} that shows that the map from equivariant algebraic G-theory with finite coefficients
to the corresponding equivariant algebraic G-theory with the Bott-element inverted is {\it not} an isomorphism for many varieties,
so that Thomason's results do {\it not} provide an Atiyah-Segal type completion theorem for equivariant Algebraic G-theory,
but only for its topological variant, namely the Bott-element inverted form of equivariant algebraic G-theory.
Therefore, one needs a completion theorem for equivariant algebraic G-theory itself. }
\item{There are certain awkward restrictions to the result, in particular the requirement that the scheme or algebraic space be over a 
separably closed base field or possibly a scheme with finitely generated K-groups, such as rings of integers in number 
fields.  
Such restrictions are mainly to ensure that the form of equivariant algebraic
K-groups that are considered are finite over ${\rm R}(\rmG)$. }
\end{enumerate}
\vskip .2cm   
The homotopy fixed point scheme that comes up in Thomason's results is based on a simplicial model of the classifying space for 
an algebraic group and this seems to be enough primarily because Thomason is considering Algebraic G-theory with the
Bott-element inverted. To be able to consider Algebraic G-theory itself, one needs a model of classifying spaces that 
is more suited to motivic contexts, such as those in \cite{tot} or \cite{MV}, which is what we consider in this paper. More details on Thomason's theorem and his strategy are discussed in section ~\ref{sec.egs}
\vskip .2cm
A recent attempt to prove an Atiyah-Segal type theorem for equivariant Algebraic K-theory, (i.e. without 
using finite coefficients or inverting the Bott element) 
is \cite[Theorem 1.2]{K}. (See also \cite{KN}.) This makes use of the usual completion of the homotopy groups of the K-theory spectrum at the augmentation ideal of the representation ring.
Though \cite[Theorem 1.2]{K} avoids Thomason's strong hypotheses, and could very well be the strongest result that could be obtained
 by the traditional completion, {\it  
the following strong list of alternate hypotheses, (i) through (iii), are required. Dropping any one of them means that the results only hold for the Grothendieck
   groups as \cite[Theorem 1.3]{K} shows.} See section 6 for more details.
\vskip .2cm
\begin{enumerate}[\rm (i)]
 \item  The groups considered for the action have to be all {\it connected split reductive groups defined over a field}.
 In particular, this means {\it no finite groups, however simple, (for example, elementary abelian groups) 
  are allowed, though finite group actions are very important in studying descent problems. }
 \item The objects considered have to be {\it schemes of finite type over a field.} In particular, this means natural objects 
 like {\it algebraic spaces cannot} be considered. 
  \item The schemes  have to be {\it both projective as well as smooth}. In particular, varieties that are non-smooth, like
  {\it normal varieties} are not allowed in general. Recall that {\it toric varieties and spherical varieties} are defined to be
  {\it normal varieties} satisfying certain further hypotheses. This shows that large classes of commonly occurring varieties like
  {\it toric varieties that are singular or non-projective (when the group is a torus) are not allowed}. Similarly,
  the large class of varieties called {\it spherical varieties (when the group is a reductive group, 
  even as simple as ${\rm GL}_n$ or
  ${\rm SL}_n$) that are singular or non-projective cannot be considered}. 
  A variant of this is the counter-example discussed in \cite[Theorem 1.5]{K} where the scheme
   is the homogeneous space ${\mathbb G}_m/\mu_2$ for the action of ${\mathbb G}_m$, all over ${\mathbb C}$.
   \item The assumption of projective smoothness over a field and the restriction to actions by connected
   split reductive groups is to be able  to reduce to the case of schemes stratified by strata that are affine spaces
   over a smooth projective scheme with trivial action, by making strong use of Bialynicki-Birula decomposition. All toric and spherical varieties have
   only finitely many fixed points for the action of a maximal torus. {\it This means the only toric and spherical varieties
   that are allowed by \cite[Theorem 1.2]{K} are those schemes stratified by strata that are affine spaces, which 
   forms an extremely narrow class of varieties. Moreover for such varieties over separably closed fields, the
   map from equivariant algebraic K-theory to equivariant algebraic K-theory with the Bott element inverted, and both with coefficients
    prime to the characteristics, is an isomorphism in non-negative degrees, so that \cite[Theorem 1.2]{K} in this case is implied by Thomason's Theorem
    (i.e. \cite[Theorem 3.2]{Th86}.)}
   
\end{enumerate}

\vskip .2cm

In fact, Thomason clearly understood  the source of the difficulties and foresaw the key ingredients necessary to resolve them.  We have the following quote 
  from one
of Thomason's papers (see [Th86], pp. 795-796) (slightly edited to remove references to the context of his paper):
\vskip .1cm
{\it ``This is because the isomorphisms (involving the comparison between the two types of equivariant theories)
is expressed at the superficial level of the corresponding homotopy groups, rather than the deeper level
of spectra, and so depend on the $I_G$-adic completion process being exact on the level of
the homotopy groups. This necessitates strong finiteness assumptions which can be met only in special cases.
If the proper homotopy theoretic construction on the spectrum of equivariant K-theory could be found, that would
induce $I_G$-adic completion on the homotopy groups in nice cases, and do something more complicated involving
the derived functors of the inverse limit functor  in general, then it could be used to formulate an extension
of the Atiyah-Segal completion theorem to more general settings. It would also extend the Atiyah-Segal theorem  without their 
finiteness hypotheses.''}
\vskip .2cm
The first author has constructed exactly such a homotopy theoretic construction (called {\it derived completion}) in \cite{C08} and the results of the present paper
make use of this derived completion following the approach conjectured by Thomason. 
A comparison of our main theorems, Theorems ~\ref{main.thm.1}, ~\ref{comparison} below 
 show that, they do not suffer from any of the above restrictions that showed up in \cite{Th86} or \cite{K}.
\begin{enumerate}[\rm(i)]
 \item As shown below, actions of any linear algebraic group on any quasi-projective scheme (all defined over a field) are allowed, and the group need {\it not} be connected. This includes all finite groups
 and diagonalizable groups when they are imbedded as subgroups of some ${\rm {GL}}_n$. (See Remark ~\ref{remark.validity} below.)
In this context, the base scheme need not be a field, and any general Noetherian regular scheme of finite type over field
 is allowed, though for the sake of simplicity we only discuss the case where the base scheme is a field.
\item Algebraic spaces are allowed as long as the groups acting on them are split tori (or diagonalizable subgroups of them).
  \item Any normal quasi-projective $\rmG$-scheme, for the action of a linear algebraic group $\rmG$ is allowed as shown below.
 This allows both singular schemes as well as schemes that are not projective, thereby allowing large classes of schemes
 like {\it all} toric varieties, {\it all} spherical varieties and {\it all} linear varieties. 
 In fact, the counter-example discussed in ~\ref{counter.eg} shows
 that Thomason's theorem
 does not imply a completion theorem in mod-$\ell^{\nu}$ equivariant algebraic K-theory for large classes of toric and spherical varieties. 
 (See the discussion in ~\ref{counter.eg} for more details.)
 Clearly,
 neither does \cite[Theorem 1.2]{K} unless these varieties are stratified by affine spaces (in which case the result is implied
 by Thomason's theorem, at least over separably closed fields and with finite coefficients prime to the characteristic, as pointed out above).
 \item Our Theorem ~\ref{comp.2} shows that when applied to projective smooth varieties over a field, the homotopy groups of the derived
 completion identify with  the usual completion of the homotopy groups, so that we recover all the results of \cite{K} 
 as a corollary to our results.
 \item Making use of our results, and combining with the usual spectral sequence relating motivic cohomology
 and algebraic K-theory, we obtain a computation of the equivariant K-groups of smooth schemes after applying derived completions, in terms
 of their equivariant motivic cohomology groups. We also obtain as corollaries, strong forms of equivariant Riemann-Roch theorems valid for all normal quasi-projective $\rmG$-schemes
  for the action of any linear algebraic group $\rmG$, and
 involving higher equivariant G-theory and equivariant motivic (and other forms of) cohomology. (These applications are 
  discussed in section 6.)
 \item More details on all of these are discussed in section 6.
\end{enumerate}
\vskip .2cm
Before we proceed further, a few brief comments on the subtle nature of
derived completion seem to be also in order. 
\begin{enumerate}[\rm(i)]
\item First, observe that {\it a key problem} with completion with respect to an ideal in a ring
is that, in general,  it is {\it neither left exact nor right exact}. Therefore, a key role of derived completion is to rectify this problem: i.e. one has to rectify the failure of both left-exactness and right exactness. This means that a simple-minded approach 
by taking either a projective resolution (i.e. a cofibrant replacement) or an injective resolution (i.e. a fibrant replacement)
will {\it not suffice}. 
\item
Instead, one has to combine suitable cofibrant and fibrant replacements, and also at the
same time, the resulting derived completion for spectra has to be computable in terms of the derived completion of
their homotopy groups. 
\item
As Thomason states in his conjecture, the derived completion also has to have the derived functor
of the inverse limit functor built into it. 
\item What the above constraints mean is that the derived completion functor is rather delicate and cannot be obtained by starting with the usual completion functor at the level of homotopy groups and performing improvements to it.  \footnote{The derived completion also has little to do with the methods of derived algebraic geometry.} 
Finally, the technology to resolve the above issues needed important improvements in the theory of spectra and
stable homotopy theory (see for example, \cite{Ship07}, \cite{Ship04}) that became available only during the last 10 years or so.
\end{enumerate}
These are among the reasons, why the existence of a derived completion with desirable 
properties remained a conjecture for many years and the issues were all successfully resolved only in the 2008 paper
\cite{C08} by the first author. Moreover, in order to apply those results to the context of this paper, we had to 
amplify some of those results: these are discussed in section 3 as well as the appendices A and C. {\it The fact that
we are able to resolve successfully all the problems discussed above with the Atiyah-Segal completion theorem, shows
the effectiveness of the derived completion.}
\vskip .2cm
Apart from \cite{C11} (see also \cite{C13}), derived completions have never been used in the context of the Atiyah-Segal framework for equivariant Algebraic K-theory. The focus in \cite{C11} (and also in \cite{C13}) is on actions
of profinite groups, such as the absolute Galois groups of fields, and therefore, is essentially disjoint from the context of actions by linear algebraic
 groups considered in this paper. However, work in progress by the authors suggest that the results of the present paper may
 also find application even when the primary focus is actions by profinite groups and also may have a bearing on the Norm-Residue
 Theorem: see \cite{HW}, \cite{V11}.
 \vskip .2cm
Though we state our main results only when the base scheme $\rmS$ is the spectrum of a field, as pointed out above, several of our results
extend to a more general setting. Therefore, 
in order to extend the validity of our results to as general a context as possible, we will assume the following 
framework for the paper.
The base scheme $\rmS$ will be a separated Noetherian scheme which is regular. For the most part, this means $\rmS$ will be the spectrum of a 
finitely 
generated algebra over a field $k$, which is regular. 
Our approach is general enough that some of our results  also extend to actions of group schemes (especially tori) on algebraic spaces. The  algebraic spaces 
and schemes we consider will always be assumed to be {\it Noetherian and separated over the base scheme} $\rmS$. 
The group schemes we consider will be
affine group schemes which are finitely presented, separated and faithfully flat over $\rmS$. When we say a group scheme $\rmG$ is affine, we mean that
it admits a closed immersion into some $\GL_{n, \rmS}$, for some integer $n >0$. 
Observe that any finite group $\rmG$ may be viewed as an affine group-scheme over $\rmS$ in the obvious manner, for example, 
by imbedding it as a subgroup of the monomial matrices in some $\GL _n$.
\vskip .2cm
{\it Our main results will be stated only under the following basic assumptions.} We will assume the base scheme $\rmS$ 
is a field $k$ of
arbitrary characteristic $p$ and {\it though we will allow actions by any linear algebraic group (connected or not), the derived completion we use will be
with respect to an ambient bigger group} which satisfies certain stronger conditions as specficied below.\footnote{This is in view of Theorem ~\ref{comparison} (below), which shows that the
derived completion with respect to a closed subgroup and an ambient bigger group are not the same in general.} The {\it ambient group} will be denoted
by $\rmG$ while {\it the closed subgroup} will be denoted $\rmH$ throughout. 
\vskip .1cm
\begin{shypothesis}
 \label{stand.hyp.1} Henceforth the ambient group $\rmG$ will denote a connected split reductive group over a field $k$ so that if $\rmT$ denotes a
maximal torus in $\rmG$, then $\rmR(\rmT)$ is free of finite rank over $\rmR(\rmG)$ and $\rmR(\rmG)$ is Noetherian. We will also restrict to $\rmG$-schemes $\rmX$ of finite type over $k$
for which there exists a closed $\rmG$-equivariant immersion  into a regular  $\rmG$-scheme $\tilde \rmX$ of finite type over $k$. {\it It is important
to point out that restriction to such affine group schemes does not restrict the range of applications at all: we discuss this in the following key observations
 ~\ref{remark.validity} below.}
 \vskip .1cm
 This hypothesis is satisfied by 
$\rmG=GL_n$ or ${\rm {SL}}_n$, for any $n$, or any finite product of these groups. (A basic hypothesis that guarantees this condition is that the algebraic fundamental group 
$\pi_1(\rmG)$ is 
torsion-free: see section ~\ref{rep.ring}.) 
\end{shypothesis}
\vskip .2cm
\begin{keyobservation}
 \label{remark.validity}
 \begin{enumerate}[\rm(i)]
\item Observe that any linear algebraic group $\rmH$ can be imbedded into $\rmG$ (as a closed sub-group-scheme), where $\rmG$ is a general linear group (i.e. a $\GL_n$) or a finite product of such groups.
 It is important to observe that, then the $\rmH$-equivariant G-theory (K-theory) of a scheme $\rmX$ is weakly equivalent to
the $\rmG$-equivariant $G$-theory ($K$-theory) of $\rmG{\underset {\rmH} \times}X$. i.e.
\be \begin{equation}
     \label{key.obs.1}
     \bG(\rmX, \rmH) \simeq \bG(\rmG{\underset {\rmH} \times}\rmX, \rmG) \mbox{ and } \bK(\rmX, \rmH) \simeq \bK(\rmG{\underset {\rmH} \times} \rmX, \rmG)
\end{equation} \ee
Moreover, if $\rmX$ is a quasi-projective scheme, then so is $\rmG{\underset {\rmH} \times}\rmX$. This enables us to just consider the equivariant $\bG$-theory and $\bK$-theory
with respect to the action of the ambient group $\rmG$.
\item If $\rmH$ is a connected split reductive group,
 then $\rmR(\rmH) \cong \rmR(\rmT)^W$, where $W$ denotes the Weyl group of $\rmH$. Therefore $\rmR(\rmH)$ is Noetherian and the same holds
  for any linear algebraic group. Therefore, one may find a closed imbedding $\rmH \ra \GL_{n_1} \times \cdots \GL_{n_m}$ for some $n_1, \cdots, n_m$ such that the
restriction $\rmR(\GL_{n_1} \times \cdots \GL_{n_m}) \ra \rmR(\rmH)$ is surjective, so that Theorem ~\ref{comparison} below applies. 
\item
In view of these observations, assuming the ambient group $\rmG$ is a finite product of $\GL_n$s is not a serious restriction at all.
\end{enumerate}
\end{keyobservation}

\vskip .2cm 

Recall from \cite[Theorem 2.5]{Sum} that if $\rmG$ is a connected linear algebraic group and $\rmX$ is any $\rmG$-scheme which is normal and quasi-projective over the base scheme $\rmS= Spec \, {\it k}$, one may
find a $\rmG$-equivariant locally closed immersion of $\rmX$ into a projective space ${\mathbb P}^n$ on which $\rmG$  acts linearly. Therefore, one may find an
open $\rmG$-stable subscheme $\tilde X$ of this ${\mathbb P}^n$  into which $\rmX$ admits a $\rmG$-equivariant closed immersion. (To see this, observe first that
one may find an open subscheme $\rmU$ of the above projective space into which $\rmX$ admits a closed immersion. If we let $\rmV$ denote the image
of $\rmG \times \rmU$ under the action map $\rmG \times {\mathbb P}^n \ra {\mathbb P}^n$, then $\rmV$ is open and $\rmG$-stable and contains $\rmX$ as a locally closed subscheme.
It suffices to show that $\rmX$ is in fact closed in $\rmV$, which should follow from the fact that any point of $\rmV$ is in the orbit of a point of $\rmU$
and because $\rmX$ is closed in $\rmU$.) Clearly such an $\tilde X$ is regular. (It is shown in \cite[p. 629]{Th88} how to
remove the restriction that $\rmG$ be connected in the above discussion. See also ~\ref{G.quasiproj}.)
\vskip .2cm
Let $\rmX$ denote such a scheme provided with an action by
 $\rmG$. For the purposes of this introduction, we will let $\bK({\rm X}, \rmG)$ ($\bG({\rm X}, \rmG)$) denote the
spectrum obtained from the category of $\rmG$-equivariant vector bundles (coherent sheaves) on $\rmX$, though a more precise
definition is given in section 2.
\vskip .2cm
We will let  $\rmE{\rmG}^{gm}$  denote the geometric classifying space
for $\rmG$ which is constructed in ~\ref{geom.class.space} as an ind-object of schemes. We let 
$\rho_{\rmG}: \bK (\rmS, \rmG) \ra \bK( \rmS)$ denote the map of commutative
ring spectra defined by restriction to the trivial subgroup-scheme. (See section ~\ref{der.completion} for further details).
 For a prime $\ell \ne p$, let 
$\rho_{\ell}: \Sigma \ra \H(\Z/\ell)$ denote the mod$-\ell$ reduction map and let
$\rho_{\ell} \circ \rho_{\rmG}: \bK(\rmS, \rmG) \ra \bK(\rmS) {\underset {\Sigma} \wedge} \H(\Z/\ell)$ 
denote the composition of $\rho_{\rmG}$ and the mod$-\ell$ 
reduction map $id_{\bK (\rmS)} \wedge \rho_{\ell}:\bK(\rmS) \ra \bK(\rmS){\underset {\Sigma} \wedge} \H(\Z/\ell)$.  The K-theory and G-theory of these objects are defined in Definition ~\ref{Borel.equiv.th}.
\vskip .2cm
We would also like to point out that Remark ~\ref{remark.validity}(i) shows that the hypotheses
of the theorem below do not put any major restrictions on the group actions that are allowed.
\begin{theorem} 
\label{main.thm.1}
Assume that the base scheme $\rmS= Spec \,{\it k}$ for a field $k$ and that
$\rmX$ denotes any scheme of finite type over $S$ satisfying the hypotheses as in ~\ref{stand.hyp.1} and provided with an action by $\rmG$. 
Let $\rmH$ denote a closed sub-group-scheme of $\rmG$ and let $\rmE{\rmG}^{gm}{\underset {\rmH}  \times}X$ denote the ind-scheme defined by the Borel construction as in
section ~\ref{geom.class.space}.
\begin{enumerate}[\rm(i)]
\item Then the map $\bG({\rm X}, \rmH) \simeq \bK(\rmS, \rmH) \wedgeKH \bG({\rm X}, \rmH) \ra \bG(E{\rmG}^{gm}, \rmH) \wedgeKH \bG({\rm X}, \rmH) \ra 
\bG(E{\rmG}^{gm}\times {\rm X}, \rmH) \simeq \bG(E{\rmG}^{gm}{\underset {\rmH}  \times}X)$ 
factors through the derived completion of $\bG({\rm X}, \rmH)$ at $\rho_{\rmG}$  and induces a weak-equivalence 
\vskip .2cm
$\bG({\rm X}, \rmH) \compl_{\rho_{\rmG}} {\overset {\simeq} \ra} \bG(E{\rmG}^{gm}{\underset {\rmH}  \times}X).$
\vskip .2cm \noindent
The spectrum on the left-hand-side is the derived completion of $\bG({\rm X}, \rmH)$ along the map $\rho_{\rmG}$. (See section 3 for further details.) 
The above map is contravariantly functorial for flat $\rmG$-equivariant maps. 
\item Let $\ell$ denote a prime different from the characteristic of $k$.
 Then one also obtains a weak-equivalence
\vskip .2cm
$\bG({\rm X}, \rmH)_{\ell} \compl_{\rho_{\rmG}} {\overset {\simeq} \ra} \bG(E{\rmG}^{gm}{\underset {\rmH}  \times}X)_{\ell}$
\vskip .2cm \noindent
where the subscript $\ell$ denotes mod$-\ell$-variants of the appropriate spectra. (See ~\eqref{KGl} for their precise definitions.). Therefore, one also obtains the weak-equivalence:
\vskip .2cm \noindent
$\bG({\rm X}, \rmH) \compl_{\rho_{\ell} \circ \rho_{\rmG}} {\overset {\simeq} \ra} \bG(E{\rmG}^{gm}{\underset {\rmH}  \times}X) \compl_{\rho_{\ell}}$.
\vskip .2cm \noindent
The spectrum on the left-side (right-side) denotes the derived completion of $\bG({\rm X}, \rmH)$ ($\bG(\rmE{\rmG}^{gm}{\underset {\rmH}  \times}\rmX)$) with respect to the composite map $\rho_{\ell} \circ \rho_{\rmG}$ 
(the map $\rho_{\ell}$, \res). The above map is also contravariantly functorial for flat $\rmG$-equivariant maps. 
\item
If the split reductive group $\rmG$ is replaced by a split torus, and $\rmH$ denotes the same torus, then all of the above results extend to the case where $\rmX$ is a separated algebraic space 
of finite type over the base field $k$.
\end{enumerate} 
\end{theorem}
In fact, if one restricts to actions of split tori, one may also consider more general base schemes than a field: but we choose not to discuss
this extension in detail, mainly for keeping the discussion simpler.
The strategy we adopt to proving the above theorem  has several similarities as well as some key differences  with  the 
proof by Atiyah and Segal (see \cite{AS69}) of their theorem. Both proofs 
proceed by reducing first 
to the case where the split reductive group is a finite product of $\GL_n$s and then to the case it is a split torus. In this case, an equivariant form of the Kunneth-formula  
as in \cite{AJ} and \cite{JK} together with some properties of the derived completion provides a proof. 
The key difference between our proof and the proof of the classical Atiyah-Segal theorem is in the use of the derived completion, which
is essential for our proof. Moreover, unlike in the classical case, for a closed sub-group-scheme $\rmH$ in $\rmG$, the derived completion of a $\bK(\rmS, \rmH)$-module 
spectrum with
respect to the maps $\rho_H:\bK(\rmS, \rmH) \ra \H(\Z)$ and $\rho_{\rmG}: \bK(\rmS, \rmG) \ra \H(\Z)$ will be different in general. (This explains
the use of the derived completion with respect to $\rho_{\rmG}$ (and not $\rho_H$) in the above theorem.) The main exception to this is
discussed in the following theorem. The reduction to the case where the split reductive group is a finite product of $\GL_n$s is handled by the
following theorem, while the case when the group is a product of $\GL_n$s is handled by Theorem ~\ref{main.thm.3} and the case when
it is a split torus is handled by Theorem ~\ref{key.thm.1}.
\vskip .2cm
\begin{theorem} 
 \label{comparison}
Let $\rmH \subseteq \rmG$ denote a closed algebraic subgroup of the linear algebraic group $\rmG$ so that the restriction map 
$\rmR(\rmG) \ra \rmR(\rmH)$ of 
representation rings is surjective. Let $\rmX$ denote a scheme  of finite type over $k$, provided with an action by $\rmG$ and satisfying the
hypothesis ~\ref{stand.hyp.1}.
Then, for any module spectrum $M$  over 
$\bK({\rm X}, \rmH)$, which is $s$-connected for some integer $s$, the derived completions of $M$ along the augmentations $\rho_{\rmH}:\bK(Spec \, k, \rmH) \ra \H(\Z)$ 
and $\rho_{\rmG}: \bK(Spec \, k, \rmG) \ra \H(\Z)$ are weakly-equivalent.
\end{theorem}
 Here are several examples where the last theorem applies. (i) A good example for the above situation is when $\rmH$ is a diagonalizable subgroup of a split torus $\rmG$. For example, the above conclusions hold if $\rmH$ is a finite abelian group imbedded as a closed subgroup of a split torus. (ii) Another is 
 when $\rmG$ denotes a Borel subgroup of a split reductive group,  and $\rmH$ is a diagonalizable subgroup of the maximal torus contained in $\rmG$. 
  (iii) Next assume that $\rmH$ is
a finite group. Since the representation ring of $\rmH$ (of representations over $k$) is Noetherian, one may imbed $\rmH$ into a finite
 product of general linear groups, ${\rm {GL}}_{n_1}, \cdots, {\rm {GL}}_{n_m}$ so that the induced map $\rmR({\rm GL}_{n_1} \times \cdots \times {\rm GL}_{n_m}) \ra \rmR(\rmH)$
is surjective. Therefore, the above theorem applies in this case as well, where $\rmG= {\rm {GL}}_{n_1} \times \cdots \times {\rm {GL}}_{n_m}$. One may observe 
as a direct corollary to the above remarks and the last statement in Theorem ~\ref{main.thm.1}, that our results extend to actions of all 
smooth diagonalizable group schemes on schemes and algebraic spaces of finite type over $k$.
\vskip .2cm
We also provide the following theorem, which shows when the derived completions reduce to the usual completions at the augmentation ideal.
\begin{theorem}
\label{comp.2}
(i) Let $\rmG$ denote a connected split reductive group satisfying our standing hypothesis ~\ref{stand.hyp.1} acting on  a projective smooth
scheme $\rmX$ over a field $k$. Then one obtains the isomorphism: $\pi_*(\bK({\rm X}, \rmG) \compl_{\rho_{\rmG}}) \cong \pi_*(\bK({\rm X}, \rmG))\compl_{I_{\rmG}}$, where
the term on the right denotes the completion of $\pi_*((\bK(\rm{\rm X}, \rmG))$ at 
the augmentation ideal $I_{\rmG} \subseteq R(\rmG)$. 
\vskip .2cm
(ii) If $\rmG$ is any connected split reductive group  acting on
 the projective smooth scheme $\rmX$ over the field $k$
and ${\rm GL}_n$ denotes a general linear group containing $\rmG$ as a closed subgroup-scheme, then
$\pi_*(\bK(\rmX, \rmG) \compl_{\rho_{{\rm GL}_n}}) \cong \pi_*(\bK(\rmX, \rmG))\compl_{I_{\rmG}}$.
\end{theorem}
\vskip .2cm
This paper originated in our efforts to provide a proof of the descent conjecture for the K-theory of fields due to the
first author: see \cite{C11}.  It is also motivated by a similar derived completion theorem proven for actions of profinite groups,
making use of another model for the classifying spaces: see \cite{C13}. In that framework, one needs to consider 
 representation rings of profinite groups, which are in general, non-Noetherian. In the present framework, since the representation
rings of linear algebraic groups are Noetherian,  the need for derived completion is only so as not to
put any strong restrictions on the schemes whose equivariant K-theory and G-theory we consider.
\vskip .2cm
Here is an outline of the paper. We review the basic properties of Equivariant K-theory, G-theory and
the geometric classifying spaces of linear algebraic groups in section 2. Section 3 is devoted to establishing
several basic results on derived completion that we use in later sections of the paper. This supplements
the results of \cite{C08} on derived completion. Sections 4 and 5 are devoted to a detailed proof of Theorems
~\ref{main.thm.1} and ~\ref{comparison}, with section 4 discussing the reduction to the case where the group is a split torus.
In this section, we also re-interpret Theorem ~\ref{main.thm.1} in terms of
pro-spectra. Section 5 discusses
 the proof of Theorem ~\ref{main.thm.1} for the action of a split torus.  We point out that several results in this section 
hold more generally for actions on algebraic spaces, though still over a base field.  Section 6 discusses several examples and compares our results with earlier results in the literature. We also provide a proof of
Theorem ~\ref{comp.2}. It also discusses various applications such as computing the completed equivariant G-theory in terms of
equivariant motivic cohomology as well as equivariant Riemann-Roch theorems.
These should convincingly show the power and utility of our techniques.
The appendices  A though C
discuss various results of a technical nature: Appendix A discusses the passage between Eilenberg-Maclane spectra and chain complexes, Appendix B discusses the role of motivic slices in establishing  key
properties of Borel-style equivariant K-theories and Appendix C 
contains supplementary 
results on derived completions in terms of ideals (in ring spectra). 
\section{\bf Equivariant K-theory and G-theory: basic terminology and properties}
Throughout the paper we will let $\rmG$ denote  a split reductive group satisfying the standing hypothesis in ~\ref{stand.hyp.1}, quite often 
this being a $\GL_n$ or a finite product of $\GL_n$s.
Let $\rmX$ denote a separated algebraic space or scheme over $\rmS$ provided with the action of an affine group scheme $\rmG$ as above. Let 
$\Pscoh (\rm{\rm X}, \rmG)$ ($\Perf (\rm{\rm X}, \rmG)$)
denote the category of pseudo-coherent complexes of $\rmG$-equivariant $\O_X$-modules with bounded coherent cohomology sheaves (the category
of perfect complexes of $\rmG$-equivariant $\O_X$-modules, \res). 
Recall that a $\rmG$-equivariant complex of $\O_X$-modules is pseudo-coherent (perfect) if it is quasi-isomorphic locally on
the appropriate topology on $\rmX$ (i.e. if $\rmX$ is a scheme, we use the Zariski topology, and if $\rmX$ is an algebraic space, and not a scheme, we use the \'etale topology) to a bounded above complex (a bounded complex, \res) of locally free $\O_{\rmX}$-modules with bounded coherent
cohomology sheaves.
We provide these categories with the structure of bi-Waldhausen categories
with cofibrations, fibrations and weak-equivalences by letting the cofibrations be the maps of complexes that are degree-wise split monomorphisms
(fibrations be the maps of complexes that are degree-wise split epimorphisms, weak-equivalences be the maps that are quasi-isomorphisms, \res).
$\bG(\rm{\rm X}, \rmG)$ ($\bK(\rm{\rm X}, \rmG)$) will denote the K-theory spectrum obtained from $\Pscoh (\rm{\rm X}, \rmG)$ ($\Perf (\rm{\rm X}, \rmG)$, \res). 
\vskip .2cm
One may also consider the category ${\rm Vect}(\rm{\rm X}, \rmG)$ of $\rmG$-equivariant vector bundles on $\rmX$. This is an exact category and one may apply Quillen's construction to it to
produce another variant of the equivariant K-theory spectrum of $\rmX$. If we assume that every $\rmG$-equivariant coherent sheaf on $\rmX$ is the $\rmG$-equivariant quotient of a 
$\rmG$-equivariant vector bundle on $\rmX$, then one may observe that this produces a spectrum weakly-equivalent to $\bK(\rm{\rm X}, \rmG)$: see 
\cite[2.3.1 Proposition]{ThTr} or 
\cite[Proposition 2.8]{J10}. It follows from \cite[Theorem 5.7 and Corollary 5.8]{Th83} that this holds in many well-known examples. It is shown in 
\cite[section 2]{J02} that, in general, the map
from the K-theory spectrum to the G-theory spectrum (sending a perfect complex to itself, but viewed as a pseudo-coherent complex) is a weak-equivalence
\be \begin{equation}
     \label{PD}
\bK(\rm{\rm X}, \rmG) \simeq \bG(\rm{\rm X}, \rmG)
    \end{equation} \ee
\vskip .2cm \noindent
provided $\rmX$ is regular. In general, such a result fails to be true for the Quillen K-theory of $\rmG$-equivariant vector bundles, which is the reason for
our preference to the  Waldhausen style K-theory and G-theory considered above. 
\vskip .2cm
We will nevertheless make the assumption that every $\rmG$-equivariant coherent sheaf on the base scheme $\rmS$ is the quotient of a 
$\rmG$-equivariant vector bundle. (Recall that, for some parts of the paper, we do not require the base scheme to be the spectrum of a field and
that, nevertheless, we require that it be always a Noetherian regular scheme.)
 This hypothesis is clearly satisfied if the base scheme is the spectrum of a field 
and more generally if the base-scheme is  the spectrum of a Noetherian regular ring and the algebraic group $\rmG$ is  split reductive: see \cite[Corollary 5.2]{Th83}.
It now follows from the hypotheses that $\bK(\rmS, \rmG) $ is weakly-equivalent to the Quillen K-theory of
$\rmG$-equivariant vector bundles on $\rmS$ and  that one obtains the weak-equivalence $\bG(\rmS, \rmG) \simeq \bK(\rmS, \rmG)$.
\vskip .2cm
If $\ell$ is a prime different from the residue characteristics, $\rho_{\ell}: \Sigma \ra \H(Z/\ell)$ will denote the obvious map from the sphere spectrum to the mod$-\ell$ Eilenberg-Maclane spectrum.
\vskip .2cm
\subsection{The geometric classifying space}
\label{geom.class.space}
 We begin by recalling
briefly the construction of the {\it geometric classifying space of a linear algebraic group}: see for example, \cite[section 1]{tot}, \cite[section 4]{MV}. Let 
$\rmG$ denote a linear algebraic group over $\rmS =Spec \quad {\it k}$, i.e. a closed subgroup-scheme in $\GL_n$ over $\rmS$ for some n. For a  (closed) embedding 
$i : \rmG \ra \GL_n$ {\it the geometric classifying space} $\rmB_{gm}(\rmG; i)$ of $\rmG$ with respect to $i$ is defined as follows. For $m \ge   1$, let 
$\rmE\rmG^{gm,m}=U_m(G)=U({\mathbb A}^{nm})$ be the open sub-scheme of ${\mathbb A}^{nm}$ where the diagonal action of 
$\rmG$ determined by $i$ is free. By choosing $m$ large enough, one can always ensure that 
$\rmU({\mathbb A}^{nm})$ is non-empty and the quotient $\rmU({\mathbb A}^{nm})/G$ is a quasi-projective scheme:
see the discussion following Definition ~\ref{adm.gadg},  where a more detailed discussion  of the 
geometric classifying spaces appears.
\vskip .2cm
Let $\rmB\rmG^{gm,m}=\rmV_m(\rmG)=U_m(\rmG)/\rmG$ denote the quotient 
$\rmS$-scheme (which will be a quasi-projective variety ) for the 
action of $\rmG$ on $\rmU_m(\rmG)$ induced by this (diagonal) action of $\rmG$ on ${\mathbb A}^{nm}$; the projection $ \rmU_m(\rmG) \ra \rmV_m(\rmG)$ defines $\rmV_m(\rmG)$ as the 
quotient  of $\rmU_m(\rmG)$ by the free action of $\rmG$ and $\rmV_m(\rmG)$ is thus smooth. We have closed embeddings 
$\rmU_m(\rmG) \ra \rmU_{m+1}(\rmG)$ and $\rmV_m(\rmG) \ra \rmV_{m+1}(\rmG)$ corresponding to the embeddings 
$Id \times  \{\rm0\} : {\mathbb A}^{nm} \ra {\mathbb A}^{nm } \times {\mathbb A}^n$. We set $\rmE\rmG^{gm} = \{U_m(\rmG)|m\} = \{ \rmE\rmG^{gm,m}|m\}$ and 
$ \rmB\rmG^{gm} = \{ \rmV_m(\rmG)|m\}$ which are ind-objects in the category of schemes. (If one prefers, one may view each $\rmE\rmG^{gm,m}$ ($\rmB\rmG^{gm,m}$)
as a sheaf on the big Nisnevich (\'etale) site of smooth schemes over $k$ and then view $\rmE\rmG^{gm}$ ($\rmB\rmG^{gm}$) as the 
 the corresponding colimit taken in the category of sheaves on $({\rm {Sm/k}})_{Nis}$ or on $({\rm {Sm/k}})_{et}$.)
\vskip .2cm
Given a scheme $\rmX$  of finite type over $\rmS$ with a $\rmG$-action satisfying the standing hypotheses ~\ref{stand.hyp.1} , we let $\rmU_m(\rmG){\underset {\rmG} \times} \rmX$ denote the {\it balanced product}, 
where $(u, x)$ and $(ug^{-1}, gx)$ are identified for all $(u, x) \eps \rmU_m \times \rmX$ and $g \eps \rmG$.  Since the $\rmG$-action on $\rmU_m(\rmG)$ is free, $\rmU_m(\rmG){\underset {\rmG} \times} \rmX$
exists as a geometric quotient which is also a quasi-projective scheme in this setting, in case $\rmX$ is assumed to be quasi-projective: see \cite[Proposition 7.1]{MFK}. (In case $\rmX$ is an algebraic space of
finite over $\rmS$, the above quotient also exists, but as an algebraic space of finite type over $\rmS$.)
\vskip .2cm
It needs to be pointed out that the construction of the geometric classifying space is not unique in general. 
When $\rmX$ is an algebraic space or a scheme in general, we will let $\{\rmU_m(\rmG){\underset {\rmG} \times}\rmX|m\}$ denote the
ind-object constructed as above for a chosen ind-scheme $\{\rmU_m(\rmG)|m \ge 0\}$. When $\rmX$ is restricted to the category of
smooth schemes over $\Spec k$, one can apply the result below in ~\eqref{indep.geom.class.sp} to show that the
choice of the ind-scheme $\{\rmU_m(\rmG)|m \ge 0\}$  is irrelevant.
\begin{definition} (Borel style equivariant K-theory and G-theory)
\label{Borel.equiv.th}
Assume first that $\rmG$ is a linear algebraic group or a finite group viewed as an algebraic group by imbedding it in some $\GL_n$. 
We define the Borel style equivariant K-theory of $\rmX$ to be $\bK(\rmE\rmG^{gm}{\underset {\rmG} \times}\rmX) = \holimm \bK(\rmU_m(\rmG){\underset {\rmG} \times}\rmX)$.
$\bG(\rmE\rmG^{gm}{\underset {\rmG} \times}\rmX)$ is defined similarly by observing that the closed imbeddings
$\rmU_m(\rmG){\underset {\rmG} \times}\rmX \ra \rmU_{m+1}(\rmG){\underset {\rmG} \times}\rmX$ are all regular closed imbeddings, and that, therefore the corresponding
pull-backs are defined at the level of G-theory. If $\ell$ is a prime different from the characteristic of the field $k$,
then $\bK(\rmE\rmG^{gm}{\underset {\rmG} \times}\rmX) \compl_{\rho_{\ell}} =  \holimm \bK(\rmU_m(G){\underset {\rmG} \times}\rmX)\compl_{\rho_{\ell}}$.
$\bG(\rmE\rmG^{gm}{\underset {\rmG} \times}\rmX) \compl_{\rho_{\ell}}$ is defined similarly.
\end{definition}
\vskip .2cm
Since the $\rmG$-action on $\rmE\rmG^{gm}$ is free, one obtains the weak-equivalences:
\vskip .2cm
$\bK(\rmE\rmG^{gm} \times \rm{\rm X}, \rmG) \simeq \bK(\rmE\rmG^{gm}{\underset {\rmG} \times}\rmX)$, 
$\bG(\rmE\rmG^{gm} \times \rm{\rm X}, \rmG) \simeq \bG(\rmE\rmG^{gm}{\underset {\rmG} \times}\rmX)$ 
\vskip .2cm \noindent
and similarly for the $\rho_{\ell}$-completed version.
\begin{example}
 Assume $\rmG$ is a subgroup of $\Sigma_n$, which is the symmetric group on $n$-letters. Then one may choose 
$\rmU_m(\Sigma_n) = \{(x_i, \cdots, x_m) \eps ({\mathbb A}^m)^n| x_i \ne x_j, i \ne j\}$, since the action on this $\rmU_m(\Sigma_n)$ by the symmetric group
 $\Sigma _n$ is free.
\end{example}
Next we make the following observations.
\begin{itemize}
 \item Let $\{\rmE\rmG^{gm,m}|m \}$ denote an ind-scheme defined above associated to the algebraic group $\rmG$.
Then if $\rmX$ is any scheme or algebraic space over $k$, then viewing everything as simplicial presheaves on
the Nisnevich site, we obtain $\colimm \rmE\rmG^{gm,m}{\underset {\rmG} \times }X \cong
(\colimm \rmE\rmG^{gm,m}){\underset {\rmG} \times}X = \rmE\rmG^{gm}{\underset {\rmG} \times}X$. (This follows readily from the
observation that the $\rmG$ action on  $\rmE\rmG^{gm,m}$ is free and that filtered colimits commute with the
balanced product construction above.)
\item It follows therefore, that if $\rmE$ is any ${\mathbb A}^1$-local spectrum and $\rmX$ is a smooth scheme of finite type over $k$, then one obtains a weak-equivalence:
\[ \rmMap(\rmE\rmG^{gm}{\underset {\rmG} \times}{\rm X}, E) = \holimm \{Map(\rmE\rmG^{gm,m}{\underset {\rmG} \times}{\rm X}, E)|m\}.\]
\end{itemize}
where $\rmMap(\quad, E)$ denotes the simplicial mapping spectrum. 
It is shown in Appendix B, making use of the properties of motivic slices that for any two different models of geometric classifying spaces
given by $\{\rmE\rmG^{gm,m}|m \}$ and $\{\widetilde {\rmE\rmG}^{gm,m}|m\}$, one obtains a weak-equivalence for any ${\mathbb A}^1$-local spectrum $E$ and any smooth scheme
$\rmX$:
\be \begin{equation}
\label{indep.geom.class.sp}
 \holimm \{\rmMap({\widetilde {\rmE\rmG}}^{gm,m}{\underset {\rmG} \times}{\rm X}, E)|m\} =  \holimm \{\rmMap ({\rmE\rmG}^{gm, m}{\underset {\rmG} \times}{\rm X}, E)|m\}.
\end{equation} \ee
\begin{remark}
\label{mainthm.comments}
 It may be worth pointing out that the only need to restrict to schemes satisfying the Standing Hypotheses as in ~\ref{stand.hyp.1} is
to be able to reduce G-theory to K-theory for regular schemes, where one may make use of the above arguments to show the Borel style
equivariant K-theory (and hence Borel-style equivariant G-theory) is independent of the choice of the geometric classifying spaces. 
This is essential, since in the proof of Theorem ~\ref{main.thm.1} one needs to use two different models of geometric classifying spaces 
for the action of a maximal torus $\rmT$ in a split reductive group $\rmG$, one of these being $\{\rmE\rmG^{gm,m}/\rmT|m \}$ and the other being $\{\rmE\rmT^{gm,m}/\rmT|m\}$. 
 It is only in this step that we need to invoke motivic arguments. If one restricts to actions of split tori, then one may use a single model of the geometric classifying space, namely $\{\rmE\rmT^{gm,m}/\rmT|m\}$, and 
therefore, Theorem ~\ref{main.thm.1} holds for actions of split tori 
on all algebraic spaces of finite type over the base field.
\end{remark}

\subsection{Key properties of equivariant K- and G-theories}
\label{KG.props}
\vskip .2cm
All of the properties below extend to the $\rho_{\ell}$-completed versions though we do not state them explicitly. Throughout $\rmX$ will denote
an algebraic space of finite type over $\rmS$, except in (ii) and (iv) where it is required to be a scheme satisfying the hypotheses in ~\ref{stand.hyp.1}.
\begin{enumerate}[\rm(i)]
 \item One may also consider the Quillen style K-theory of the abelian category of $\rmG$-equivariant coherent sheaves on $\rmX$. This will always produce
a spectrum weakly-equivalent to $\bG(\rm{\rm X}, \rmG)$: see, for example, \cite[Proposition 2.7]{J10}.
\item Let $\rmH$ denote a closed subgroup scheme of $\rmG$. Then one obtains the weak-equivalences: $\bG(\rm{\rm X}, \rmH) \simeq \bG(\rmG{\underset {\rmH} \times}\rm{\rm X}, \rmG)$ and
$\bK({\rm X}, \rmH) \simeq \bK(\rmG{\underset {\rmH} \times}{\rm X}, \rmG)$. Here $\rmG{\underset {\rmH} \times }{\rmX}$ denotes the quotient of $\rmG \times {\rmX}$ by the action of $\rmH$ given by 
$h(g, x) = (gh^{-1}, hx)$. Under the same hypotheses, one also obtains the weak-equivalences of inverse systems (i.e. a weak-equivalence on taking their homotopy inverse limits): 
$\{\bK(\rmE\rmH^{gm, m}{\underset \rmH \times}X)|m\} \simeq \{\bK(\rmE\rmG^{gm, m}{\underset {\rmG} \times}G{\underset {\rmH} \times}X)|m\}$
and $\{\bG(\rmE\rmH^{gm, m}{\underset {\rmH}  \times}X)|m\} \simeq \{\bG(\rmE\rmG^{gm, m}{\underset {\rmG} \times}G{\underset {\rmH}  \times}X)|m\}$. (These 
follow readily in view of the results in Appendix B, and the property (viii) below, which shows how to reduce to the case $\rmX$ is a regular scheme.)
\item
The $\rmG$-equivariant flat  map $\pi:\rmG{\underset {\rmH}  \times}\rmX \ra {\rmX}$, $(gh^{-1}, hx) \mapsto gh^{-1}hx = gx$
induces  maps $\pi^*: \bG(\rm{\rm X}, \rmG) \ra \bG(\rmG{\underset {\rmH}  \times}\rm{\rm X}, G) \simeq \bG(\rm{\rm X}, \rmH)$, $\pi^*: \bG(\rmE\rmG^{gm, m}{\underset {\rmG} \times}\rmX) \ra \bG(\rmE\rmG^{gm, m }{\underset {\rmG} \times} (\rmG {\underset {\rmH}  \times}\rmX))$
 which identify with the corresponding maps obtained by restricting the group action from $\rmG$
to $\rmH$. 
\item
In view of the above properties, given any linear algebraic group $\rmG$, we will fix a closed imbedding $\rmG \ra GL_n$ for some $n$ and
identify $\bG(\rm{\rm X}, \rmG)$ with $\bG(\GL_n{\underset {\rmG} \times}\rm{\rm X}, \GL_n)$, $\bK(\rm{\rm X}, \rmG)$ with $\bK(\GL_n{\underset {\rmG} \times}\rm{\rm X}, \GL_n)$,
\[\bG(\rmE\rmG^{gm, m}{\underset {\rmG} \times}\rmX) \mbox{ with } \bG(\rmE{\GL_n}^{gm, m }{\underset {\GL_n} \times}(\GL_n{\underset {\rmG} \times}\rmX)) \mbox{ and } \]
\[\bK(\rmE\rmG^{gm, m}{\underset {\rmG} \times}\rmX) \mbox{ with } \bK(\rmE{\GL_n}^{gm, m}{\underset {\GL_n} \times}(\GL_n{\underset {\rmG} \times}\rmX)).\]
\item
Next assume that ${\rmG}$ is a split reductive group over $\rmS$ and $\rmH=\rmB$ i.e. $\rmH$ is a Borel subgroup  
of $\rmG$. Then using 
the observation that $\rmG/\rmB$ is
proper over ${\rm S}$ and $R^n \pi_* =0$ for $n$ large enough, one sees that the
  map $\pi$ also induces  push-forwards $\pi_*: \bG(\rmG{\underset {\rmB} \times}\rm{\rm X}, \rmG) \ra \bG(\rm{\rm X}, \rmG)$ and
$\pi_*:  \bG(\rmE\rmG^{gm, m}{\underset {\rmG} \times} G{\underset {\rmB} \times}\rmX) \ra \bG(\rmE\rmG^{gm, m}{\underset {\rmG} \times}\rmX)$ induced by the derived direct image
functor $R\pi_*$. (Such a derived direct image functor may be made functorial at the level of complexes by considering pseudo-coherent complexes which
are injective $\O_{\rmX}$-modules in each degree.)
\item
Assume the above situation. Then the projection formula applied to $R\pi_*$ shows that the composition $R\pi_*\pi^*(F) = F \otimes R\pi_*(\O_{\rmG{\underset {\rmB} \times}\rmX}) \cong 
F $, since 
\be \begin{align}
     \label{derived.direct}
R^n \pi_*(\O_{\rmG{\underset {\rmB} \times}X}) &= \O_{\rmX}, \mbox{ if n=0 and}\\
&=0, \mbox{ if $n>0$}. \notag
    \end{align} \ee
\vskip .2cm \noindent
It follows $\pi^*$ is a split monomorphism in this case, with the splitting provided by $\pi_*$.
\item
Assume the situation as above with $\rmT=$ the maximal torus  in ${\rmG}$. Then $\rmB/\rmT = R_u(\rmB)$= an affine space. Now
one obtains the weak-equivalences: 
\[\bG(\rm{\rm X}, \rmB) \simeq \bG(\rmB{\underset {\rmT} \times}\rm{\rm X}, \rmB) \simeq \bG(\rm{\rm X}, \rmT) \mbox{  and } \bG(\rmE\rmB^{gm, m }{\underset {\rmB} \times}\rmX) \simeq \bG(E\rmB^{gm, m}{\underset {\rmB} \times}\rmB{\underset {\rmT} \times}\rmX) \simeq \bG(\rmE\rmB^{gm, m}{\underset {\rmT} \times}\rmX)\]
 where the first weak-equivalences
are from the homotopy property of $G$-theory and the observation that $R_u(\rmB)$ is an affine space over $\rmS$.
\item Finally assume that $\rmX$ is a  scheme satisfying the standing hypotheses as in ~\ref{stand.hyp.1}. 
Let $i: \rmX \ra \tilde \rmX$ denote a $\rmG$-equivariant closed immersion into a regular $\rmG$-scheme. Then one obtains the weak-equivalence: 
$\bG({\rm X}, \rmG) \simeq 
(hofib (i^*: \bK(\tilde {\rm X}, \rmG) \ra \bK(\tilde \rmX - {\rm X}, \rmG)))$ where ${\rm hofib}$ denotes the canonical homotopy fiber. Similarly, for each non-negative integer $m$,
 $\bG(\rmE\rmG^{gm, m}{\underset {\rmG} \times} \rmX)$ 
is weakly-equivalent to the canonical homotopy fiber of the restriction 
$i^*: \bK(\rmE\rmG^{gm, m}{\underset {\rmG} \times} \tilde \rmX) \ra \bK(\rmE\rmG^{gm, m}{\underset {\rmG} \times} (\tilde \rmX -\rmX))$. These identifications
enable us to readily extend the main results of this paper to schemes that are not regular, but satisfy the standing hypotheses ~\ref{stand.hyp.1}.
\end{enumerate}
\subsection{Properties of the representation ring}
\label{rep.ring}
Assume throughout this subsection that the base scheme is the spectrum of a field $k$.
First recall that the algebraic fundamental group associated to a split reductive group $\rmG$ over $k$ may be defined as $\Lambda/\rmX( \rmT)$, where $\Lambda$
($\rmX( \rmT)$) denotes the weight lattice (the lattice of characters of the  maximal torus in $\rmG$, \res): see for example, \cite[1.1]{Merk}. Then it is
observed in \cite[Proposition 1.22]{Merk}, making use of \cite[Theorem 1.3]{St}, that if this fundamental group  is torsion-free, then
$\rmR(\rmT)$ is a  free module over $\rmR(\rmG)$. Here $\rmT$ denotes a maximal torus in $\rmG$ and $\rmR(\rmG)$ ($\rmR(\rmT)$) 
denotes the representation ring of $\rmG$ ($\rmT$, \res).
Making use of the observation that ${\rm SL}_n$ is simply-connected (i.e. the above fundamental group is trivial), for any $n$, one may conclude that 
$\pi_1(\GL_n) \cong \pi_1({\mathbb G}_m) \cong {\mathbb Z}$ where ${\mathbb G}_m$ denotes the central torus in $\GL_n$. Therefore $\rmR(\rmT)$ is free
over $\rmR(\GL_n)$, where $\rmT$ denotes a maximal torus in $\GL_n$.
\subsection{\bf Basic conventions and terminology}
\label{basic.terminology}
\begin{itemize}
 \item{First we clarify that all spectra used in
this paper are  symmetric $S^1$-spectra which, when restricted to the category of smooth schemes over ${\rm S}$ 
 are presheaves on the big Nisnevich site of smooth schemes (or a suitable subcategory) over the given base $\rmS$ and are ${\mathbb A}^1$-homotopy invariant
when restricted to smooth schemes over $\rmS$. This category will be denoted
$\Spt_{S^1}(\rmS)$. An example to keep in mind is the equivariant K-theory spectrum (with respect to the action of a fixed linear algebraic group $\rmG$) 
which is defined only on schemes or algebraic spaces provided with actions by $\rmG$).  It is important to observe  that
some such spectra admit extensions to schemes that are not necessarily smooth over $\rmS$, as well as to algebraic spaces, the main examples of which are the equivariant $\rmG$-theory and $\rmK$-theory spectra. Ring and module spectra 
will have the usual familiar meaning but viewed as objects in $\Spt_{S^1}(\rmS )$. An appropriate context for much our work would be that of model categories for ring and module spectra as worked
out in \cite[Theorem 4.1]{SS} as well as \cite{Ship04}.}
\vskip .2cm
\item{We will make extensive use of the model structures defined in \cite{Ship04} to produce cofibrant replacements
for commutative algebra spectra over a given commutative ring spectrum. The commutative ring spectra that
show up in the paper are largely the K-theory spectra  and the 
equivariant K-theory spectra associated to  the actions of a 
linear algebraic group.}
\vskip .2cm
\item{In addition, we will also need to consider the usual Eilenberg-Maclane spectra associated to
commutative rings. Let $\R$ denote such a commutative ring and let $Alg(\R)$ denote the category of
commutative algebras over $\R$. In this case we will consider the free-commutative algebra functor to find a
simplicial resolution of a commutative $\R$-algebra $\S$, and use it define a commutative algebra spectrum
$\H(\tilde {\S})$ which will be a cofibrant replacement of $\H(\S)$ in the model category of module spectra over $\H(\R)$. See the proof of Theorem ~\ref{T.comp.vs.G.comp} and
~\ref{res.free.comm.alg} for more details on this technique.}
\vskip .2cm
\item{Let $\Spt_{S^1}(\rmS)$ denote the category of spectra and let $\rmI$ denote a small category. 
Then we provide the category of diagrams of spectra of type $I$, $\Spt_{S^1}(\rmS) ^{\rmI}$, with the projective model structure as 
in \cite[Chapter XI]{BK}. i.e. The fibrations (weak-equivalences) are maps $\{f^i:K^i \ra L^i|i \eps \rmI\}$
so that each $f^i$ is a fibration (weak-equivalence, \res) with the cofibrations defined by the lifting property with respect to
 trivial fibrations. Since the homotopy inverse limit functor
is not well-behaved unless one restricts to fibrant objects in this model structure, we will always
implicitly replace a given diagram of spectra functorially by a fibrant one before applying the homotopy inverse limit.
i.e. If $\{K^i|i \eps \rmI \}$ is a diagram of spectra, $\holim \{K^i|i \eps \rmI\}$ actually denotes $\holim \{R(K^i)|i \eps \rmI\}$
where $\{R(K^i)| i \eps \rmI\}$ is a fibrant replacement of $\{K^i| i \eps \rmI\}$ in $\Spt_{S^1}(\rmS) ^{\rmI}$.}
\vskip .2cm   
\end{itemize}
\section{\bf Derived completions}
\label{der.completion}
In this section, we begin by reviewing some of the basic results from \cite{C08} and then go onto establish several supplemental results
on the derived completion which will play a key role in later sections of the paper. 
The essential use of  derived smash products
in Proposition ~\ref{key.prop} and Theorem ~\ref{key.thm.1}, which play a key-role in our results, 
make the framework of the derived completions as worked out in \cite{C08} the appropriate framework, for this paper.
\vskip .2cm
Let $f: B \ra C$ denote a map of commutative ring spectra in $\Spt_{S^1}(\rmS )$ which are both $-1$-connected and let $\tilde C$ denote a cofibrant replacement of $C$ in the category of commutative algebra spectra over $B$.  Let $Mod(B)$ ($Mod(\tilde C)$)
denote the category of module spectra over $B$ ($\tilde C$, \res). Then one
defines a triple $Mod(B) \ra Mod(B)$ by sending a $B$-module spectrum $M$ to $M{\underset B \wedge}\tilde C$, which belongs to $Mod(\tilde C)$,
 and then by sending it to the corresponding $B$-module spectrum using restriction of scalars from $\tilde C$ to $B$. Iterating this defines
a cosimplicial object ${\mathcal T}_B^{\bullet}(M, \tilde C)$ in $Mod(B)$ whose $n$-th degree term is the iterated smash product:
\[{\overset {n+1} {\overbrace {M{\underset B \wedge} \tilde C \cdots {{\underset B \wedge } \tilde C}}}}.\]
We will often denote this object as $M{\overset {n+1} {\overbrace {{\overset L {\underset B \wedge}} C \cdots {{\overset L {\underset B \wedge }}  C}}}}$,
meaning that an implicit choice of a  cofibrant replacement of $C$ has been made. Similarly ${\mathcal T}^{\bullet}_B(M, \tilde C)$ will  often be denoted
${\mathcal T}_B^{\bullet}(M, C)$, meaning an implicit choice of a cofibrant replacement of $C$ has been made. 
\begin{definition}
\label{def.der.compl}
In this set-up, {\it the 
derived completion of $M$ 
along $f$} is defined to be the homotopy inverse limit: $\holimD {\mathcal T}_B^{\bullet}(M, C)$ and often denoted
$M \compl_{f}$. 
\end{definition}
\vskip .1cm \noindent
The derived completion is therefore a completion with respect to the above triple in the sense of \cite{Rad}.
\begin{remark}
 \label{alternate.form} Let $Alg(B)$ denote the category of commutative algebra spectra over $B$. Observe
that in the model structure on $Alg(B)$ defined in \cite{Ship04}, the weak-equivalences are maps in $Alg(B)$
which are stable equivalences of symmetric spectra and that any cofibrant object  in $Alg(B)$ is also cofibrant
in $Mod(B)$. Observe also that the weak-equivalences in $Mod(B)$ are
maps in $Mod(B)$ which are stable equivalences of symmetric spectra. Therefore first observe that if
$\tilde C$ is chosen as above, then it is also a cofibrant replacement of $C$ in $Mod(B)$.
Therefore, if $\hat C \eps Alg(B)$ is such that it is a cofibrant replacement of $C$ in $Mod(B)$, then for any
$M \eps Mod (B)$, 
one obtains a weak-equivalence between $M{\underset B \wedge}\tilde C$ and $M{\underset B \wedge}\hat C$ as well
as between the corresponding iterated smash products with $\tilde C$ and $\hat C $ over $B$. Therefore, one
may also define the cosimplicial object ${\mathcal T}^{\bullet}_B(M, \tilde C)$ using $\hat C$ in the place of 
$\tilde C$. This observation will be used in the proof of Theorem ~\ref{T.comp.vs.G.comp}.
\end{remark}

\vskip .2cm
\subsubsection{Partial derived completions} 
\label{def.part.der.compl}
The partial derived completion of $M$ along $f$ up to order $m$ will be defined
to be the $\holimDm \tau_{\le m}{\mathcal T}^{\bullet}_B(M, C)$, where $\tau_{\le m} {\mathcal T}^{\bullet}_B(M, C)
= \{{\mathcal T}^{i}_B(M, C)|i \le m\}$. This will be denoted $M \compl_{f, m}$.
\vskip .2cm
Let $I_B$ denote the homotopy fiber of $f$. This is a $B$-module spectrum and it is observed in \cite[Corollary 6.7]{C08} that
one obtains the stable cofiber sequence:
\be \begin{equation}
\label{key.cof.seq}
({\overset L {\underset B \wedge}} _{i=1}^{m+1} I_B ){\overset L {\underset B \wedge}} M \ra M \ra \holimDm \tau_{\le m}{\mathcal T}^{\bullet}_B(M, C) = M \compl_{f, m}
\end{equation} \ee
\vskip .2cm
Given a commutative ring $R$ (a module $M$ over $R$) $\H(R)$ ($\H(M)$, \res) will denote the 
Eilenberg-Maclane spectrum associated to $R$ ($M$, \res). (Recall that $\H(R)$ ($\H(M)$) has a single non-trivial stable homotopy group
in degree $0$, where it is $R$ ($M$, \res).) Then one observes readily that $\H(R)$ is a commutative
connected ring spectrum with $\H(M)$ a module-spectrum over $\H(R)$.
\begin{lemma} 
\label{flatness}
Suppose $R$ is a commutative Noetherian ring with $I$ an ideal in $R$ so that $I$ is flat as an $R$-module. Let $M$ be
 a flat $R$-module. Then one obtains the following weak-equivalences, where $f:\H(R) \ra \H(R/I)$ denotes the map induced by the map $R \ra R/I$:
\be \begin{align}
\label{part.compl.1}
({\overset L {\underset {\H(R)} \wedge}} _{i=1}^{m+1} \H(I)) {\overset L {\underset {\H(R)} \wedge}} \H(M) &\simeq
\H(({\overset L {\underset R \otimes}} _{i=1}^{m+1} I) {\overset L {\underset R \otimes}} M) \mbox{ and }\\
H(M) \compl_{f, m} &= \holimDm \tau_{\le m}{\mathcal T}^{\bullet}_{\H(R)}(\H(M), \H(R/I)) \simeq \H(\limm \{M/I^k| k \le m+1\}) \notag \end{align} \ee
\vskip .2cm \noindent
where one makes use of the free resolutions by the free commutative algebra functor discussed
in ~\ref{res.free.comm.alg} to define the truncated cosimplicial object $\tau_{\le m}{\mathcal T}^{\bullet}_{\H(R)}(\H(M), \H(R/I))$.
\end{lemma}
\begin{proof} In view of the flatness assumptions, the spectral sequence computing the derived smash products
of the left-hand-side in ~\eqref{part.compl.1} degenerates. This proves the first weak-equivalence: observe that
the flatness assumptions imply the left-derived functors of the tensor product there may be replaced by the
tensor product. Next observe that the iterated derived tensor product 
$M{\overset {m+1} {\overbrace {{\overset L {\underset R \otimes}} I \cdots {\overset L {\underset R \otimes }} I}}}$ 
identifies with
$M {\overset {m+1} {\overbrace {{\underset R \otimes} I \cdots {{\underset R \otimes } I}}}} = M \otimes I^{m+1}$
for similar reasons, and that therefore, this maps injectively to 
$ M {\overset {m+1} {\overbrace {{\underset R \otimes} R \cdots {{\underset R \otimes } R}}}} \cong M$. Then the resulting inverse system $\{M \otimes R/I^k =M/(I^kM)|k \}$ has surjective 
transition maps so that there are no $lim^1$-terms. This provides the second weak-equivalence.
\end{proof}
\begin{remark} A corresponding result for the completion $ M \mapsto M \compl_{f}$ is proven in \cite[Theorem 4.4]{C08}
which holds more generally without any of the flatness assumptions above. However, that proof does not apply to
partial derived completions: hence the need for the above result.
\end{remark}
\subsubsection{Basic spectral sequences and reduction to derived completions of modules over rings}
\label{der.compl.props}
If $M$ is an $s$-connected module spectrum over $B$, (for some integer $s$) there exist  convergent second quadrant spectral sequences:
\be \begin{align}
     \label{alg.geom.ss}
E_1^{p,q} &= \pi_{q+2p}(\H(\pi_{-p}(M) )\, \,  {\widehat {}}_{{\H(\pi_0(f))}}) \Ra \pi_{q+p}(M \, \, {\widehat {}}_{f}),\\
E_1^{p,q} &= \pi_{q+2p}(\H(\pi_{-p}(M) )\, \,  {\widehat {}}_{f,m}) \Ra \pi_{q+p}(M \, \, {\widehat {}}_{f,m}). \notag
    \end{align} \ee
\vskip .2cm \noindent
The first is discussed explicitly in \cite[Theorem 7.1]{C08} and the second is obtained in an entirely similar manner
using \cite[Proposition 7.7 and Corollary 7.8]{C08}. One may want to observe that one cannot, in general, identify the
$E_1^{p,q}$-terms for the partial derived completion with $\pi_{q+2p}(\H(\pi_{-p}(M)) \, \, {\widehat {}}_{\H(\pi_0(f)), m})$.
\vskip .2cm
However, since the convergence of the above spectral sequences may not be strong, what seems more useful are the arguments
in \cite[Corollary 7.8]{C08} that enable one to reduce questions on derived completions to derived completions of modules over rings.
Accordingly let $M \eps Mod(B)$ for some commutative ring spectrum $B$, let $f: B \ra C$ denote a map of commutative ring spectra and let
$\tilde C$ denote a cofibrant replacement of $C$ in the category of commutative algebra spectra over $B$. Let $\{\cdots \ra M<s> \ra M<s-1> \cdots|s \}$ denote the
canonical Postnikov tower for $M$. We will assume that $M$ is $s$-connected, for some integer $s$. Then the fiber of the fibration $M<s> \ra M<s-1>$ is an Eilenberg Maclane spectrum 
$K(\pi_s(M), s)$. We will denote this by $\H(\pi_s(M))[s]$. Observe that each $M<s>$ is a module spectrum over $\H(\pi_0(B), 0)$: using
the map $B \ra \H(\pi_0(B), 0)$ of ring spectra, one may view each $M<s>$ and the fiber of the map $M<s> \ra M<s-1>$ as module spectra over $B$.
It is shown in \cite[Corollary 7.8]{C08} that, then 
\begin{enumerate}
\item [(i)] $\{\cdots \ra (M<t>)\compl_{f} \ra (M<t-1>)\compl_{f} \cdots|t \} $
is also a tower of fibrations (up to homotopy),
\item [(ii)] that the homotopy fiber of the map  $(M<t>)\compl_{f} \ra (M<t-1>)\compl_{f}$ is 
$\H(\pi_t(M)) \compl_{\H(\pi_0(f))}[t]$, where $\H(\pi_t(M)) \compl_{\H(\pi_0(f))}$ denotes the derived completion of $\H(\pi_t(M))$ along 
the map $\H(\pi_0(f))$
\newline \noindent
 (with $\H(\pi_t(M)) \compl_{\H(\pi_0(f))}[t]$ denoting a suitable shift), and that
\item [(iii)] $M\compl_f \simeq \holimt (M<t> \compl_{f})$.
\end{enumerate}
Since $M$ is assumed to be $s$-connected, $M<s> = K(\pi_s(M), s)$, so that we may start an inductive argument at $s$ and step along the
terms in the tower in (i) and finally use (iii) to deduce properties of derived completion on $M$.
\subsubsection{Mod-$\ell$ completions}
\label{mod-l-completion}
Given a commutative ring $\rmR$ with a unit, let $\H(\rmR)$ denote the Eilenberg-Maclane spectrum. Let $\ell$ denote a fixed prime. Then the mod-$\ell$ completion for us will denote completion with respect to the 
mod-$\ell$ reduction map $\rho_{\ell}: \Sigma \ra \H(\Z/\ell)$, where $\Sigma$ denotes the usual ($S^1$)-sphere spectrum. {\it We will assume throughout
that $\H(\Z/\ell)$ is a cofibrant (commutative) algebra-spectrum over the sphere spectrum $\Sigma$.}
\subsubsection{Derived completions for equivariant K-theory}
\label{der.compl.eqK}
For the rest of this discussion we will assume that $\rmS$ is the spectrum of a field.
We denote the restriction  map 
\be \begin{equation}
\label{IG}
\bK(\rmS, \rmG) \ra \bK(\rmS) \mbox{ by } \rho_G \mbox{ and the Postnikov-truncation map } \bK(\rmS) \ra \H(\pi_0(\bK(\rmS))) \mbox{ by } <0>.
\end{equation} \ee
\vskip .2cm \noindent
The composite map $\bK(\rmS, \rmG) \ra \H(\pi_0(\bK(\rmS)))$ will be denoted $\tilde \rho_{\rmG}$.  Observe that $\pi_0(\bK(\rmS, \rmG)) = R(\rmG)$= the representation ring of $\rmG$.
Clearly the restriction $\rho_{\rmG}$ induces a surjection on taking $\pi_0$: this is clear, since if $V$ is any finite dimensional representation of
$\rmG$ of dimension $d$, the $d$-th exterior power of $V$ will be a $1$-dimensional representation of $\rmG$. The  Postnikov-truncation map $<0>:  \bK(\rmS) \ra \H(\pi_0(\bK(\rmS)))$  clearly induces an 
isomorphism on $\pi_0$. Therefore, the hypotheses of \cite[Theorem 6.1]{C08} are satisfied and it shows that the 
derived completion functors with respect to the two maps $\rho_{\rmG}$ and $\tilde \rho_{\rmG}$ identify up to weak-equivalence.  Therefore, while we will
always make  use of the completion with respect to the map $\rho_{\rmG}$, we will often find it convenient to identify this completion with the completion with respect to
$\tilde \rho_{\rmG}$.
\vskip .2cm
Observe that the above completion $\rho_{\rmG}$ has the following explicit description.
Let $Mod(\bK(\rmS, \rmG))$ ($Mod(\bK(\rmS ))$) denote the category of module-spectra over $\bK(\rmS, \rmG)$ ($\bK(\rmS )$, \res). Then sending a module spectrum
$M \eps Mod(\bK(\rmS, \rmG))$ to $M \wedgeK \bK(\rmS )$ and then viewing it as a $\bK(\rmS, \rmG)$-module spectrum using the ring map $\bK(\rmS, \rmG) \ra K(\rmS )$ defines
a triple. The corresponding  cosimplicial object of spectra is given by
\vskip .2cm
${\mathcal T}^{\bullet}_{\bK(\rmS, \rmG)}(M, \bK(\rmS )): M \cosimp1 M \wedgeK \bK(\rmS )  \cosimp1 \cdots M \wedgeK \bK(\rmS ) \wedgeK \cdots \wedgeK \bK(\rmS ) \cdots $
\vskip .2cm \noindent
with the obvious structure maps. Then the derived completion of $M$ with respect to $\rho_{\rmG}$  is  
\be \begin{equation}
     \label{der.compl.1}
M\, \, \compl_{\rho_{\rmG}} = \holimD {\mathcal T}^{\bullet}_{\bK(\rmS, \rmG)}(M, \bK(\rmS ))
\end{equation} \ee
and the corresponding partial derived completion to degree $m$ will be denoted $\rho_{\rmG, m}$.
\vskip .2cm \noindent
We may also consider the following variant.
For each prime $\ell$, we let
\be \begin{equation}
\label{IGl}
\rho_{\ell} \circ \rho_{\rmG}:\bK(\rmS, \rmG) \ra \bK(\rmS )/\ell= \bK(\rmS ) {\underset {\Sigma} \wedge} \H({\mathbb Z}/\ell) 
\end{equation} \ee
\vskip .2cm \noindent
denote the composition of $\rho_{\rmG}$ and the mod-$\ell$ reduction map $id_{\bK(\rmS )} \wedge \rho_{\ell}$. 
The derived completion with respect to $\rho_{\ell} \circ \rho_{\rmG}$ clearly has a description similar to the one in ~\eqref{der.compl.1}.
\vskip .2cm
If $\rmX$ is a scheme or algebraic space provided with an action by the group $\rmG$, we let 
\be \begin{equation}
\begin{split}
     \label{KGl}
\bK({\rm X}, \rmG)_{\ell} = {\overset {n} {\overbrace {\bK({\rm X}, \rmG){\underset {\Sigma} \wedge} \H(\Z/\ell) \cdots 
{{\underset {\Sigma} \wedge } \H(\Z/\ell)}}}} \\
\bK(\rmE\rmG^{gm}{\underset \rmG \times}X)_{\ell} = {\overset {n} {\overbrace {\bK(\rmE\rmG^{gm}{\underset \rmG \times}X){\underset {\Sigma} \wedge} \H(\Z/\ell) \cdots 
{{\underset {\Sigma} \wedge } \H(\Z/\ell)}}}}
\end{split}
    \end{equation} \ee
\vskip .2cm \noindent 
for some positive integer $n$. If it becomes important to indicate the $n$ above, we will denote $\bK({\rm X}, \rmG)_{\ell} $ 
($\bK(\rmE\rmG^{gm}{\underset \rmG \times}X)_{\ell}$) by $\bK({\rm X}, \rmG)_{\ell,n}$ ($\bK(\rmE\rmG^{gm}{\underset \rmG \times}X)_{\ell,n}$, \res).
One defines $\bG({\rm X}, \rmG)_{\ell}$ and $\bG(\rmE\rmG^{gm}{\underset \rmG \times}X)_{\ell}$ similarly. (Recall from ~\ref{mod-l-completion}, the
smash products above are all derived smash products.)
\vskip .2cm
Next we proceed to compare the derived completions with respect to the map $\rho_{\rmG}:\bK(\rmS, \rmG) \ra \bK(\rmS )$ for a split reductive group $\rmG$
and the derived completion with respect to the map $\rho_{\rmT}:\bK(\rmS, \rmT) \ra \bK(\rmS )$, where $\rmT$ is a maximal torus in $\rmG$.  Given a module
spectrum $M$ over $\bK(\rmS, \rmT)$, $res_*(M)$ will denote $M$ with the induced $\bK(\rmS, \rmG)$-module structure, where
$res:\bK(\rmS, G) \ra \bK(\rmS, \rmT)$ denotes the restriction map. Clearly there is a map $res_*(M) \ra M$ (which is the identity map on
the underlying spectra) compatible with the restriction map $\bK(\rmS, \rmG) \ra \bK(\rmS, \rmT)$. Therefore one obtains induced maps
$res_*(M){\overset L {\underset {\bK(\rmS, \rmG)} \wedge}} \bK(\rmS ) \cdots {\overset L {\underset {\bK(\rmS, \rmG)} \wedge}} \bK(\rmS ) \ra
M {\overset L {\underset {\bK(\rmS, \rmT)} \wedge}} \bK(\rmS ) \cdots {\overset L {\underset {\bK(\rmS, \rmT)} \wedge}} \bK(\rmS )$, which 
induces a map from the derived completion of $res_*(M)$ with respect to $\rho_{\rmG}$ to the derived completion of $M$ with respect 
to $\rho_{\rmT}$.
(Similar conclusions hold for the maps $\rho_{\ell} \circ \rho_{\rmG}$ and $\rho_{\ell} \circ \rho_{\rmT}$.) These may be seen more readily from the following lemma.
\begin{lemma}
 \label{res.1} Let $A \ra B \ra C$ denote maps of commutative ring spectra that are $-1$-connected. Let $f: B \ra C$ and $res: A \ra B$ denote
the given maps. Assume that $C$ is a cofibrant object in the category of commutative algebra spectra over $B$ and that $f: B \ra C$ is the resulting
cofibration. Then the following hold.
\vskip .2cm \noindent
(i) One may find a cofibrant replacement $\tilde B \ra B$ of $B$ in the model category of commutative algebra spectra over $A$ and $\tilde C \ra C$
of $C$ in the model category of commutative algebra spectra over $\tilde B$ so that the induced maps $ A \ra \tilde B$ and $ \tilde B \ra \tilde C$
are cofibrations in the model category of commutative algebra spectra over $A$. 
\vskip .2cm \noindent
(ii) Let $\widetilde {res}: A \ra \tilde B$ and $\tilde f: \tilde B \ra \tilde C$ denote the induced maps of ring spectra. For any $B$-module spectrum $M$,
one obtains an induced map of inverse systems of truncated cosimplicial) objects of spectra: 
\[\{\tau_{\le m}{\mathcal T}^{\bullet}_A(res_*(M), \tilde C) \ra \tau_{\le m}{\mathcal T}^{\bullet}_{\tilde B}(M, \tilde C) {\overset {\simeq} \ra} \tau_{\le m}{\mathcal T}^{\bullet}_{ B}(M,  C)|m\}\]
where $res_*(M)$ denote $M$ viewed as an $A$-module spectrum by restriction of scalars.
\vskip .2cm \noindent
(iii) Assume that $A \ra B $ and $B \ra C$ are both cofibrations in the category of commutative $A$-algebra spectra. Let $I_B$ ($I_A$) denote
the canonical homotopy fiber of the map $f: B \ra C$ (the composite map $A \ra C$, \res).  Then there exist cofibrant replacements $\tilde I_B \ra I_B$ in
 the category of module spectra over $B$
($\tilde I_A \ra I_A$ in the category of module spectra over $A$) together with a map $\tilde I_A \ra \tilde I_B$ compatible with the 
maps $A \ra B$. It follows that for a $B$-module spectrum as in (ii), one obtains an inverse system of compatible maps of spectra:
\[\{\tilde I_{A}^{ m+1} = {\overset {m+1} {\overbrace { \tilde I_{A} { {\underset {A} \wedge}}  \tilde I_{A}  \cdots { {\underset {A} \wedge}}  \tilde I_{A}}}}{ {\underset {A} \wedge}} (res_*(M)) \ra \tilde I_{ B}^{ m+1} = {\overset {m+1} {\overbrace { \tilde I_{ B} { {\underset { B} \wedge}}  \tilde I_{ B}  \cdots { {\underset { B} \wedge}}  \tilde I_{ B}}}}{ {\underset { B} \wedge}} (M)|m\}.\]
\end{lemma}
\begin{proof} (i) This follows readily from \cite[(3.10) Remark]{DS}, since the category of commutative $\tilde B$-algebra spectra may be viewed as the
 {\it under category} 
$\tilde B$ \textdownarrow $( \mbox {commutative A-algebra spectra})$ with the model structure induced from the model structure on the 
category of commutative $A$-algebra spectra, so that the cofibration $\tilde B \ra \tilde C$
will be a cofibration in the category of commutative $A$-algebra spectra. This proves (i). 
Then the existence of the  map $\{\tau_{\le m}{\mathcal T}^{\bullet}_A(res_*(M), \tilde C) \ra \tau_{\le m}{\mathcal T}^{\bullet}_{\tilde B}(M, \tilde C)$ in (ii) is immediate, making use of the map $A \ra \tilde B$ and observing that the composite map $A \ra \tilde B {\overset f \ra} \tilde C$
is a cofibration in the category of commutative $A$-algebra spectra. The existence of the map $\tau_{\le m}{\mathcal T}^{\bullet}_{\tilde B}(M, \tilde C) {\overset {\simeq} \ra} \tau_{\le m}{\mathcal T}^{\bullet}_{ B}(M,  C)|m\}$ is clear. That it is a weak-equivalence follows
from the fact that the maps $\tilde B \ra B$ and $\tilde C \ra C$ are weak-equivalences as well as the assumption that $C$ is a cofibrant algebra over $B$.
\vskip .2cm
(iii) Let $\tilde I_B \ra I_B$ denote a cofibrant replacement in the category of $B$-module spectra. Observe that this is a trivial fibration in the
same model category and hence a trivial fibration of the underlying spectra. Therefore, let $\hat I_A = I_A{\underset {I_B} \times} \tilde I_B$.
Then $\hat I_A$ is a module spectrum over $A$ and the induced map $\hat I_A \ra I_A$ is a map of $A$-module spectra which is a trivial fibration of
$A$-module spectra. Therefore a cofibrant replacement $\tilde I_A$ of $\hat I_A$ in the category of $A$-module spectra is a cofibrant replacement of $I_A$ in the category of $A$-module spectra
and provided with a map $\tilde I_A \ra \tilde I_B$ compatible with the map $A \ra B$. The last statement is clear.
\end{proof}
We will invoke the above lemma in the following theorem with $A= \bK(\rmS, \rmG)$, $B= \bK(\rmS, \rmT)$ and
 $C= \bK(\rmS )$ with $res: A \ra B$ and $f: B \ra C$ denoting the obvious maps induced by restriction
of groups.
\begin{theorem}
\label{T.comp.vs.G.comp}
Assume that $\rmS = Spec \, k$ for a field $k$ and that the split reductive group $\rmG$ has $\pi_1(\rmG)$ torsion-free. 
 Let $M$ be an $s$-connected module spectrum over $\bK(\rmS, \rmT)$ for some integer $s$. Then the inverse system of maps defined above, 
\[\{res_*(M)  \, {\widehat {}}_{\rho _{\rmG},m} \ra M \, {\widehat {}}_{\rho_{\rmT},m}|m\}\]
induces a  weak-equivalence on taking the homotopy inverse limit as $m \ra \infty$. A corresponding result holds for the completions with respect to $\rho_{\ell} \circ \rho_{\rmG}$ and $\rho_{\ell} \circ \rho_{\rmT}$.
\end{theorem}
\begin{proof}
In view of the arguments in ~\ref{der.compl.props}, one reduces immediately to proving the corresponding statement when
$M$ is an $R(\rmT)$-module and the completions are with respect to the rank-maps $R(\rmT) \ra \Z$ and $R(\rmG) \ra \Z$. 
 We will denote the rank-map $R(\rmT) \ra \Z$ by $\rho_{\rmT}$ and the rank-map $R(\rmG) \ra \Z$ by $\rho_{\rmG}$.
i.e. With this notation, we now reduce to proving 
\[\H(res_*(M)) \compl_{\H(\rho_{\rmG})} \simeq \H(M) \compl_{\H(\rho_{\rmT})}\]
where $\H$ denotes the Eilenberg-Maclane spectrum functor discussed in detail in Appendix A.
Next one observes the isomorphisms:
\[R(\rmT){\underset {R(\rmG)} \otimes} R(\rmG)/ker (\rho_{\rmG}) \cong R(\rmT){\underset {R(\rmT)} \otimes} R(\rmT)/(R(\rmT).ker (\rho_{\rmG})) \cong R(\rmT)/(R(\rmT).ker (\rho_{\rmG})).\]
\vskip .2cm \noindent
Next we make use of Remark ~\ref{alternate.form} to find a more convenient definition of the derived completions.
Let 
\[{\widetilde {R(\rmG)/ker(\rho_{\rmG})}}_{\bullet} \ra R(\rmG)/ker(\rho_{\rmG})\]
 denote a simplicial resolution by the free commutative algebra functor
over $R(\rmG)$ as discussed in  ~\ref{res.free.comm.alg}. 
It follows that if $N$ denotes the functor sending a simplicial abelian group to the corresponding
normalized chain complex,
\be \begin{equation}
\label{flat.res.0}
N(R(\rmT) {\underset {R(\rmG)} \otimes}{\widetilde {R(\rmG)/ker(\rho_{\rmG})}}_{\bullet}) \ra R(\rmT){\underset {R(\rmG)} \otimes} R(\rmG)/ker(\rho_{\rmG}) \cong R(\rmT)/(R(\rmT).ker (\rho_{\rmG}))
 \end{equation} \ee
\vskip .2cm \noindent
is a flat resolution by a commutative dg-algebra over $R(\rmT)$. One may observe that this step strongly  uses the assumption that
$\pi_1(\rmG)$ is torsion-free, so that $\rmR(\rmT)$ is free over $\rmR(\rmG)$. In view of the discussion in Appendix A (see especially ~\eqref{tens.struct}), this implies the weak-equivalence:
\be \begin{equation}
     \label{cof.res.H.0}
\H(R(\rmT)){\underset  {\H(R(\rmG))} \wedge }\H({\widetilde {R(\rmG)/ker(\rho_{\rmG})}}) \simeq \H(R(\rmT)/R(\rmT)ker(\rho_{\rmG})).
    \end{equation} \ee
Recall from the lines following ~\eqref{tens.struct}, $\H({\widetilde {R(\rmG)/ker(\rho_{\rmG})}}) =\hocolimD \{\H({\widetilde {R(\rmG)/ker(\rho_{\rmG})}}_n)|n\}$ is a commutative algebra spectrum over $\H(R(\rmG))$ and is 
cofibrant as an object in $Mod(\H(R(\rmG)))$. Now ~\eqref{flat.res.0} implies the identifications:
\be \begin{align}
     \label{cof.res.0}
M{\overset L{\underset {R(\rmT)} \otimes}}R(\rmT)/(R(\rmT).ker (\rho_{\rmG})) &\cong M {\underset {R(\rmT)} \otimes}R(\rmT) {\underset {R(\rmG)} \otimes}N({\widetilde {R(\rmG)/ker(\rho_{\rmG})}}_{\bullet}) \\
res_*(M){\overset L{\underset {R(\rmG)} \otimes}}R(\rmG)/(ker (\rho_{\rmG})) &\cong res_*(M){\underset {R(\rmG)} \otimes}N({\widetilde {R(\rmG)/ker(\rho_{\rmG})}}_{\bullet}). \notag
\end{align} \ee
\vskip .2cm \noindent
Here $res_*(M)$ denotes $M$ viewed as an $R(\rmG)$-module using restriction of scalars. These, together with ~\eqref{monoidal.prop.H},  then also provide the identifications:
\be \begin{align}
     \label{cof.res.0}
\H(M){\overset L{\underset {\H(R(\rmT))} \wedge}}\H(R(\rmT)/(R(\rmT).ker (\rho_{\rmG}))) &\cong \H(M) {\underset {\H(R(\rmT))} \wedge}\H(R(\rmT)) {\underset {\H(R(\rmG))} \wedge}\H({\widetilde {R(\rmG)/ker(\rho_{\rmG})}}) \\
\H(res_*(M)){\overset L{\underset {\H(R(\rmG))} \wedge}}\H(R(\rmG)/(ker (\rho_{\rmG}))) &\cong \H(res_*(M)){\underset {\H(R(\rmG))} \wedge}\H({\widetilde {R(\rmG)/ker(\rho_{\rmG})}}).\notag
\end{align} \ee
\vskip .2cm \noindent
Recall again that $\H({\widetilde {R(\rmG)/ker(\rho_{\rmG})}}) =\hocolimD \{\H({\widetilde {R(\rmG)/ker(\rho_{\rmG})}}_n)|n\}$ is a commutative algebra spectrum over $\H(R(\rmG))$ and is 
cofibrant as an object in $Mod(\H(R(\rmG)))$. Therefore, one may identify the two right-hand-sides of ~\eqref{cof.res.0} and therefore
the corresponding left-hand-sides.
 \vskip .2cm
Next one repeats the same argument with $M$ replaced by the terms in each degree of  
$M {\underset {R(\rmT)} \otimes}R(\rmT) {\underset {R(\rmG)} \otimes}{\widetilde {R(\rmG)/ker(\rho_{\rmG})}}_{\bullet}$ so that one obtains 
an identification of the cosimplicial objects of spectra:
\be \begin{equation}
     \label{cof.res.1}
{\mathcal T}^{\bullet}_{\H(R(\rmG))}(\H(res_*(M)), \H(R(\rmG)/(ker(\rho_{\rmG})))) \cong {\mathcal T}^{\bullet}_{\H(R(\rmT))}(\H(M),  \H(R(\rmT)/(R(\rmT). ker (\rho_{\rmG}))))
\end{equation} \ee
\vskip .2cm \noindent
Observe that the proof will be complete, if we show that there is a weak-equivalence of the homotopy inverse limit of the term on the right
with the homotopy inverse limit of the cosimplicial object ${\mathcal T}^{\bullet}_{\H(R(\rmT))}(\H(M),  \H(R(\rmT)/(ker (\rho_{\rmT}))))$.
\vskip .2cm
Therefore we will let $I$ denote either one of the ideals $ker (\rho_{\rmT})$ or $R(\rmT). ker (\rho_{\rmG})$.
Now Proposition ~\ref{derived.comp.for.rings} below provides  the weak-equivalence:
\be \begin{equation}
     \label{key.weak.eq.0}
   \holim  {\mathcal T}^{\bullet}_{\H(R(\rmT))}(\H(M),  \H(R(\rmT)/I)) \simeq \holimk \{\H(M){\overset L {\underset {\H(R(\rmT))} \wedge}} (\H(R(\rmT)/I^k))|k\}
 \end{equation} \ee
\vskip .2cm \noindent
where the first homotopy inverse limit is of the cosimplicial object there and the second homotopy inverse limit is that of the
inverse system. (Again we make use of the simplicial resolution by the free commutative algebra functor
over $R(\rmT)$ as discussed in  ~\ref{res.free.comm.alg} to define   ${\mathcal T}^{\bullet}_{\H(R(\rmT))}(\H(M),  \H(R(\rmT)/I))$.)
\vskip .2cm
Next one observes that $R(\rmT)$ is a 
finite $R(\rmG)$-module and hence that 
 there exists some positive integer $n_0$ so that $ker(\rho_{\rmT})^{n_0} \subseteq R(\rmT).ker(\rho_{\rmG}) \subseteq ker (\rho_{\rmT})$. 
Therefore, we proceed to show
that the inverse systems 
\vskip .2cm 
$\{\H(M){\overset L {\underset {\H(R(\rmT))} \wedge}}\H(R(\rmT)/(R(\rmT). ker (\rho_{\rmG}))^k)|k \ge 1\}$ and
$\{\H(M){\overset L {\underset {\H(R(\rmT))} \wedge}} \H(R(\rmT)/( ker (\rho_{\rmT}))^k)|k \ge 1\}$ 
\vskip .2cm \noindent
define weakly-equivalent homotopy inverse limits. By taking a simplicial resolution $\widetilde M^{\bullet} \ra M$ of $M$ by free $R(\rmT)$-modules, one may replace $\H(M)$ by $\hocolimD \H(\widetilde M^{\bullet})$.
Let $\widetilde {\H( M )}= \hocolimD \H(\widetilde M^{\bullet})$. Now it 
suffices to show that, 
\be \begin{equation}
 \label{pro.isom}
\{{\widetilde {\H(M)}}{\underset {\H(R(\rmT))} \wedge} \H(R(\rmT)/(R(\rmT).ker(\rho_{\rmG}))^k)|k \ge 1\} \mbox{ and } 
\{{\widetilde {\H(M )}}{\underset {\H(R(\rmT))} \wedge}\H(R(\rmT)/( ker (\rho_{\rmT}))^k)|k \ge 1\} 
\end{equation} \ee
are isomorphic as pro-objects of $\H(R(\rmT))$-module spectra. Recall that if
$P=\{P_{\alpha}|\alpha \}$ and $Q=\{Q_{\beta}| \beta\}$ are pro-objects of $\H(R(\rmT))$-modules spectra, then, 
\[Hom(P, Q) = \limbeta \colimalpha Hom_{\H(R(\rmT))}(P_{\alpha}, Q_{\beta}).\]

\vskip .2cm 
Therefore, it suffices to show that for any $\H(R(\rmT))$-module spectrum $K$, one obtains an isomorphism:
\be \begin{equation}
\begin{split}
\colimn Hom_{\H(R(\rmT))}({\widetilde {\H(M)}} {\underset {\H(R(\rmT))} \wedge}\H(R(\rmT)/(R(\rmT).ker(\rho_{\rmG}))^n), K) \\
\cong \colimn 
Hom_{\H(R(\rmT))}({\widetilde {\H(M )}}{\underset {\H(R(\rmT))} \wedge} \H(R(\rmT)/(ker(\rho_{\rmT}))^n), K) \end{split} \notag
\end{equation} \ee
\vskip .2cm \noindent
But the left-hand-side identifies with $\colimn Hom_{\H(R(\rmT))}(\H(R(\rmT)/(R(\rmT).ker(\rho_{\rmG}))^n), \Hom_{\H(R(\rmT))}({\widetilde {\H(M)}}, K))$ and the
 right-hand-side identifies with $\colimn Hom_{\H(R(\rmT))}(\H(R(\rmT)/(ker(\rho_{\rmT}))^n), \Hom_{\H(R(\rmT))}({\widetilde {\H(M)}}, K))$. 
Here $\Hom_{\H(R(\rmT))}$ denotes the internal hom in the category of $\H(R(\rmT))$-spectra. Observe also that
since $\H(R(\rmT))$ is a commutative ring spectrum, the internal Hom in the category of modules over $\H(R(\rmT))$ belongs to the same category.
  Since the pro-objects $\{\H(R(\rmT)/(R(\rmT).ker(\rho_{\rmG}))^n)|n\}$ and $\{\H(R(\rmT)/Ker(\rho_{\rmT})^n)|n\}$
are isomorphic, it follows that one obtains an isomorphism of the pro-objects in ~\eqref{pro.isom}.
\vskip .2cm
i.e. In view of ~\eqref{key.weak.eq.0},  we have shown that the two cosimplicial objects
\[{\mathcal T}^{\bullet}_{\H(R(\rmT))}(\H(M),  \H(R(\rmT)/(R(\rmT). ker (\rho_{\rmG})))) \mbox{ and } {\mathcal T}^{\bullet}_{\H(R(\rmT))}(\H(M),  
\H(R(\rmT)/( ker (\rho_{\rmT}))))\]
\vskip .2cm \noindent
 define
weakly-equivalent objects on taking the homotopy inverse limits.
In view of the identification in ~\eqref{cof.res.1}, it therefore also follows that the two cosimplicial objects 
\[{\mathcal T}^{\bullet}_{\H(R(\rmG))}(\H(res_*(M)),  \H(R(\rmG)/(ker (\rho_{\rmG})))) \mbox{ and } {\mathcal T}^{\bullet}_{\H(R(\rmT))}(\H(M),  
\H(R(\rmT)/(ker (\rho_{\rmT}))))\] 
\vskip .2cm \noindent
define weakly-equivalent homotopy inverse limits, i.e. $\H(res_*(M)) \, \, {\widehat {}}_{\H(\rho_{\rmG})} \simeq \H(M )\, \, {\widehat {}}_{\H(\rho_{\rmT})}$.
 The proof of the corresponding statement for the derived completions with respect to $\rho_{\ell} \circ \rho_{\rmG}$ and $\rho_{\ell} \circ \rho_{\rmT}$ is similar.
\end{proof}

\begin{remark}
 For the applications in the rest of the paper, we will restrict mostly to the case where the group $\rmG=\GL_n$ for some $n$ or a finite product $\Pi_{i=1}^m \GL_{n_i}$, since the
hypotheses of the Theorem are satisfied in this case.
\end{remark}

\begin{proposition} 
 \label{derived.comp.for.rings} 
Let $R$ denote a Noetherian ring and let $I$ denote an ideal in $R$. For each $k \ge 1$, 
let $S_k=R/I^k$ and let $\tilde S_k$ denote a free resolution of $S_k$ as a commutative dg-algebra over $R$ as in ~\ref{res.free.comm.alg}. 
Let $\tilde S= \tilde S_1$. 
\vskip .2cm
Let $T_{\H(R)}(\M, \H(\tilde S))$ denote the triple sending an $\H(R)$-module spectrum $\M$ to the $\H(R)$-module spectrum $\M {\underset {\H(R)} \otimes} \H(\tilde S)$ and let
${\mathcal T}^{\bullet}_{\H(R)}(\M)$ denote the corresponding cosimplicial resolution of the $\H(R)$-module $\M$. 
If $M$ is any
$R$-module, not necessarily finitely generated, then one obtains the identification (up to weak-equivalence):
\vskip .2cm
$\holim{\mathcal T}^{\bullet}_{\H(R)}(\H(M), \H(\tilde S)) \simeq \holimk \{\H(M ){\underset {\H(R)} \wedge}\H(\tilde S_k)|k\} \simeq 
\holimk \{\H(M{\underset R \otimes} \tilde S_k)|k\} \simeq \holimk \{\H(M{\overset L {\underset R \otimes}}R/I^k)|k\}.$
\vskip .2cm \noindent
(Again we make use of the simplicial resolution by the free commutative algebra functor
over $R$ as discussed in  ~\ref{res.free.comm.alg} to define   ${\mathcal T}^{\bullet}_{\H(R)}(\H(M),  \H(\tilde S))$.)

\end{proposition}
\begin{proof} See \cite[Corollary 5.5]{C08} for a somewhat similar result. 
First,  the definition of the cosimplicial object ${\mathcal T}^{\bullet}_{\H(R)}(\H(M),  \H(\tilde S))$ provides the identification:
\vskip .2cm
${\mathcal T}^{\bullet}_{\H(R)}(\H(M),  \H(\tilde S)) \cong 
\H(M){\overset L {\underset {\H(R)} \wedge}} {\mathcal T}^{\bullet}_{\H(R)}(\H(R),  \H(R/I))$.
\vskip .2cm \noindent
At this point, one observes that one obtains a distinguished triangle of pro-objects of cosimplicial objects (i.e. a diagram
of pro-cosimplicial spectra, which is a stable cofiber sequence for each $k$ and each cosimplicial degree):
\be \begin{equation}
\label{pro.map}
\begin{split}
\{\H(M){\overset L {\underset {\H(R)} \wedge}} {\mathcal T}^{\bullet}_{\H(R)}(\H(I^k),  H(R/I))|k\} \ra
\{\H(M){\overset L {\underset {\H(R)} \wedge}} {\mathcal T}^{\bullet}_{\H(R)}(\H(R),  \H(R/I))|k\} \\
\ra  \{\H(M){\overset L {\underset {\H(R)} \wedge}} {\mathcal T}^{\bullet}_{\H(R)}(\H(R/I^k),  \H(R/I))|k\} \end{split}
\end{equation} \ee
\vskip .2cm \noindent
The middle term $\{\H(M){\overset L {\underset {\H(R)} \wedge}} {\mathcal T}^{\bullet}_{\H(R)}(\H(R),  \H(R/I))|k\}$ is the constant pro-object
consisting of the same object $\H(M){\overset L {\underset {\H(R)} \wedge}} {\mathcal T}^{\bullet}_{\H(R)}(\H(R),  \H(R/I))$ for all $k$.
For each fixed $k$, and each fixed cosimplicial degree $m$, the homotopy groups of 
${\mathcal T}^{m}_{\H(R)}(\H(I^k),  \H(R/I))$ are computed by a strongly convergent spectral sequence
whose $E_2$-terms are  the cohomology of the iterated derived tensor product 
$I^k {\overset L {\underset {R} \otimes}} R/I \cdots  {\overset L {\underset {R} \otimes}}R/I$. 
\vskip .2cm 
For a finitely generated module
$K$ over $R$, 
we will denote the cohomology groups of the iterated derived tensor product
$K{\overset L {\underset {R} \otimes}} R/I \cdots  {\overset L {\underset {R} \otimes}}R/I$ as 
$MultiTor_{R,*}^m(K,R/I, \cdots, R/I)$.
One considers the functor $K \mapsto \{I^kK|k\}={\underline I}K$ sending a finitely generated module $K$ over the Noetherian ring $R$ to 
the pro-object
$\{I^kK|k \}$ in the category of $R$-modules. As 
observed in the proof of \cite[Theorem 4.4]{C08}, the pro-object 
\vskip .2cm
$\{MultiTor_{R,*}^m(I^k, R/I, \cdots, R/I)|k\}$ 
\vskip .2cm \noindent
identifies with the
pro-object 
\vskip .2cm
${\underline I}MultiTor_{R,*}^m(R, R/I, \cdots, R/I) = \{I^kMultiTor_{R, *}^m(R, R/I, \cdots, R/I)|k\}$.
\vskip .2cm \noindent
 Since the cohomology groups
of the iterated derived tensor product 
$R{\overset L {\underset {R} \otimes}} R/I \cdots  {\overset L {\underset {R} \otimes}}R/I$ are $R/I$-modules, it follows
that $I^k$ acts trivially on the above cohomology groups. These observations prove that
the pro-object $\{{\mathcal T}^{\bullet}_{R}(I^k,  R/I)|k\}$, and therefore the pro-objects
\[\{M{\overset L {\underset {R} \otimes}} {\mathcal T}^{\bullet}_{R}(I^k,  R/I)|k\} \mbox{  and } \{\pi_n(H(M){\overset L {\underset {\H(R)} \wedge}} {\mathcal T}^{\bullet}_{\H(R)}(\H(I^k), \H( R/I)))|k\}\]
for any fixed integer $n$ and each fixed cosimplicial degree 
are trivial. This shows that the last map in ~\eqref{pro.map} induces a  weak-equivalence
on taking the homotopy inverse limit  of the pro-objects (for each fixed cosimplicial degree).
\vskip .2cm
Next we fix an integer $k$ and consider the homotopy inverse limit of the cosimplicial object 
\newline \noindent
$\H(M){\overset L {\underset {\H(R)} \wedge}} {\mathcal T}^{\bullet}_{\H(R)}(\H(R/I^k),  \H(R/I))$.  The discussion in \cite{C08} 
immediately following Proposition 4.3,  as well as Lemma ~\ref{devissage} below now show that for any fixed $k$, the homotopy inverse limit of the
last cosimplicial object identifies with $\H(M){\overset L {\underset {\H(R)} \wedge}} (\H(R/I^k))$.   Finally we take the
homotopy inverse limits over $k \ra \infty$ followed by the homotopy inverse limits of the resulting cosimplicial objects for each of the
three terms in ~\eqref{pro.map}. The above discussion shows that this induces a weak-equivalence of the resulting last two terms.
However, one may clearly interchange the two homotopy inverse limits. Since the middle term in ~\eqref{pro.map} is a constant pro-object, the above 
 observations now
provide the weak-equivalences:
\be \begin{equation}
     \label{key.weak.eq}
     \begin{split}
   \holim  {\mathcal T}^{\bullet}_{\H(R)}(\H(M),  \H(R/I)) \simeq \holim (\H(M){\overset L {\underset {\H(R)} \wedge}} {\mathcal T}^{\bullet}_{\H(R)}(\H(R),  \H(R/I))) \simeq \holimk \{\H(M){\overset L {\underset {\H(R)} \wedge}} (\H(R/I^k))|k\} \\
=\holimk \{\H(M) {\underset {\H(R)} \wedge}\H(\tilde S_k)|k\} \simeq \holimk \{\H(M{\underset R \otimes}\tilde S_k)|k\} \end{split}
 \end{equation} \ee
\vskip .2cm \noindent
where the first two homotopy inverse limits are of the cosimplicial objects there and the remaining homotopy inverse limits are that of the
inverse system.
\end{proof}
\begin{lemma}
 \label{devissage}
Let $\rho:A \ra B$ denote a surjective map of commutative Noetherian rings with $I=ker(\rho)$. 
For any chain  complexes of 
$A$-modules $M$ and $N$ that are trivial in negative degrees,
the augmentation 
\[ \H(N){\overset L {\underset {\H(A)} \wedge}} \H(M/I^n) \ra \holimD \H(N){\overset L {\underset {\H(A )} \wedge}} 
{\mathcal T}_{\H(A)}^{\bullet}(\H(M/I^n), \H(B))  \] is a weak-equivalence for any fixed positive integer $n$.
(Again we make use of the simplicial resolution by the free commutative algebra functor
over $A$ as discussed in  ~\ref{res.free.comm.alg} to define   
${\mathcal T}_{\H(A)}^{\bullet}(\H(M/I^n), \H(B))$.)
\end{lemma}
\begin{proof} The proof will be by ascending induction on $n$. The assertion for $n=1$, is 
essentially \cite[Proposition 3.2, 6]{C08}. i.e. In this case, it suffices to observe that the cosimplicial
object $\H(N){\overset L {\underset {\H(A)} \wedge}} {\mathcal T}_{\H(A)}^{\bullet}(\H(M/I^n), \H(B))  $ has an 
extra co-degeneracy which provides a contracting homotopy. Next we consider the case of a general $n$. By taking a free resolution, we may assume 
$M$ and $N$ consist of free $A$-modules in each degree.
 In view of the above flatness assumptions on $M$, one has the short-exact sequence (of complexes):
\vskip .2cm
$0 \ra M {\underset {A} \otimes}I^{k-1}/I^k \ra M/I^{k} \ra M/I^{k-1} \ra 0$. 
\vskip .2cm \noindent
This provides a commutative diagram of cosimplicial objects:
\vskip .2cm
\xymatrix{{\H(N){\overset L {\underset {\H(A)} \wedge}} (\H(M {\underset A  \otimes} I^{k-1}/I^k))} \ar@<1ex>[r] \ar@<-1ex>[d] & {\H(N){\overset L {\underset {\H(A)} \wedge}} (\H(M/I^k))} \ar@<1ex>[r] \ar@<-1ex>[d] & {\H(N){\overset L {\underset {\H(A)} \wedge}} (\H(M/I^{k-1}))} \ar@<-1ex>[d]\\
  {\H(N){\overset L {\underset {\H(A)} \wedge}} {\mathcal T}^{\bullet}_{\H(A)}(\H(M{\underset A \otimes}I^{k-1}/I^k), \H(B))}  \ar@<1ex>[r] &{\H(N){\overset L {\underset {\H(A)} \wedge}} {\mathcal T}^{\bullet}_{\H(A)}(\H(M/I^{k}), \H(B))} \ar@<1ex>[r] &  P }
\vskip .2cm \noindent
where $P = {\H(N){\overset L {\underset {\H(A)} \wedge}} {\mathcal T}^{\bullet}_{\H(A)}(\H(M/I^{k-1}), \H(B))}$. The two rows are clearly stable 
cofiber sequences and remain so on taking the homotopy inverse limit of the cosimplicial objects in the last row. Observe that $M {\underset {A} \otimes}I^{k-1}/I^k$ is an $A/I=B$-module. Therefore
the case $n=1$ applies to show the first vertical map is a weak-equivalence on taking the homotopy inverse limit of the cosimplicial objects.
By the inductive assumption, the last vertical map induces a weak-equivalence on taking the homotopy inverse limit. Therefore, the 
middle map also 
induces a weak-equivalence on taking the homotopy inverse limit.
\end{proof}
\section{\bf Reduction to the case of a torus in ${\rmG}$ and an outline of the proof of the main theorem}
\label{reduct.torus}
Recall that following the hypotheses of Theorem ~\ref{main.thm.1}: $\rmG$ denotes a split reductive group acting on  a scheme $\rmX$, where both $\rmG$ and $\rmX$ satisfy the standing hypothesis ~\ref{stand.hyp.1} and $\rmH$ denotes a closed subgroup scheme. 
Making use of
the identifications $\bG(\rm{\rm X}, \rmH) \simeq \bG({\rmG}{\underset {\rmH}  \times}\rm{\rm X}, {\rmG})$, $\bG(\rmE\rmH^{gm}{\underset {\rmH}  \times}\rmX) \simeq 
\bG(\rmE{\rmG}^{gm}{\underset {\rmH}  \times}({\rmG}{\underset {\rmH}  \times}\rmX))$, we reduce to considering only actions of the group-scheme $\rmG$. 
i.e. We may assume without loss of generality, in the rest of this section that $\rmG= GL_n$ for some $n \ge 1$.
\vskip .2cm
Next we recall a particularly nice way to construct geometric classifying spaces through what are called {\it admissible gadgets}. 
\vskip .2cm
\begin{definition} (Admissible gadgets) (See \cite[4.2]{MV} and also \cite[3.1]{K}.)
\label{adm.gadg}
The first step in constructing an admissible gadget is to start  with a good pair $(\rmW, \rmU)$ for $\rmG$: this is a pair 
$(\rmW, \rmU)$ of smooth schemes over $k$
where $\rmW$ is a $k$-rational representation of $\rmG$ and
$\rmU \subsetneq \rmW$ is a $\rmG$-invariant non-empty open subset which is a smooth scheme with a free
action by $\rmG$, so that the quotient $\rmU/\rmG$ is a scheme. It is known ({\sl cf.} \cite[Remark~1.4]{tot}) that a 
good pair for $\rmG$ always exists.
The following choice of a good pair is often convenient. Choose a faithful $k$-rational representation $\rmR$ of $\rmG$ of dimension $n$. i.e. $\rmG$ admits a closed
immersion into $\GL(\rmR)$. Then 
$\rmG$ acts freely on an open subset $\rmU$ of $\rmW = End_k(\rmR)$. (Observe that  if $\rmR=k^{\oplus n}$, $\rmW= \rmR^{\oplus n}$. Now one may take $\rmU=\GL(\rmR)$.)
Let $Z = \rmW \setminus \rmU$. A sequence of pairs $ \{ (\rmW_i, \rmU_i)|i \ge 1\}$ of smooth schemes
over $k$ is called an {\sl admissible gadget} for $\rmG$, if there exists a
good pair $(\rmW, \rmU)$ for $\rmG$ such that $\rmW_i = \rmW^{\times ^{ i}}$ and $\rmU_i \subsetneq \rmW_i$
is a $\rmG$-invariant open subset such that the following conditions hold for each $i \ge 1$. 
\vskip .2cm
\begin{enumerate}
\item
$\left(\rmU_i \times \rmW\right) \cup \left(\rmW \times \rmU_i\right)
\subseteq \rmU_{i+1}$ as $\rmG$-invariant open subsets.
\item $\{\codim_{\rmU_{i+1}}\left(\rmU_{i+1} \setminus \left(\rmU_{i} \times \rmW\right)\right)|i\}$ is 
a strictly increasing sequence,  
\newline \noindent
i.e. $\codim_{\rmU_{i+2}}\left(\rmU_{i+2} \setminus 
\left(\rmU_{i+1} \times \rmW\right)\right) > 
\codim_{\rmU_{i+1}}\left(\rmU_{i+1} \setminus \left(\rmU_{i} \times \rmW\right)\right)$.
\item $\{\codim_{\rmW _i}\left(\rmW_i \setminus {\rmU _i}\right)|i\}$ is a strictly increasing sequence, 
\newline \noindent
i.e. $\codim_{\rmW_{i+1}}\left(\rmW_{i+1} \setminus \rmU_{i+1}\right) > \codim_{\rmW_i}\left(\rmW_i \setminus \rmU_i\right)$.
\item
$\rmU_i$ is a smooth quasi-projective scheme over $k$ with a free $\rmG$-action, so that the quotient $\rmU_i/G$ is also a smooth quasi-projective
scheme over $k$.
\end{enumerate}
\end{definition}
A {\it particularly nice example} of an admissible gadget for $\rmG$ can be constructed as follows. 
Given a good pair $(\rmW, \rmU)$, we now let
\be \begin{equation}
\label{adm.gadget.1}
 \rmW_i = \rmW^{\times ^{ i}},  \rmU_1 = \rmU \mbox{ and } E\rmG^{gm,i}= \rmU_{i+1} = \left(\rmU_i \times \rmW \right) \cup
\left(\rmW \times \rmU_i \right) \mbox{ for }i \ge 1.
\end{equation} \ee
where $\rmU_{i+1}$ is  viewed as a subscheme of 
$\rmW^{\times ^{i+1}} $. Observe that 
\be \begin{equation}
     \label{EG.fin}
\rmE\rmG^{gm,i}=U_{i+1} = \rmU \times \rmW^{\times i} \cup \rmW \times \rmU \times \rmW^{\times i-1} \cup \cdots \cup \rmW^{\times i} \times \rmU
    \end{equation} \ee
\vskip .2cm \noindent
Setting 
$Z_1 = Z$ and $Z_{i+1} = \rmU_{i+1} \setminus \left(\rmU_i \times \rmW\right)$ for 
$i \ge 1$, one checks that $\rmW_i \setminus \rmU_i = Z^{\times ^{i}}$ and
$Z_{i+1} = Z^{\times ^{i}} \oplus \rmU$.
In particular, $\codim_{\rmW_i}\left(\rmW_i \setminus \rmU_i\right) =
i (\codim_{\rmW}(Z))$ and
$\codim_{\rmU_{i+1}}\left(Z_{i+1}\right) = (i+1)d - i(\dim(Z))- d =  i (\codim_{\rmW}(Z))$,
where $d = \dim(\rmW)$. Moreover, $\rmU_i \to {\rmU_i}/\rmG$ is a principal $\rmG$-bundle and that the quotient $\rmV_i= \rmU_i/\rmG$ exists 
 as a smooth quasi-projective scheme (since the $\rmG$-action on $\rmU_i$ is free and $\rmU/\rmG$ is a scheme). We will 
often use $\rmE\rmG^{gm,i}$ to denote the $i+1$-th term of an admissible gadget $\{\rmU_i|i \}$.
\vskip .2cm
In particular if $\rmG=GL_n$, one starts with a faithful representation on the affine space $\rmV={\mathbb A}^n$. Let
$\rmW= End(\rmV)$ and $\rmU= GL(\rmV) = GL_n$.  
\vskip .2cm
In case $\rmG=\rmT = {\mathbb G}_m^n$, $\rmE\rmT^{gm,i}$ will be defined
 as follows. Assume first that $\rmT={\mathbb G}_m$. Then one may let $\rmE\rmT^{gm,i}= {\mathbb A}^{i+1}-0$ with the
diagonal action of $\rmT={\mathbb G}_m$ on ${\mathbb A}^i$. Now $\rmB\rmT^{gm,i+1} = {\mathbb P}^i$, which clearly has a Zariski open covering by $i+1$ affine spaces.
If $\rmT= {\mathbb G}_m^n$, then 
\be \begin{equation}
     \label{toric.case}
 \rmE\rmT^{gm,i}= ({\mathbb A}^{i+1}-0)^{\times n} 
\end{equation} \ee
with the $j$-th copy of ${\mathbb G}_m$ acting on the $j$-th factor ${\mathbb A}^{i+1}-0$. 
Now $\rmB\rmT^{gm,i} = ({\mathbb P}^i)^{\times n}$.

Next we consider the following preliminary result.
\begin{proposition}
 \label{coveringlemma} If $\rmG=\GL_n$, a finite product of $\GL_n$s or a split torus, then there exists a finite Zariski open covering, $\{V_{\alpha}|\alpha\}$ of each $\rmB\rmG^{gm,i}$ with each
$\rmV_{\alpha}$ being ${\mathbb A}^1$-acyclic, i.e. each $\rmV_{\alpha}$ has a $k$-rational point $p_{\alpha}$ so that $\rmV_{\alpha}$ and $p_{\alpha}$ are equivalent in the ${\mathbb A}^1$-homotopy category.
In fact one may choose a common $k$-rational point $p$ in the intersection of all the open $\rmU_{\alpha}$. Moreover, in all three cases, one may choose
a Zariski open covering of $\rmB\rmG^{gm,i}$ of the above form,  to consist of $i+1$ or more open sets.
\end{proposition}
\begin{proof} We will recall the construction of the geometric classifying spaces $\rmB\rmG^{gm,i}$  given above.
 Observe that when $\rmG=\GL_n$, $\rmU=\GL_n$ so that
\[\rmE\rmG^{gm,i}= \rmU_{i+1} = \GL_n \times \rmW^{\times i} \cup \rmW \times \GL_n \times \rmW^{\times i-1} \cup \cdots \cup \rmW^{\times i} \times \GL_n \]
Therefore $\rmB\rmG^{gm,i}$ is covered by the Zariski open cover
$V_1= \GL_n{\underset {\GL_n} \times} \rmW ^{\times i},  V_2=\rmW {\underset { \GL_n } \times}\GL_n \times \rmW^{\times i-1}, \cdots, V_{i+1}= \rmW^{\times i} {\underset {\GL_n} \times}GL_n$ each of which is isomorphic to $\rmW^{\times i}$ and hence ${\mathbb A}^1$-acyclic.
Observe that there are exactly $i+1$-open sets in this cover.
Moreover, one may choose a $k$-rational point $q \eps \GL_n= \rmU$ and let $p' =( q, \cdots , q) \eps \rmW^{\times i+1}$
denote a $k$-rational point. Clearly this belongs to $\rmE\rmG^{gm,i}$ and provides a common $k$-rational point in 
the intersection of the above open cover of $\rmB\rmG^{gm, i}$. When $\rmG = \GL_{n_1} \times \cdots \times \GL_{n_m}$, one
may take for $\rmB\rmG^{g,m} = \rmB\GL_{n_1}^{gm,m} \times \cdots \times \rmB\GL_{n_m}^{gm,m}$, so that the conclusion is also clear in this
case.
\vskip .2cm
If $\rmT= {\mathbb G}_m^n$, then recall $\rmE\rmT^{gm,i}= ({\mathbb A}^{i}-0)^{\times n} $
with the $j$-th copy of ${\mathbb G}_m$ acting on the $j$-th factor ${\mathbb A}^{i}-0$. 
Now $\rmB\rmT^{gm,i} = ({\mathbb P}^i)^{\times n}$.  The conclusions are clear in this case.\end{proof}
\vskip .2cm
\begin{remark}
\label{two.class.spaces}
Let $\rmG=\rmT={\mathbb G}_m^n$. Now we have already two constructions for $\rmE\rmT^{gm, i}$.  The first, which we
denote by  $\rmE\rmT^{gm,i}(1)$ starts with $\rmE\rmT^{gm,1}(1) = \GL_n$ with the successive $\rmE\rmT^{gm,i}(1)$ obtained by applying the construction in 
~\eqref{adm.gadget.1} with $\rmE\rmT^{gm,i}(1) = \rmE\GL_n^{gm,i}$. Let $\rmE\rmT^{gm, i}(2)$ denote the second construction with $\rmE\rmT^{gm,i} = ({\mathbb A}^{i+1}-0)^n$. The explicit relationship between two such constructions
is often delicate. But as shown in Appendix B, the equivariant Borel-style generalized cohomology is independent of the choice of 
the ind-scheme $\{\rmE\rmG^{gm,m}|m\}$ used in defining a Borel-style generalized
cohomology, at least on considering smooth schemes of finite type over $k$.
\end{remark}
\begin{definition}
 \label{rel.KG}
Let the smooth scheme $\rmX$ be provided with an action by the linear algebraic group $\rmG$ and let
$\rmA$ denote a locally closed $\rmG$-stable smooth subscheme of $\rmX$. Then we define $\bK({\rm X}, \rmA, \rmG)$ as
the canonical homotopy fiber of the restriction $\bK({\rm X}, \rmG) \ra \bK(\rmA, \rmG)$. 
\end{definition}
Now we obtain the following result whose proof is straightforward and is therefore skipped.
\begin{lemma}
 \label{pairing.rel.KG}
 Let  $\rmX$ and $\rmY$ be smooth schemes provided with actions by the linear algebraic
group $\rmG$. Let $\rmA$ ($\rmB$) denote a locally closed smooth $\rmG$-stable subscheme of $\rmX$ ($\rmY$, \res).
Then one obtains a pairing:
\[\bK({\rm X}, \rmA, \rmG) {\overset L {\underset {\bK(Spec \, k, \rmG)} \wedge}} \bK(\rmY, \rmB, \rmG) \ra \bK(\rmX \times \rmY, \rmA \times \rmY \cup \rmX \times \rmB, \rmG)\]
that is contravariantly functorial in all the arguments.
\end{lemma}

\begin{proposition} 
\label{factoring.through.part.compl}
Let $\rmG$ denote $\GL_n$ for some $n$ or a split torus. Let $\rmX$ denote a scheme or algebraic space of finite 
type over $k$. Then, with the choice of the classifying spaces as in ~\eqref{EG.fin} and ~\eqref{toric.case}, the following hold.
\vskip .1cm
(i) If $I_G$ denotes 
the homotopy fiber of the restriction map
 $\bK(Spec \, k, \rmG) \ra \bK(Spec \, k)$, then the (obvious) map 
\[\bG(\rm{\rm X}, \rmG) \ra \bG(\rmE\rmG^{gm, m} \times \rm{\rm X}, \rmG) 
\simeq \bG(\rmE\rmG^{gm,m} \times _{\rmG} \rmX)\]
 factors
through $\bG(\rm{\rm X}, \rmG)/I_{\rmG}^{Lm'+1}\bG(\rm{\rm X}, \rmG)$ for some positive integer $m' \ge m$. Here
\[I_{\rmG}^{L m'+1}\bG(\rm{\rm X}, \rmG) = {\overset {m'+1} {\overbrace { I_{\rmG} {\overset L {\underset {\bK(Spec \, k, \rmG)} \wedge}}  I_{\rmG}  \cdots {\overset L {\underset {\bK(Spec \, k, \rmG)} \wedge}}  I_{\rmG}}}}{\overset L {\underset {\bK(Spec \, k, \rmG)} \wedge}} \bG(\rm{\rm X}, \rmG),\]
$I_{\rmG}$ is the homotopy fiber of the restriction $\bK(Spec \, k, \rmG) \ra \bK(Spec \, k)$ and 
$\bG(\rm{\rm X}, \rmG)/I_{\rmG}^{Lm'+1}\bG(\rm{\rm X}, \rmG)$ is the homotopy cofiber of the map $I_{\rmG}^{Lm'+1}\bG(\rm{\rm X}, \rmG) {\overset L {\underset {\bK(Spec \, k, \rmG)} \wedge}} \bG(\rm{\rm X}, \rmG) \ra \bG(\rm{\rm X}, \rmG)$.
\vskip .1cm
(ii) It follows that the map $\bG(\rm{\rm X}, \rmG) \ra  
\bG(\rmE\rmG^{gm,m} \times \rm{\rm X}, \rmG)$ factors through the partial
derived completion $\bG(\rm{\rm X}, \rmG)\compl_{\rho_{\rmG}, m'}$ for some positive integer $m' \ge m+1$. If $\pi: \rmX \ra \rmY$ is a flat map (proper map with finite cohomological dimension), then
the above factorization is compatible with the induced map $\pi^*:\bG(\rmY, \rmG) \ra \bG(\rm{\rm X}, \rmG)$ 
(the induced map $\pi_*:\bG(\rm{\rm X}, \rmG) \ra \bG(\rmY, \rmG)$, \res.)
\vskip .2cm
(iii) More generally, if $\rmG_i$, $i=1, \cdots, q$ are either general linear groups or split tori, the map
\[\bG(\rm{\rm X}, \rmG_1 \times \cdots \times \rmG_q) \ra \bG(\rmE\rmG_1^{gm, m} \times \cdots \times \rmE\rmG_q^{gm, m} \times \rm{\rm X}, \rmG_1 \times \cdots \times \rmG_q)\]
factors through 
\[\bG(\rm{\rm X}, \rmG_1 \times \cdots \times \rmG_q) \compl_{\rho_{\rmG_1 \times \cdots \rmG_q}, m'}\] for 
some positive integer $m' \ge m+1$. Let $\pi: \rmX \ra \rmY$ denote an equivariant map with respect to the action of $\rmG= \rmG_1 \times \cdots \times \rmG_q$.
 If $\pi:\rmX \ra \rmY$ is also a flat  map (also a proper map with finite cohomological dimension), then
the above factorization is compatible with the induced map $\pi^*:\bG(\rmY, \rmG_1 \times \cdots \times \rmG_q) \ra \bG(\rm{\rm X}, \rmG_1 \times \cdots \times \rmG_q)$ (the induced map $\pi_*:\bG(\rm{\rm X}, \rmG_1 \times \cdots \times \rmG_q) \ra \bG(\rmY, \rmG_1 \times \cdots \times \rmG_q)$, \res.)
\end{proposition}
\begin{proof} This parallels the original proof in \cite{AS69} for the usual completion for equivariant topological K-theory. For a scheme $\rmY$, with a
chosen $k$-rational point $p$, let 
$\tilde \bK(\rmY)$ 
 denote the homotopy fiber of the obvious restriction map $ \bK(\rmY) \ra \bK(p)$
corresponding to the immersion of the given $k$-rational point $p \ra \rmX$. (Observe that $\tilde \bK(\rmY)$ identifies up to 
weak-equivalence with the homotopy cofiber of the map $\pi^*: \bK(p) \ra \bK(\rmY)$, where $\pi: \rmY \ra Spec \, k=p$ is
the structure map of $\rmY$. Therefore, $\tilde \bK(\rmY)$ is independent of the choice of the $k$-rational point $p$ in $\rmY$.)
Let $\rmY=BG^{gm,m}$ for some fixed $m$ and let $\{\rmV_i|i\}$ denote
a finite Zariski open cover of $\rmY$ with chosen $k$-rational point $p \eps \cap _i \rmV_i$, so that $p \simeq \rmV_i$ in 
the ${\mathbb A}^1$-homotopy category.
We let $\bK(\rmV_i, p) = $ the homotopy fiber of the restriction $\bK(\rmV_i) \ra \bK(p)$. Clearly each 
$\bK(\rmV_i, p)$ is weakly-contractible. This follows from the homotopy property for equivariant $\rmG$-theory and the property ~\eqref{PD},
since the $\rmV_i$ are regular. The cofiber-sequence $\rmV_i/p \ra Y/p\ra Y/V_i$ shows 
that one obtains the stable fiber sequence: $\bK(\rmY, V_i) \ra \bK(\rmY, p) \ra \bK(V_i, p)$.
Since $\bK(\rmV_i, p)$ is weakly-contractible, it follows that the map $\bK(\rmY, V_i) \ra \bK(\rmY, p)$ is a weak-equivalence for each $i$. 
\vskip .2cm
Next observe the weak-equivalences (since the action of $\rmG$ is free on $\rmEG^{gm,m}$) and where $\pi_m: \rmEG^{gm,m} \ra \rmBG^{gm,m}$ is the given projection:
\be \begin{multline}
     \label{key.ids}
\bK(\rmBG^{gm,m}) \simeq \bK(\rmEG^{gm,m},\rmG), \, \bK(\rmBG^{gm,m}, V_i) \simeq \bK(\rmEG^{gm,m}, \pi_m^{-1}(V_i), \rmG), \\
 \bK(\rmBG^{gm,m}, p) \simeq \bK(\rmEG^{gm,m}, \pi_m^{-1}(p), \rmG) \mbox{ and } \bK(V_i, p) \simeq \bK(\pi_m^{-1}(V_i), \pi_m^{-1}(p), \rmG).\end{multline} \ee
\vskip .2cm \noindent
In particular, it follows that all terms on the  right sides above  are module spectra over $\bK(Spec \, k, \rmG)$.
Next recall from Proposition ~\ref{coveringlemma}, that the cardinality of the given open cover $\{\rmV_i|i\}$ of $\rmB\rmG^{gm,m}$ is $m'$ for some  $m' \ge m+1$, which we will denote by $n$, for 
convenience in the following argument. With $\rmY= BG^{gm,m}$ again, let ${\mathcal V} = \rmV_1 \times \rmY^{n-1} \bigcup \rmY \times \rmV_2 \times \rmY^{n-2} \bigcup
\rmY \times \rmY \times \rmV_3 \times \rmY^{n-3} \bigcup \cdots \bigcup \rmY^{n-1} \times \rmV_n$. We let $\tilde \rmY = \rmEG^{gm,m}$, $\tilde \rmV_i = \pi_m^{-1}(\rmV_i)$,
 $\tilde {\mathcal V}  = {\tilde \rmV}_1 \times {\tilde \rmY}^{n-1} \bigcup \tilde \rmY \times {\tilde \rmV}_2 \times {\tilde \rmY}^{n-2} \bigcup
{\tilde \rmY} \times {\tilde \rmY } \times {\tilde \rmV}_3 \times {\tilde \rmY}^{n-3} \bigcup \cdots \bigcup {\tilde \rmY}^{n-1} \times {\tilde \rmV}_n$.
\vskip .2cm
For schemes $\rmW$ and $\rmW'$ with actions by the group $\rmG$, with $\rmW' $ a subscheme of $\rmW$, we let $\bK(\rmW, \rmW', \rmG)$ = the homotopy fiber of
the restriction $\bK(\rmW, \rmG) \ra \bK(\rmW', \rmG)$. 
We let $\tilde {\bK}(\tilde \rmY, \rmG) $= the homotopy fiber of  the map $\bK(\tilde \rmY, \rmG) \ra \bK(\pi_m^{-1}(p), \rmG)$, and
$ {\bK}(\tilde \rmY, \tilde \rmV_i, \rmG)$ = the homotopy fiber of the map 
$\tilde \bK( \tilde \rmY, \rmG) = \bK(\tilde \rmY, \pi_m^{-1}(p), \rmG) \ra \bK(\pi_m^{-1}(\rmV_i), \pi_m^{-1}(p), \rmG)$. (Since the  last term is
weakly contractible, it follows that the map $\bK(\tilde \rmY, \tilde \rmV_i, \rmG) \ra \tilde \bK(\tilde \rmY, \rmG)$
 is a weak-equivalence.) ${{\tilde \bK}({\tilde \rmY}^{\times ^n}, \rmG)}$ is defined similarly. Now the following diagram
\vskip .2cm
\be \begin{equation}
     \label{K.diagm}
\xymatrix{{\bK(\tilde \rmY, {\tilde \rmV}_1, \rmG) \wedge^L \bK(\tilde \rmY, {\tilde \rmV}_2, \rmG) \wedge^L \cdots \wedge^L \bK(\tilde \rmY, {\tilde \rmV}_n, \rmG)} \ar@<1ex>[r]^(.56){\simeq} \ar@<-1ex>[d]^{\times} & {{\tilde \bK}(\tilde \rmY, \rmG) \wedge^L {\tilde \bK}(\tilde \rmY, \rmG) \wedge^L \cdots {\tilde \bK}(\tilde \rmY, \rmG)} \ar@<1ex>[d]\\
{\bK({\tilde \rmY}^{\times n}, \tilde {\mathcal V}, \rmG)} \ar@<1ex>[r] \ar@<-1ex>[d]^{\Delta^*} & {{\tilde \bK}({\tilde \rmY}^{\times ^n}, \rmG)} \ar@<1ex>[d] ^{\Delta^*} \\
{\bK(\tilde \rmY, \tilde \rmY, \rmG)} \ar@<1ex>[r] & {\tilde \bK}(\tilde \rmY, \rmG)}
\end{equation} \ee
\vskip .2cm \noindent
commutes up to homotopy. The pairings forming the top vertical maps in the left and and right column
are provided by Lemma ~\ref{pairing.rel.KG}. The contravariant functoriality of the pairings there
show that the top square homotopy commutes. Since all maps are maps of module spectra over $\bK(Spec \, k, \rmG)$ (as observed in ~\eqref{key.ids}), the $\wedge^{\rm L}$ above is 
over $\bK(Spec \, k, \rmG)$. Since the top horizontal map is a weak-equivalence, and since $\bK(\tilde \rmY, \tilde \rmY, \rmG)$ is weakly contractible, it follows that the 
composition of the maps in the right column is null-homotopic.
\vskip .2cm
Next consider the homotopy commutative diagram
\vskip .2cm
\xymatrix {{I_{\rmG}} \ar@<1ex>[r] \ar@<-1ex>[d] & {\bK(Spec \, k, \rmG)} \ar@<1ex>[r] \ar@<1ex>[d] & {\bK(Spec \, k)}  \ar@<1ex>[d]^{id}\\
{{\tilde \bK}(\rmE\rmG^{gm,m}, \rmG)} \ar@<1ex>[r] & {{ \bK}(\rmE\rmG^{gm,m}, \rmG)} \ar@<1ex>[r] & {\bK(\rmG, \rmG) \simeq \bK(Spec \, k)}}
\vskip .2cm \noindent
Taking $\rmY=\rmB\rmG^{gm,m}$, $\tilde \rmY = \rmE \rmG^{gm,m}$  in ~\eqref{K.diagm}, we see that the map 
\[I_{\rmG} ^{L\wedge ^n}=  {\overset {n} {\overbrace { I_{\rmG} {\underset {\bK(Spec \, k, \rmG)} {\overset L {\wedge}}} I_{\rmG}  \cdots {\underset {\bK(Spec \, k, \rmG)} {\overset L {\wedge}}} \ I_{\rmG} }}}\ra {\tilde \bK}(\rmE\rmG^{gm, m}, \rmG ) \simeq {\tilde \bK}(B\rmG^{gm,m})\]
 is null-homotopic, 
since it factorizes as 
\[I_{\rmG}^{L\wedge ^n} \ra {\tilde \bK}(\rmE\rmG^{gm,m}, G)^{L\wedge^n} =  {\overset {n} {\overbrace { {\tilde \bK}(\rmE\rmG^{gm,m}, G) {\underset {\bK(Spec \, k, \rmG)} {\overset L {\wedge}}} {\tilde \bK}(\rmE\rmG^{gm,m}, G)  \cdots {\underset {\bK(Spec \, k, \rmG)} {\overset L {\wedge}}} {\tilde  \bK}(\rmE\rmG^{gm,m}, G) }}}\]
\[\ra {\tilde \bK}(\rmE\rmG^{gm, m}, \rmG).\]
Therefore,  the middle vertical map
$\bK(Spec \, k, \rmG) \ra \bK(\rmE\rmG^{gm, m}, \rmG)$ factors through $\bK(Spec \, k, \rmG)/I_{\rmG}^{L\wedge ^n}$
 which is the homotopy cofiber of the map $I_{\rmG}^{L\wedge ^n} \ra \bK(Spec \, k, \rmG)$. Finally one makes use of the pairing:
$\bK(Spec \, k, \rmG) {\overset L \wedge} \bG(\rm{\rm X}, \rmG) \ra \bG(\rm{\rm X}, \rmG)$ and the homotopy commutative square (where the vertical maps are the obvious ones)
\be \begin{equation}
     \label{null.homotopic}
\xymatrix{{{I_{\rmG}^{L\wedge ^n}}  {\overset L {\wedge}} \bG(\rm{\rm X}, \rmG)} \ar@<1ex>[r] \ar@<1ex>[d] & {\bK(Spec \, k, \rmG)  {\overset L {\wedge}} \bG(\rm{\rm X}, \rmG)} \ar@<1ex>[r] \ar@<1ex>[d] & {\bG(\rm{\rm X}, \rmG)} \ar@<1ex>[d]\\
{\bK(\rmE\rmG^{gm, m}, \rmG)  {\overset L {\wedge}} \bG(\rm{\rm X}, \rmG)} \ar@<1ex>[r] & {\bK(\rmE\rmG^{gm, m}, \rmG)  {\overset L {\wedge}} \bG(\rm{\rm X}, \rmG)} \ar@<1ex>[r] & \bG(\rmE\rmG^{gm,m} \times \rm{\rm X}, \rmG)}
\end{equation} \ee
\vskip .2cm \noindent
which shows the composite map from the top left-corner to the bottom-right corner is null-homotopic. This proves the first statement since we have let $n=m'\ge m+1$.  
\vskip .2cm
Next we consider the second statement as well as the functoriality of the factorization in (i). Recall (see \cite[Corollary 6.7]{C08}) that the partial derived 
completion $\bG(\rm{\rm X}, \rmG)\compl_{\rho_{\rmG}, m'} =$ 
the homotopy cofiber of the map
\[{\overset {m'+1} {\overbrace {\tilde I_{\rmG} {\underset {\bK(Spec \, k, \rmG)} \wedge} \tilde I_{\rmG}  \cdots {\underset {\bK(Spec \, k, \rmG)} \wedge} \tilde I_G}}} 
 {\underset {\bK(Spec \, k , \rmG)} \wedge} {\widetilde {\bG(\rm{\rm X}, \rmG)}} \ra 
 I_G^{\wedge^{m'+1}} {\underset {\bK(Spec \, k , \rmG)} \wedge} {\widetilde {\bG({\rm X}, \rmG)}} \]
\[ \ra \bK(Spec \, k, \rmG) {\underset {\bK(Spec \, k , \rmG)} \wedge} {\widetilde {\bG(\rm{\rm X}, \rmG)}} = {\widetilde {\bG(\rm{\rm X}, \rmG)}}\]
which maps into  $(\bK(Spec \, k, \rmG)/I_{\rmG}^{\wedge^{m'+1}} ){\underset {\bK(Spec \, k , \rmG)} \wedge} {\widetilde {\bG(\rm{\rm X}, \rmG)}}$.
Here $\tilde I_{\rmG} \ra I_{\rmG}$ (${\widetilde {\bG(\rm{\rm X}, \rmG)}} \ra \bG(\rm{\rm X}, \rmG)$) is a cofibrant replacement in the category of $\bK(Spec \, k, \rmG)$-module spectra. This proves 
the first statement in (ii).  The
functoriality of the factorization in (i) and (ii) follows readily by observing that 
both $\pi^*$ (when $\pi$ is flat) and $\pi_*$ (when $\pi$ is proper of finite cohomological dimension)
are module maps over $\bK(Spec \, k, \rmG)$. Therefore, these maps will induce maps between the diagrams in ~\eqref{null.homotopic} corresponding to
$\rmX$ and $\rmY$. 
\vskip .2cm
To prove the last statement, one first observes that $\bG(\rm{\rm X}, \rmG_1 \times \cdots \times \rmG_q)$ is a module
spectrum over $\bK(Spec \, k, \rmG_1 \times \cdots \times \rmG_q)$. Now the homotopy commutative square
 (where the vertical maps are the obvious ones)
\vskip .2cm
\xymatrix{{\bK(Spec \, k, \Pi_{i=1}^q\rmG_i ) \wedge \bG(\rm{\rm X}, \Pi_{i=1}^q \rmG_i)} \ar@<1ex>[r] \ar@<1ex>[d] & {\bG(\rm{\rm X}, \Pi_{i=1}^q \rmG_i )} \ar@<1ex>[d]\\
{\bK(\rmE\rmG_1^{gm, m} \times \cdots \rmE\rmG_n^{gm, m}, \Pi_{i=1}^q\rmG_i )  \wedge \bG(\rm{\rm X}, \Pi_{i=1}^q\rmG_i )} \ar@<1ex>[r] & \bG(\rmE\rmG_1^{gm,m} \times \cdots \rmE\rmG_n^{gm, m} \times \rm{\rm X}, \Pi_{i=1}^q\rmG_i)}
\vskip .2cm \noindent
and an argument exactly as in the case of a single group, proves the factorization in the last statement. The functoriality of this factorization in $\rmX$ may be proven as in the case
of a single group.
\end{proof}

\vskip .2cm 
\begin{definition}
 \label{funct.alpha} For each linear algebraic group $\rmG$ acting on a scheme $\rmX$ satisfying the basic hypotheses as in ~\ref{stand.hyp.1} and each
choice of admissible gadgets as in ~\ref{adm.gadg}, we will
define a function $\alpha:\N \ra \N$, by $\alpha(m) =m'$ where $m' \ge m+1$ is some choice of $m'$ as in the last Proposition.
\end{definition}
\vskip .2cm \noindent

The main result of this section is the following theorem:
\begin{theorem}
\label{main.thm.3}
Let $\rmG$ denote a finite product of general linear groups acting on a scheme $\rmX$ of finite type over 
a field $k$ satisfying the standing hypotheses ~\ref{stand.hyp.1}. 
 Then the map $\{\bG(\rm{\rm X}, \rmG)_{\rho_{\rmG}, \alpha(m)}|m\} \ra  \{\bG(E{\rmG}^{gm,m}{\underset {\rmG} \times}\rmX)|m\} $ of pro-spectra induces a weak-equivalence on
taking the homotopy inverse limit as $m \ra \infty$. The same conclusion holds for any algebraic space of $\rmX$ of finite type over $k$, provided the
group $\rmG$ is a split torus and $\{\rmE{\rmG}^{gm,m}|m\}$ is chosen as in ~\eqref{toric.case}.
\end{theorem}
\vskip .2cm
The rest of this section will be devoted to a proof of the above theorem. The main idea of the proof is as follows.
Observe that the maps $\pi^*$ and $\pi_*$ associated to $\pi: \rmG{\underset {\rmB} \times}X \ra X$ induce maps that make the following
diagram homotopy commute, for each fixed $m$:
\be \begin{equation}
     \label{splittings}
\xymatrix{{\bG(\rm{\rm X}, \rmG)_{\rho_{\rmG}, \alpha(m)}} \ar@<1ex>[d]^{\pi^*} \ar@<1ex>[rr] && {\bG(\rmE{\rmG}^{gm,m}{\underset {\rmG} \times}\rmX)} \ar@<1ex>[d]^{\pi^*} \\
{\bG(\rm{\rm X}, \rmB)_{\rho_{\rmG}, \alpha(m)} \simeq \bG(\rmG{\underset {\rmB} \times}\rm{\rm X}, \rmG)_{\rho_{\rmG}, \alpha(m)}} \ar@<1ex>[d]^{\pi_*} \ar@<1ex>[rr] && {\bG(\rmE{\rmG}^{gm,m}{\underset {\rmB} \times}\rmX) \simeq 
\bG(\rmE\rmG^{gm, m}{\underset {\rmG} \times}(\rmG{\underset {\rmB} \times}\rmX))} \ar@<1ex>[d]^{\pi_*}\\
{\bG(\rm{\rm X}, \rmG)_{\rho_{\rmG}, \alpha(m)}}  \ar@<1ex>[rr] && {\bG(\rmE{\rmG}^{gm, m}{\underset {\rmG} \times}\rmX)}}.
\end{equation} \ee
\vskip .2cm \noindent
On varying $m$, we 
will view the above diagram as a level diagram of  pro-spectra. Observe that we may replace the Borel subgroup ${\rmB}$ everywhere by its maximal torus 
${\rmT}$ since ${\rmB}/{\rmT}$ is an affine space and hence ${\mathbb A}^1$-acylic.
\vskip .2cm
We will next show how to replace the middle-map up to weak-equivalence by a corresponding map
when the derived completions with respect to $\rho_{\rmG}$ are replaced by derived completions with respect to 
$\rho_T$. We will then show that the middle
map induces a weak-equivalence on taking the homotopy inverse limit as $m \ra \infty$. Since the top horizontal map (which is also the bottom horizontal map) 
is a retract of the middle horizontal map, it follows
that the top horizontal map also induces a weak-equivalence on taking the homotopy inverse limit as $m \ra \infty$. This will prove the theorem.
\vskip .2cm
Therefore, the rest of the proof will be to show that the middle horizontal map induces a weak-equivalence on 
taking the homotopy inverse limit as $m \ra \infty$. Next we invoke Lemma ~\ref{res.1} with $A= \bK(Spec \, k, \rmG)$,
$B= \bK(Spec \, k, \rmT)$ and $C = \bK(Spec \, k)$: in fact we will assume that we have already replaced $C$ with $\tilde C$
and $B$ with $\tilde B$ following the terminology of Lemma ~\ref{res.1}. Let $I_{\rmG}$ denote the homotopy fiber of 
the composite restriction $\bK(Spec \, k, \rmG) \ra \bK(Spec \, k, \rmT) \ra \bK(Spec \, k)$ and let $I_{\rmT}$ denote the
homotopy fiber of the restriction $\bK(Spec \, k, \rmT) \ra \bK(Spec \, k)$. For each integer $m \ge 1$, we will let
\be \begin{equation}
\label{IG.IT.powers}
 I_{\rmG} ^{L\wedge ^m}=  {\overset {m} {\overbrace { I_{\rmG} {\underset {\bK(Spec \, k, \rmG)} {\overset L {\wedge}}} I_{\rmG}  \cdots {\underset {\bK(Spec \, k, \rmG)} {\overset L {\wedge}}} \ I_{\rmG} }}} \mbox{ and }  I_{\rmT} ^{L\wedge ^m}=  {\overset {m} {\overbrace { I_{\rmT} {\underset {\bK(Spec \, k, \rmT)} {\overset L {\wedge}}} I_{\rmT}  \cdots {\underset {\bK(Spec \, k, \rmT)} {\overset L {\wedge}}} \ I_{\rmT} }}} 
    \end{equation}\ee
Now Lemma ~\ref{res.1} and ~\eqref{key.cof.seq} provides the following inverse system of commutative squares:
\be \begin{equation}
\label{K.diagm.0}
\xymatrix{ {\{\bG(\rm{\rm X}, \rmT)\compl_{\rho_{\rmG}, \alpha(m)}|m\}} \ar@<1ex>[d] & {\{\bG(\rm{\rm X}, \rmT)/{I_{\rmG}^{L\wedge^{\alpha(m)+1}}}|m\}} \ar@<-1ex>[l]^{\simeq} \ar@<1ex>[d] \\
{\{\bG(\rm{\rm X}, \rmT)\compl_{\rho_{\rmT}, \alpha(m)}|m\}}  & {\{\bG(\rm{\rm X}, \rmT)/{I_{\rmT}^{L\wedge^{\alpha(m)+1}}}|m\}}  \ar@<-1ex>[l]^{\simeq}  }
\end{equation} \ee
\vskip .2cm
Making use of Proposition ~\ref{factoring.through.part.compl}(i) and making a suitable choice of the function $\alpha$ 
 following the terminology in Definition ~\ref{funct.alpha}, so that the map
$\bG(\rm{\rm X}, \rmT) \ra \bG(E\rmG^{gm, n}   \times \rm{\rm X}, \rmT)$ factors through $\bG(\rm{\rm X}, \rmT)/{I_{\rmG}^{L\wedge^{\alpha(n)+1}}}   $ and also
the map $\bG(\rm{\rm X}, \rmT) \ra \bG(E\rmT^{gm, n}   \times \rm{\rm X}, \rmT)$ factors through $\bG(\rm{\rm X}, \rmT)/{I_{\rmT}^{L\wedge^{\alpha(n)+1}}}$,
the above
diagrams now provide the inverse system of diagrams:
\be \begin{equation}
\label{K.diagm.1}
\xymatrix{ {\{\bG(\rm{\rm X}, \rmT)\compl_{\rho_{\rmG}, \alpha(n)}|n\}} \ar@<1ex>[dd] & {\{\bG(\rm{\rm X}, \rmT)/{I_{\rmG}^{L\wedge^{\alpha(n)+1}}}|n\}} \ar@<1ex>[dd] \ar@<-1ex>[l]^{\simeq} \ar@<1ex>[r] & {\{\bG(\rmE\rmG^{gm, n}   \times \rm{\rm X}, \rmT)|n\}} \ar@<1ex>[d]\\
& & {\{\bG(\rmE\rmG^{gm, n}  \times E\rmT^{gm, n} \times \rm{\rm X}, \rmT)|n\}}  \\
{\{\bG(\rm{\rm X}, \rmT)\compl_{\rho_{\rmT}, \alpha(n)}|n\}}  & {\{\bG(\rm{\rm X}, \rmT)/{I_{\rmT}^{L\wedge^{\alpha(n)+1}}}|n\}} \ar@<1ex>[r] \ar@<-1ex>[l]^{\simeq}&{\{ \bG(\rmE\rmT^{gm, n}\times \rm{\rm X}, \rmT)|n\}} \ar@<-1ex>[u] }
\end{equation} \ee
\vskip .2cm \noindent
Here $\bG(\rm{\rm X}, \rmT)/I_{\rmG}^{L\wedge ^{\alpha(n)+1}}$ denotes the homotopy cofiber of the map 
$I_{\rmG}^{L\wedge ^{\alpha(n)+1}} {\overset L{\underset {\bK(Spec \, k, G)} \wedge}} \bG(\rm{\rm X}, \rmT) \ra \bG(\rm{\rm X}, \rmT)$ and
\newline \noindent 
$\bG(\rm{\rm X}, \rmT)/I_{\rmT}^{L\wedge ^{\alpha(n)+1}}$ denotes the homotopy cofiber of the map $I_{\rmT}^{L\wedge ^{\alpha(n)+1}} {\overset L{\underset {\bK(Spec \, k, T)} \wedge}} \bG(\rm{\rm X}, \rmT) \ra \bG(\rm{\rm X}, \rmT)$.
The map  
\xymatrix{{\{\bG(\rm{\rm X}, \rmT)/{I_{\rmG}^{L\wedge^{\alpha(n)+1}}}|n\}}  \ar@<1ex>[r] & {\{\bG(\rmE\rmG^{gm, n}   \times \rm{\rm X}, \rmT)|n\}}}
exists because of the identifications: $\bG(\rm{\rm X}, \rmT) = \bG(\rmG{\underset {\rmT} \times}{\rm X}, \rmG)$ and
$\bG(\rmE\rmG^{gm, n}   \times \rm{\rm X}, \rmT) = \bG(\rmE\rmG^{gm,n}{\underset {} \times}({\rmG}{\underset {\rmT} \times}X), \rmG)$.
\vskip .2cm
Recall from Proposition ~\ref{coveringlemma} that there
are two distinct models of the direct system of geometric classifying spaces for split maximal tori $\rmT$ in $\GL_n$.
The product of the universal $\rmT$-bundles of these two models is yet another model for the universal bundle
over a direct system of geometric classifying space for $\rmT$, which
 shows up  in ${\{\bG(\rmE\rmG^{gm, n}  \times \rmE\rmT^{gm, n} \times \rm{\rm X}, \rmT)|n\}}$. 
 The projections $\rmE\rmG^{gm, n}   \times \rmX \leftarrow  \rmE\rmG^{gm, n}  \times \rmE\rmT^{gm, n} \times \rmX \ra  \rmE\rmT^{gm,n} \times \rmX$
are flat which provide the right-most two vertical maps.
\vskip .2cm
The left most square is simply the commutative square in ~\eqref{K.diagm.0}, which therefore commutes. To see the commutativity of
the right square, one may simply observe that the two maps are two different factorizations of the map
corresponding to the functor that pulls-back a $\rmT$-equivariant coherent sheaf on $\rmX$ to 
a $\rmT$-equivariant coherent sheaf on $\rmE\rmG^{gm, n}  \times \rmE\rmT^{gm, n} \times \rmX$. 
The left-vertical map clearly 
is a weak-equivalence on taking the homotopy inverse limit $n \ra \infty$ by Theorem ~\ref{T.comp.vs.G.comp}. The bottom horizontal map in the right square 
 is provided by 
Propositions ~\ref{coveringlemma} and  ~\ref{factoring.through.part.compl} and it
induces a weak-equivalence on taking the homotopy inverse limit as $n \ra \infty$  as shown in Theorem ~\ref{key.thm.1}.
Therefore, modulo Theorem ~\ref{key.thm.1} and the following proposition, this completes the proof of Theorem ~\ref{main.thm.3}. \qed
\vskip .2cm
 
\begin{proposition}
 \label{derived.compl.pro.1} Assume $\rmX$ is a smooth scheme. Then, assuming Theorem  ~\ref{key.thm.1},
the map in the top row in the diagram ~\ref{K.diagm.1} induces a weak-equivalence on taking the homotopy 
inverse limit as $n \ra \infty$. It follows that the middle row in the diagram ~\eqref{splittings} induces a weak-equivalence
on taking the homotopy inverse limit as $n \ra \infty$. 
\end{proposition}
\begin{proof} The vertical maps on the right in the diagram ~\ref{K.diagm.1} induce a weak-equivalence on taking the
homotopy inverse limit as $n \ra \infty$ by the properties of the geometric classifying spaces: 
see appendix B. In view of what we already observed, we now obtain the homotopy commutative diagram:
\be \begin{equation}
\xymatrix{{\bG(\rm{\rm X}, \rmT)\compl_{\rho_{\rmG}}  } \ar@<1ex>[dd]^{\simeq} \ar@<1ex>[r]& {\holimn\{\bG(\rmE\rmG^{gm, n}  \times \rm{\rm X}, \rmT)|n\}} \ar@<1ex>[d]^{\simeq}\\
& {\holimn\{\bG(\rmE\rmG^{gm, n}  \times \rmE\rmT^{gm,n}\times \rm{\rm X}, \rmT)|n\}} \\
{\bG(\rm{\rm X}, \rmT)\compl_{\rho_{\rmT}}} \ar@<1ex>[r]^{\simeq} & {\holimn \{ \bG(\rmE\rmT^{gm, n} \times \rm{\rm X}, \rmT)|n\}} \ar@<-1ex>[u]_{\simeq}}
\end{equation} \ee
\vskip .2cm
 Since all other maps in the above diagram are
now weak-equivalences except possibly for the map in the top row, it follows that is also a
weak-equivalence. This proves the first statement. The second statement is clear.
\end{proof}

\section{\bf Proof in the case of a split torus}
In view of the above reductions, it suffices to assume that the group scheme $\rmH=\rmT= {\mathbb G}_m^n$= a split torus. In case the group scheme
is a smooth diagonalizable group-scheme, we may imbed that as a closed subgroup scheme of a split torus. Then, making use of 
the remarks following Theorem ~\ref{comparison} and the remarks in ~\ref{remark.validity}, one may extend the results in this section to
actions of smooth diagonalizable group schemes also. However, we will not discuss this case explicitly.
Thus assuming $\rmH=\rmT= {\mathbb G}_m^n$= a split torus, a key step is provided
by the following result. 
\vskip .2cm
First assume that $\rmX$ and $\rmY$ are schemes or algebraic spaces of finite type and flat over the base-scheme $\rmS$ (note: the flatness hypothesis  is automatically satisfied when $\rmS = Spec \, k$, with $k$  a field).
Assume $\rmX$ and $\rmY$ are acted on by an affine group-scheme $\rmH$ over $S$.
In this case the external tensor product of coherent $\O$-modules induces a pairing $\bG(\rm{\rm X}, \rmH) \wedge \bG(\rmY, \rmH) \ra \bG(\rmX {\underset {\rmS} \times}\rmY, \rmH)$
where G-theory denotes the Quillen K-theory spectra of the exact category of coherent $\O$-modules. The action of $\rmH$ on $\rmX{\underset {\rmS} \times}\rmY$ 
is the diagonal action. This pairing is compatible with 
the structure of the above spectra as module-spectra over the ring spectrum $\bK(\rmS, \rmH)$ so that one obtains the induced pairing:
\vskip .2cm
$p_1^* \wedge p_2^*: \bG(\rm{\rm X}, \rmH) {\overset L {\underset {\bK(\rmS, \rmH)} \wedge}} \bG(\rmY, \rmH) \ra \bG(\rmX {\underset {\rmS} \times}\rmY, \rmH)$.
\vskip .2cm \noindent
One obtains a similar pairing where the Quillen-style G-theory and K-theory above are replaced by the corresponding Waldhausen-style
theories. Once again, the flatness hypothesis on $\rmX$ and $\rmY$ over $\rmS$ ensures that the external product preserves cofibrations and
weak-equivalences after fixing one of the arguments. In addition one may check that if $K' \ra K$ is a cofibration of $\O_{\rmX}$-modules
and $M' \ra M$ is a cofibration of $\O_{\rmY}$-modules, then the induced map 
$K{\underset {\O_{\rmS}} \otimes}M' {\underset {K' {\underset {\O_{\rmS}} \otimes}M'} \oplus} K' {\underset {\O_{\rmS}} \otimes}M \ra K {\underset {\O_{\rmS}} \otimes} M$
is a cofibration which is a quasi-isomorphism if either of the maps $K' \ra K$ or $M' \ra M$ are quasi-isomorphisms.
\begin{lemma}  (See \cite[Proposition 3.1]{JK}.)
\label{Kunneth.formula}
Let $\rm{\rmX}$ denote cellular scheme,i.e. a scheme stratified by locally closed smooth subschemes all of which are affine spaces over the base scheme 
$\rmS$ which will be assumed to be a field.
Assume that $\rmH$ is an affine group scheme acting on $\rmX$ so that the strata of $\rmX$ are $\rmH$-stable. Let $\rmY$ denote any scheme or algebraic 
space of finite type
 over $\rmS$ also 
provided with an $\rmH$-action.  Then the
 map (induced by the above external product)
\be \begin{equation}
\label{prdct.pairing.1}
p_1^* \wedge p_2^*:\bG(\rm{\rm X}, \rmH) {\overset L {\underset {\bK(\rmS, \rmH)} \wedge}} \bG(\rmY, \rmH) \ra \bG(\rmX{\underset {\rmS} \times} \rmY, \rmH)
\end{equation} \ee
\vskip .2cm \noindent
is a weak-equivalence of spectra. Here $p_1: \rmX{\underset {\rmS} \times} \rmY \ra \rmX$, $p_2:\rmX{\underset {\rmS} \times} \rmY \ra \rmY$ is the projection to the
$i$-th factor. A corresponding result holds when the G-theory spectrum is replaced everywhere by the  mod-$\ell$ G-theory spectrum defined as 
in ~\eqref{KGl} and where $\ell$ is a prime different from the residue characteristics.
When $\rmX$ and $\rmY$ are  regular, we may replace the G-theory spectrum by the K-theory spectrum everywhere.
\end{lemma}
\begin{proof} This is essentially \cite[Proposition 3.1]{JK} where it is stated in terms of the K-theory and G-theory of quotient stacks. 
One considers the localization sequences in $\rmH$-equivariant G-theory defined using the given strata on $\rmX$ and the induced strata
 on $\rmX \times \rmY$. The map $p_1^* \wedge p_2^*$ is compatible with these localization sequences. Therefore, one reduces to the case where 
$\rmX$ is an affine space over $\rmS$, in which case the conclusion follows by the homotopy property of $\rmH$-equivariant G-theory. 
Clearly the homotopy property for equivariant G-theory extends to the homotopy property for the mod-$\ell$ equivariant G-theory 
defined as in ~\eqref{KGl}
proving the second statement. When $\rmX$ and $\rmY$ are regular, the equivariant G-theory spectrum is weakly equivalent to
 the corresponding equivariant K-theory spectrum as observed earlier, thereby proving the last statement.
\end{proof}
\begin{proposition}
\label{key.prop}
 Assume that $\rmH=T =\GG_m^n =$ a split torus and let $\rmY$ denote any scheme or algebraic space of finite type over $k$ provided with an 
$\rmH$-action. In this case the map (defined as in ~\eqref{prdct.pairing.1})
\vskip .2cm
$  \bK(\rmE\rmH^{gm,m}, \rmH) \wedgeKH \bG(\rmY, \rmH) \ra \bG(\rmE\rmH^{gm,m} \times \rmY, \rmH) \simeq \bG(\rmE\rmH^{gm,m}{\underset {\rmH}  \times}\rmY)$
\vskip .2cm \noindent
is a weak-equivalence for all positive integers $m$, with the product $E\rmH^{gm,m} \times \rmY$ denoting the fibered 
product over the base-scheme $\rmS$, which is once again assumed to be a field $k$. Therefore, the induced map
\vskip .2cm
$\holimm \{\bK(\rmE\rmH^{gm,m}, \rmH) \wedgeKH \bG(\rmY, \rmH)|m\} \ra \holimm \{\bG(E\rmH^{gm,m} \times \rmY, \rmH)|m\}
 \simeq \holimm \{\bG(\rmE\rmH^{gm,m}{\underset {\rmH}  \times}\rmY)|m\}$
\vskip .2cm \noindent
is also a weak-equivalence. The corresponding statements also hold for the G-theory and K-theory spectra
replaced everywhere by the corresponding mod-$\ell$ spectra defined as in ~\eqref{KGl}.
The corresponding assertions also  hold
with the G-theory spectrum replaced by the K-theory spectrum when $\rmY$ is regular.
\end{proposition}
\begin{proof} We will only consider the case of G-theory and K-theory spectra, since the extensions to the
mod-$\ell$  versions are obvious. A point to observe is that   $\rmEH^{gm,m}$ is not a cellular scheme, so that the last lemma does
not apply directly.
We first recall that if $\rmT=\GG_m^n$ is a split torus, then $\rmE\rmT^{gm,m}=({\mathbb A}^{m+1} -0)^n$ and $\rmB\rmT^{gm,m} = \rmE\rmT^{gm,m}/T =
(({\mathbb A}^{m+1} -0)/{\GG_m})^n = ({\mathbb P}^m)^n$. Therefore the first statement in the proposition says the obvious map
\vskip .2cm
$\bG(({\mathbb A}^{m+1} -0)^n, \GG_m^n) {\overset L {\underset {\bK(\rmS, \GG_m^n)} \wedge}} \bG(\rmY, \GG_m^n) \ra 
\bG(({\mathbb A}^{m+1}-0)^n \times \rmY, {\GG_m^n}) \simeq 
\bG(({\mathbb A}^{m+1}-0)^n{\underset {\GG_m^n} \times} \rmY)$
\vskip .2cm \noindent
is a weak-equivalence. The last weak-equivalence follows from the observation that ${\mathbb G}_m^n$ acts freely
on the product $({\mathbb A}^{m+1}-0)^n$, with each factor of ${\mathbb G}_m $ acting on the corresponding factor of 
${\mathbb A}^{m+1}-0$. Now it remains to prove the first map is a weak-equivalence.
\vskip .2cm
For the proof, it will be convenient to denote the rank of the split torus as $n$ and the
$n$ in $({\mathbb A}^{m+1}-0)^n$ as $u$. For $u <n$, the scheme $({\mathbb A}^{m+1}-0)^u$ will be identified
as the closed subscheme $({\mathbb A}^{m+1}-0)^u \times {0}^{n-u}$ in  $({\mathbb A}^{m+1}-0)^u \times ({\mathbb A}^{m+1})^{n-u}$ with the 
 last $n-u$ factors being the origin of 
${\mathbb A}^{m+1}$ and the
remaining factors being $({\mathbb A}^{m+1}-0)$. In this case, the 
${\mathbb G}_m$ forming the last $n-u$ factors  will clearly act trivially on the corresponding factors which are 
identified with the origin in ${\mathbb A}^{m+1}$. Therefore, making use of the localization sequences in 
$\rmG$-equivariant G-theory, one may prove this by ascending 
induction on $u$. Assume that the above map is a weak-equivalence for a fixed $u <n$, $0 \le u$; we will then show that it is also a weak-equivalence for $u$ replaced by $u+1$
as follows. Since both the source and the target of the above map are compatible with localization sequences, we obtain the following commutative diagram with
 each row a  stable cofiber sequence:
\vskip .2cm
\xymatrix{{\bG(({\mathbb A}^{m+1} -0)^u \times {0}^{n-u}, \GG_m^n) {\overset L {\underset {\bK(\rmS, \GG_m^n)} \wedge}} \bG(\rmY, \GG_m^n)} \ar@<1ex>[r] \ar@<-1ex>[d] & 
{\bG(({\mathbb A}^{m+1} -0)^u \times {\mathbb A}^{m+1} \times {0}^{n-u-1}, \GG_m^n) {\overset L {\underset {\bK(\rmS, \GG_m^n)} \wedge}} \bG(\rmY, \GG_m^n)}  \ar@<-1ex>[d]  \\
{\bG(({\mathbb A}^{m+1}-0)^u \times {0}^{n-u} \times Y, {\GG_m^n})} \ar@<1ex>[r] & 
{\bG(({\mathbb A}^{m+1}-0)^u \times {\mathbb A}^{m+1} \times {0}^{n-u-1} \times Y, {\GG_m^n})} }
\vskip .2cm
\xymatrix{ {}  \ar@<1ex>[r] & {\bG(({\mathbb A}^{m+1} -0)^{u+1} \times {0}^{n-u-1}, \GG_m^n) {\overset L {\underset {\bK(\rmS, \GG_m^n)} \wedge}} \bG(\rmY, \GG_m^n)} \ar@<-1ex>[d] \\
{} \ar@<1ex>[r]  & {\bG(({\mathbb A}^{m+1}-0)^{u+1}  \times {0}^{n-u-1} \times Y, {\GG_m^n})}}
\vskip .2cm \noindent
The first vertical map is a weak-equivalence by the inductive assumption and the second vertical map is a weak-equivalence making use of this inductive assumption and the homotopy property for
equivariant G-theory. Therefore, so is the last vertical maps which completes the induction.
This reduces everything to the case
where $u=0$, which is clear. When $\rmY$ is regular, the weak-equivalences $\bK(\rmY, \GG_m^n) \simeq \bG(\rmY, \GG_m^n)$,
$\bK(({\mathbb A}^{m+1} -0)^n, \GG_m^n) \simeq \bG({\mathbb A}^{m+1} -0)^n, \GG_m^n)$, 
$\bK(({\mathbb A}^{m+1}-0)^n{\underset {} \times} \rmY, \GG_m^n) \simeq \bG(({\mathbb A}^{m+1}-0)^n{\underset {} \times} \rmY, \GG_m^n)$ and
$\bK(({\mathbb A}^{m+1}-0)^n{\underset {\GG_m^n} \times} \rmY) \simeq \bG(({\mathbb A}^{m+1}-0)^n{\underset {\GG_m^n} \times} \rmY)$ complete the
proof. \end{proof}
\vskip .2cm
\subsubsection{Multiple derived completions}
\label{mult.der.compl}
Let $R$ denote a commutative ring spectrum and let $Alg(R)$ denote the category of commutative $R$-algebra spectra. We will provide $Alg(R)$ with the cofibrantly generated model category structure
discussed in \cite{Ship04}. Letting $n$ denote a fixed positive integer and let $B_i, C_i\eps Alg(R)$, $i=1, \cdots, n$ be 
provided with
 maps of commutative algebra spectra $f_i:B_i \ra C_i$ in $Alg(R)$. We will also assume that one is provided with augmentations
$C_i \ra R$ that are maps in $Alg(R)$. (In the examples that we consider below, each $C_i=R$, so this last condition is automatic.)
We will assume the $B_i, C_i$  $i=1, \cdots, n$ are
cofibrant in $Alg(R)$, so that each map $B_i \ra C_i$ is a cofibration in $Alg(R)$. 
We will let $B=B_1 \wedge_R B_2 \wedge_R \cdots \wedge _R B_n$ and 
$C=C_1 \wedge_R C_2 \wedge_R \cdots \wedge_R C_n$.  Since each $B_i \ra C_i$ is a cofibration, it follows that the induced map $B \ra C$ is 
a cofibration  in $Alg(R)$.
 Then one first observes that  these are commutative algebra spectra over $R$ and that $f=f_1\wedge_R f_2 \wedge_R \wedge_R \cdots \wedge_R f_n$
is a map of such algebra spectra.
We proceed to  consider derived completions of a module spectrum
$M \eps Mod(B)$ with respect to the map 
$f$  and explore how it relates to the derived completions with respect to the maps
$f_i$, $i=1, \cdots, n$.
\vskip .2cm
It needs to pointed out that there is an entirely parallel set-up where $R$ denotes a commutative ring and $B_i, C_i$, $i=1, \cdots, n$ are
commutative algebras over $R$ (or commutative dg-algebras over $R$) which will all be assumed to be flat over $R$. In this set-up 
$M$ will be a module (or dg-module) over $B$.
\vskip .2cm
The {\it main example} to keep in mind is what appears in Proposition ~\ref{key.prop}, namely $R= \bK(\rmS)$,
$B_i = \bK(\rmS, {\mathbb G}_m)$  with the ${\mathbb G}_m$ forming the $i$-th factor in $\rmT={\mathbb G}_m^n$ and
$C_i$ being suitable cofibrant replacements of $\bK(\rmS)$  for all $i=1, \cdots , n$. Then $\bK(\rmS, \rmT) $ identifies with $\bK(\rmS, {\mathbb G}_m)\wedge_{\bK(\rmS)} K(\rmS, {\mathbb G}_m) \cdots \wedge _{\bK(\rmS)} \bK(\rmS, {\mathbb G}_m)$ so that
we will let $B= \bK(\rmS, \rmT)$.
\vskip .2cm
 For each $i=1, \cdots, n$, let $T_{B_i}:Mod(B) \ra Mod(B)$ denote the triple defined by $M \mapsto M\wedge_{B_i}C_i$
which is $M \wedge _{B_i}C_i$ viewed as a $B$-module through the obvious $B$-bi-module  structure on $M$. Clearly $T_{B_i}(M)$ has a (right) $C_i$-module
structure. Now one obtains an induced $B_i$-module structure on $T_{B_i}(M)$ obtained by restriction of scalars along the map $B_i \ra C_i$. 
Observe that this $B_i$-module
structure is compatible with  the $B$-module structure on $T_{B_i}(M)$ considered earlier (and used in the definition of the functor $T_{B_i}$.)
We let ${\mathcal T}_{B_i}^{\bullet}$ denote the cosimplicial object obtained by iterating $T_{B_i}$ and denoted
${\mathcal T}_{B_i}^{\bullet}(M, C_i)$ elsewhere. Observe that  it is also possible
to compose $T_{B_i}$ and $T_{B_j}$ and thereby define a triple $T_{B_1, \cdots, B_n}:Mod(B) \ra Mod(B)$ by $M \mapsto
M\wedge_{B_1}C_1 \wedge _{B_2} C_2 \cdots \wedge _{B_n} C_n$, i.e.  
\[T_{B_1, \cdots, B_n} = T_{B_n} \circ \cdots \circ T_{B_1} \]
\vskip .2cm
Observe $T_{B_1, \cdots, B_n}(M)$  has the obvious structure of a (right) $C=C_1 \wedge_R \cdots \wedge _R C_n$-module and that the $B$-module structure
on it obtained by restriction of scalars along the map $B_1 \wedge _R \cdots \wedge _R B_n \ra C_1 \wedge _R \cdots \wedge _R C_n$ is compatible
with the $B$-module structure on $T_{B_1, \cdots, B_n}(M)$ used in its definition.
 i.e. If $m \eps M, b_i \eps B_i, c_i \eps C_i$ and $r_i \eps R$, $i=1, \cdots , n$,
then $m\wedge r_1b_1c_1 \wedge r_2b_2c_2 \wedge \cdots \wedge r_n b_nc_n $ is identified with
$mr_1\cdots r_nb_1\cdots b_n \wedge c_1 \wedge c_2 \wedge \cdots \wedge c_n$. Therefore, one obtains the identifications for any  $M \eps Mod(B)$:
\vskip .2cm
$T_{B_n}( T_{B_{n-1}} \cdots (T_{B_1}(M)))=T_{B_1, \cdots, B_n}(M) = T_{B_1 \wedge_R \cdots \wedge_R B_n}(M) = 
M {\underset {B_1 \wedge_R \cdots \wedge_R B_n} \wedge}C_1 \wedge_R \cdots \wedge_R C_n$,
\vskip .2cm \noindent
where the $B_i$-module structure on $T_{B_{i-1}} \cdots  T_{B_1}(M)$ is from the original $B$-module
structure on $M$. Moreover if $M= M_1 \wedge_R \cdots \wedge_R M_n$, 
then $T_{B_1, \cdots, B_n}(M) = T_{B_1}(M_1) \wedge_R T_{B_2}(M_2) \wedge _R \cdots \wedge_R T_{B_n}(M_n)$. 
\vskip .2cm
One may iterate the above constructions to obtain an $n$-fold multi-cosimplicial object:
${\mathcal T}_{B_1, \cdots, B_n}^{\bullet, \cdots, \bullet}(M)$ defined by
\vskip .2cm
${\mathcal T}_{B_1, \cdots, B_n}^{m_1, \cdots, m_n}(M) = {\mathcal T}_{B_n}^{m_n}( \cdots {\mathcal T}_{B_1}^{m_1}(M))$
\vskip .2cm \noindent
where an $M \eps Mod(B)$ is viewed as an object in $Mod(B_1)$ for forming ${\mathcal T}_{B_1}^{m_1}(M)$ and
each ${\mathcal T}_{B_{i-1}}^{m_{i-1}}(\cdots {\mathcal T}_{B_1}^{m_1}(M))$ is given the $B_i$-module structure induced from
 the original $B$-module
structure on $M$ before applying ${\mathcal T}_{B_{i}}^{m_i}$. For a sequence of non-negative integers $m_1, \cdots, m_n$, we define
$\sigma_{\le m_1, \cdots, m_n}{\mathcal T}_{B_1, \cdots, B_n}^{\bullet, \cdots, \bullet}(M) = \{
{\mathcal T}_{B_1, \cdots, B_n}^{i_1, \cdots, i_n}(M)| i_1 \le m_1, \cdots. i_n \le m_n\}$ which is a 
truncated multi-cosimplicial object. Then we let
\be \begin{equation}
 \label{M.f.1}
M \compl_{f, m} = \holimD (\sigma_{\le m}(
\Delta {\mathcal T}_{B_1, \cdots, B_n}^{\bullet, \cdots, \bullet}( M))) \mbox{ where } \Delta \mbox{ denotes the diagonal and }
\end{equation} \ee
\be \begin{equation}
     \label{M.f.2}
M \compl_{f, m_1, \cdots, m_n} = \hlimDone \cdots \hlimDn(\sigma_{\le m_1, \cdots, \le m_n}({\mathcal T}_{B_1, \cdots, B_n}^{\bullet, \cdots, \bullet}( M))). 
    \end{equation} \ee
\vskip .2cm \noindent
This denotes the partial derived completion up to degrees $m_1, \cdots, m_n$.
\begin{proposition}
 \label{multiple.derived.compl.thm}
Assume the above situation. Then the following hold for an $M \eps Mod(B)$.
\vskip .2cm
(i) For any sequence of positive integers $m_1, \cdots, m_n$,
\vskip .2cm \noindent
$M \compl_{f, m_1, \cdots, m_n} = \hlimDone \cdots \hlimDn
(\sigma_{\le m_1, \cdots, \le m_n}({\mathcal T}_{B_1, \cdots, B_n}^{\bullet, \cdots, \bullet}( M)))$
\vskip .2cm \noindent
$\simeq
\hlimDn \sigma_{\le m_n}{\mathcal T}^{\bullet}_{B_n}( \hlimD2 \sigma_{\le m_{n-1}}{\mathcal T}^{\bullet}_{B_{n-1}}( \cdots (
\hlimDone \sigma_{\le m_1} {\mathcal T}_{B_1}^{\bullet}(M))))$
\vskip .2cm \noindent
(ii) $M\compl_f = \holim M\compl_{f,m} \simeq {\underset {{\infty, \cdots , \infty} \leftarrow (m_1, \cdots, m_n)} \holim} M \compl_{f, m_1, \cdots, m_n}$ which
denotes the derived completion along $f$.
\vskip .2cm \noindent
(iii) In case $M = M_1 \wedge_R M_2 \wedge_R \cdots \wedge_R M_n$, with $M _i \eps Mod(B_i)$, then
$M \compl_{f, m_1, \cdots, m_n} \simeq \hlimDone \sigma_{\le m_1}({\mathcal T}^{\bullet}_{B_1}(M_1, C_1) \wedge_R \cdots \wedge_R
\hlimDn( \sigma_{\le m_n} {\mathcal T}_{B_n}(M, C_n))  \simeq M_{1} \compl_{f_1, m_1} \wedge_R M_{2} 
\compl_{ f_2, m_2} \wedge_R \cdots \wedge_R M_{n} \compl_{ f_n, m_n}$.
\end{proposition}
\begin{proof} The first equality in (i) follows  from the definition of the partial derived completion functor
in ~\eqref{M.f.2} and the second weak-equivalence in (i) then 
 follows from \cite[Proposition 2.9]{C08}. (The main point here is that the partial derived completions involve only finite
homotopy pull-backs.). The statement in (ii) now follows because the iterated homotopy inverse limit of a multi-cosimplicial object may be
replaced by the homotopy inverse limit of the diagonal cosimplicial object. (See, for example, \cite[Lemma 5.33]{Th85} for a proof.)
   The last statement then follows from the observation that
the triples $T_{B_i}$ commute with each other and by an application of \cite[Proposition 2.9]{C08}.
\end{proof}
 \begin{remark}
 For $i=1, \cdots, n$,  let $\rho_{\rmT_i}=f_i: \bK(\rmS, {\mathbb G}_m) \ra \bK(\rmS)$ denote the map induced by the rank-map where ${\mathbb G}_m$ is the 
$i$-th factor in the split torus $\rmT ={\mathbb G}_m^{\times n}$. Let $f= f_1 \wedge \cdots \wedge f_n$. Then using the terminology above, we proceed to establish the weak-equivalences
 (in the following theorem):
\be \begin{align}
     \label{der.compl.0}
\bK(\rmS,  \rmT)\compl_{\rho_T, m_1, \cdots, m_n} &\simeq \bK(\rmB\rmT^{gm,m_1, \cdots, m_n})  = \bK({\mathbb P}^{m_1} \times \cdots {\mathbb P}^{m_n}), \mbox{ and }\\
\bK(\rmS, \rmT)\compl_{\rho_T}  &\simeq \bK(\rmB\rmT^{gm}) = \holimm \bK({\mathbb P}^{m} \times \cdots \times {\mathbb P}^m) \notag
\end{align} \ee
\vskip .2cm \noindent
Corresponding statements hold for the completion with respect to the map $\rho_{\ell} \circ \rho_{T}:\bK(\rmS, T) \ra \bK(\rmS ) \ra  \bK(\rmS ){\underset {\Sigma} \wedge} \H(\Z/\ell)$ for a fixed prime $\ell$ different
from $char(k)=p$.
\end{remark}
The following results will play a major role in the proof of the main theorem. First observe that if $\rmT={\mathbb G}_m^n$ is an $n$-dimensional split torus, then
$\rmE\rmT^{gm, m} = ({\mathbb A}^{m+1}-0)^n$ and $\rmB\rmT^{gm,m} = ({\mathbb P}^m)^n$. The obvious map $\rmE\rmT^{gm,m} \ra S$ induces a 
compatible collection of maps $\{\bK(\rmS, \rmT) \ra \bK(\rmE\rmT^{gm,m}, \rmT) = \bK(\rmB\rmT^{gm,m})| m \ge 0\}$. 
\begin{theorem}
 \label{key.part.thm}
Assume the base scheme $\rmS$ is a field $k$ and that $\rmT= {\mathbb G}_m^n$ is a split torus over $k$. For each $i=1, \cdots, n$, let $\rmT_i$ denote
the ${\mathbb G}_m$ forming the $i$-th factor of $\rmT$. Let $\rho_{\rmT}: \bK(\rmS, \rmT) \ra 
\bK(\rmS )$  denote the restriction map as in ~\ref{der.compl.eqK}. 
Then the map $\bK(\rmS, \rmT) \ra \bK(\rmE\rmT^{gm,m}, \rmT) = \bK(\rmB\rmT^{gm,m})$ factors through the multiple partial derived completion
\[\bK(\rmS, \rmT) \compl_{\rho_T, m, \cdots, m} = \holimDm \sigma _{\le m} \cdots \holimDm \sigma_{\le m}{\mathcal T}^{\bullet, \cdots, \bullet}_{\bK(\rmS, \rmT_1), \cdots, \bK(\rmS, \rmT_n)}(\bK(\rmS, \rmT), \bK(\rmS))\]
 and induces a weak-equivalence:
\[\bK(\rmS, \rmT)\compl_{\rho_{\rmT, m, \cdots, m}} \simeq \bK(\rmB \rmT^{gm, m} \times \cdots \times \rmB\rmT^{gm,m}).\]
 Corresponding results also hold for the K-theory spectra replaced by the mod-$\ell$ $K$-theory
 spectra defined as in ~\eqref{KGl}.
\end{theorem}
\begin{proof} Since all the arguments in the proof carry over when the $K$-theory spectrum $\bK$ is replaced by the
mod-$\ell$ K-theory spectra, we will only consider the K-theory spectrum. First observe that 
\be \begin{align}
     \label{KST}
\pi_*(\bK(\rmS, \rmT)) &= \rmR(\rmT) {\underset {\Z} \otimes} \pi_*(\bK(\rmS)) 
\end{align} \ee
\vskip .2cm \noindent
{\it Step 1}. We will next restrict to the case where the torus $\rmT$ is one dimensional. 
Let $\lambda$ denote the $1$-dimensional representation of $\rmT$ corresponding to a generator of $\rmR(\rmT)$. For each $m \ge 0$,
$(\lambda -1)^{m+1}$ defines a class in $\pi_0(\bK(\rmS, \rmT)) = \rmR(\rmT) \otimes \pi_0(\bK(\rmS))$. Multiplication by this class
defines a map of spectra 
\[\bK(\rmS, \rmT) \ra \bK(\rmS, \rmT)\]
which will be still denoted $(\lambda -1)^{m+1}$. {\it The heart of the proof of this theorem involves showing that the homotopy cofiber of
the map $(\lambda -1)^{m+1}$ identifies with the partial derived completion to the order $m$, along the homotopy fiber of
the restriction map $\rho_{\rmT}: \bK(\rmS, \rmT) \ra \bK(\rmS)$.}
\vskip .2cm \noindent
{\it Step 2}. Next observe that the principal ideal $(\lambda -1)^{m+1}$ in $\rmR(\rmT)$ is a {\it flat} module over $\rmR(\rmT)$. Observe that
$\rmR(\rmT) = {\mathbb Z}[\lambda , \lambda ^{-1}]$. It is enough to show that the principal ideal $(\lambda -1)^{m+1}$ 
defines a flat ideal in each ${\mathbb Z}/p[\lambda, \lambda ^{-1}]$ for each prime number $p \eps {\mathbb Z}$ and also
a flat ideal in ${\mathbb Q}[\lambda, \lambda ^{-1}]$. This is clear since, each ${\mathbb Z}/p[\lambda]$ and ${\mathbb Q}[\lambda ]$
is a principal ideal domain and the ideal $(\lambda -1)$ is torsion free.  
\vskip .2cm
Let $\rmS \times {\mathbb A}^{m+1}$ denote the vector bundle of rank $m+1$ over $\rmS$ on which the torus $\rmT$ acts diagonally.
The {\it Koszul-Thom class} of this bundle is $\pi^*(\lambda -1)^{m+1}$, where $\pi: \rmS \times {\mathbb A}^{m+1} \ra \rmS$ is the
projection. This defines a class in $\pi_0({\bK(\rmS \times {\mathbb A}^{m+1}, \rmS \times ({\mathbb A}^{m+1}-0), T)})$. 
 Next consider the  diagram:
\be \begin{equation}
     \label{key.diagm}
\xymatrix{{\bK(\rmS, \rmT)} \ar@<1ex>[r]^{(\lambda-1)^{m+1}} \ar@<-1ex>[d]^{\simeq} & {\bK(\rmS, \rmT)} \ar@<1ex>[r] \ar@<-1ex>[d] ^{\simeq} & {cofiber ((\lambda-1)^{m+1})} \ar@<1ex>[d]\\
{\bK(\rmS \times {\mathbb A}^{m+1}, \rmS \times ({\mathbb A}^{m+1}-0), T)} \ar@<1ex>[r] & {\bK(\rmS \times {\mathbb A}^{m+1}, \rmT)} \ar@<1ex>[r] & {\bK(\rmS \times ({\mathbb A}^{m+1}-0), \rmT)}}
  \end{equation} \ee
\vskip .2cm \noindent
where the bottom row is the stable homotopy fiber sequence associated to the  homotopy cofiber sequence: $\rmS\times( {\mathbb A}^{m+1}-0) \ra \rmS \times {\mathbb A}^{m+1} \ra 
(\rmS \times {\mathbb A}^{m+1})/(\rmS \times ({\mathbb A}^{m+1}-0))$ in the ${\mathbb A}^1$-stable homotopy category. The left-most vertical map
is provided by {\it Thom-isomorphism}, i.e. by cup-product with the Koszul-Thom-class $\pi^*(\lambda -1)^{m+1}$. Therefore, this map is
a weak-equivalence. The second vertical map is a weak-equivalence by the homotopy property. It is also clear that the the left-most
square homotopy commutes. Therefore, the right square also homotopy commutes and the right most vertical map is also a weak-equivalence.
\vskip .2cm \noindent
{\it Step 3}. Next we take $m=0$. Then it follows from what we just showed that the map $cofiber(\lambda -1) \ra \bK(\rmS \times {\mathbb A}^1-0, \rmT) = \bK(\rmS)$
(forming the last vertical map in ~\eqref{key.diagm}) is a weak-equivalence. Now the map 
$\bK(\rmS \times {\mathbb A}^{m+1}, {\mathbb G}_m) \ra \bK(S \times {\mathbb G}_m, {\mathbb G}_m) = \bK(\rmS)$
forming the bottom row in the right-most  square of ~\eqref{key.diagm} identifies with restriction
map $\bK(\rmS, {\mathbb G}_m) \ra \bK(\rmS)$. Therefore, by the homotopy commutativity of the right-most square of 
~\eqref{key.diagm}, it follows that the homotopy fiber of the restriction map $\rho_{\rmT}:\bK(\rmS, \rmT) \ra \bK(\rmS)$ identities 
up to weak-equivalence with the homotopy fiber of the map $\bK(\rmS, \rmT) \ra cofiber(\lambda -1)$, i.e. 
with the map $(\lambda -1):\bK(\rmS, \rmT) \ra \bK(\rmS, \rmT)$. 
\vskip .2cm
{\it Step 4}. Next recall (using \cite[Corollary 6.7]{C08}) that the partial derived completion $\bK(\rmS, \rmT) \compl_{\rho_{\rmT}, m}$  
may be identified as follows.
Let $\tilde I_{\rmT} \ra I_{\rmT}$ denote a cofibrant replacement in the 
category of module spectra over $\bK(\rmS, \rmT)$, with $I_{\rmT}$ denoting the homotopy fiber of the restriction $\bK(\rmS, \rmT) \ra \bK(\rmS)$. Then
\[\bK(\rmS, \rmT) \compl_{\rho_{\rmT}, m}=  Cofib( {\overset {m+1} {\overbrace {{\tilde I_{\rmT}} {\underset {\bK(\rmS, \rmT)} \wedge} \tilde I_{\rmT} {\underset {\bK(\rmS, \rmT)} \wedge} \cdots {\underset {\bK(\rmS, \rmT)} \wedge} \tilde I_{\rmT}}}} \ra \bK(\rmS, \rmT)).\]
What we have just shown in Step 3 is  that the homotopy fiber $I_{\rmT}$ of the restriction  map $ \bK(\rmS, \rmT) \ra \bK(\rmS)$ identifies with the map $\bK(\rmS, \rmT) {\overset {(\lambda -1)} \ra} \bK(\rmS, \rmT)$ and hence is clearly cofibrant over $\bK(\rmS, \rmT)$. 
i.e. $\tilde I_{\rmT} = I_{\rmT}$ is simply $\bK(\rmS, \rmT)$ mapping into
$\bK(\rmS, \rmT)$ by the map $(\lambda -1)$. Therefore the homotopy cofiber
\[Cofib( {\overset {m+1} {\overbrace {{\tilde I_{\rmT}} {\underset {\bK(\rmS, \rmT)} \wedge} \tilde I_{\rmT} {\underset {\bK(\rmS, \rmT)} \wedge} \cdots {\underset {\bK(\rmS, \rmT)} \wedge} \tilde I_{\rmT}}}} \ra \bK(\rmS, \rmT)) = Cofiber (\bK(\rmS, \rmT) {\overset {(\lambda -1)^{m+1}} \ra} \bK(\rmS.\rmT)).\]
 Making use of the  the diagram ~\eqref{key.diagm}, these observations prove  that
\[\bK(\rmS, \rmT) \compl_{\rho_{\rmT}, m} \simeq cofiber (\lambda -1)^{m+1} \simeq \bK(\rmS \times ({\mathbb A}^{m+1}-0), \rmT) =\bK(\rmS \times {\mathbb P}^m) = \bK(BT^{gm,m}).\]
\vskip .2cm \noindent
 This completes the proof for the case of the $1$-dimensional 
torus.  (It may be worthwhile pointing out that classical argument due to Atiyah and Segal for the
usual completion for torus actions in equivariant topological K-theory is very similar: see \cite[section 3, Step 1]{AS69}.)
\vskip .2cm
{\it Step 5}. Next suppose that $\rmT = {\mathbb G}_m^n$. Letting $\rmT_{n-1}=$ the product of the first $(n-1)$ copies
of ${\mathbb G}_m$, one observes that $\bK(\rmS, \rmT) \simeq \bK(\rmS, \rmT_{n-1})\wedge_{\bK(\rmS)} \bK(\rmS, {\mathbb bG}_m)$. Therefore one may use ascending induction on $n$ and the above discussion on multiple
derived completions to complete the proof of the theorem. \end{proof}
\vskip .2cm
\begin{theorem}
\label{key.thm.1} Assume the base scheme $\rmS$ is a field $k$ and that $\rmX$ is either a scheme or algebraic space of finite type over $k$. Let $\rmH=\rmT={\mathbb G}_m^n$ denote the maximal torus in $\rmG$. Then, for each fixed positive integer $m$,
 the map 
\[\bG(\rm{\rm X}, \rmH) \simeq \bK(\rmS, \rmH) \wedgeKH \bG(\rm{\rm X}, \rmH) \ra \bK(E\rmH^{gm,m}, \rmH) \wedgeKH \bG(\rm{\rm X}, \rmH) \ra \bG(E\rmH^{gm,m}{\underset {\rmH}  \times}\rmX)\]
 factors through the partial derived completion $\bG(\rm{\rm X}, \rmH)\compl_{\rho_{\rmH},m}$ and the induced map
$\bG(\rm{\rm X}, \rmH)\compl_{\rho_{H}} \ra \holimm \bG(E\rmH^{gm,m}{\underset {\rmH}  \times}\rmX)$ is a weak-equivalence. Moreover, this weak-equivalence is natural in 
both $\rmX$ and $\rmH$.
 The corresponding assertions hold
with the $G$-theory spectrum replaced by the mod-$\ell$ $G$-theory spectrum (defined as in ~\eqref{KGl}) in general, and by the $K$-theory and mod-$\ell$ K-theory spectrum
 when $\rmX$ is  regular.
\end{theorem}
\begin{proof} 
Since this is one of the main results of the paper, with several potential applications, we will provide two distinct proofs. First observe that $\rmE\rmH^{gm,m}$ identifies with the product ${\overset {n} {\overbrace {\rmE{\mathbb G}_m^{gm,m} \times \cdots \times \rmE{\mathbb G}_m^{gm,m}}}}$. We will first discuss the
shortest proof.
\vskip .2cm \noindent
{\it First proof}. For each $i=1, \cdots, n$, let $\rmT_i$ denote
the ${\mathbb G}_m$ forming the $i$-th factor of $\rmT$. We begin with the isomorphism of truncated multi-cosimplicial objects provided by multiple partial derived completion along the 
map $\bK(\rmS, \rmT_i) \ra \bK(\rmS)$ (which follows readily from the definition of the triple: see section 
~\ref{mult.der.compl}):
\be \begin{align}
     \label{key.isom}
\begin{split}
& \sigma _{\le m_1, \cdots, m_n} ({\mathcal T}^{\bullet, \cdots, \bullet}_{\bK(\rmS, \rmT_1), \cdots, \bK(\rmS, \rmT_n)}(\bK(\rmS, \rmH){\overset L {\underset {\bK(\rmS, \rmH)} \wedge}} \bG(\rm{\rm X}, \rmH), \bK(\rmS)) \\
& \cong (\sigma_{\le m_1, \cdots, \le m_n} ({\mathcal T}^{\bullet, \cdots, \bullet}_{\bK(\rmS, \rmT_1), \cdots \bK(\rmS, \rmT_n)}(\bK(\rmS, \rmH), \bK(\rmS)))) {\overset L {\underset {\bK(\rmS, \rmH)} \wedge}} \bG(\rm{\rm X}, \rmH) \notag
\end{split}
    \end{align} \ee
\vskip .2cm \noindent
Clearly the left-hand-side identifies with $\sigma _{\le m_1, \cdots, \le m_n} ({\mathcal T}^{\bullet, \cdots, \bullet}_{\bK(\rmS, \rmT_1), \cdots \bK(\rmS, \rmT_n)}(\bG(\rm{\rm X}, \rmH), \bK(\rmS)))$. Next one takes the homotopy inverse
limit over $\Delta ^{\le m}$ of the diagonal of the  truncated cosimplicial objects. Clearly the homotopy inverse limit of the left-hand-side is
\vskip .2cm
$\holimDm \sigma _{\le m, \cdots, \le m} ({\mathcal T}^{\bullet, \cdots \bullet}_{\bK(\rmS, \rmT_1), \cdots \bK(\rmS, \rmT_n)}(\bG(\rm{\rm X}, \rmH), \bK(\rmS))) = \bG(\rmY, \rmH) \compl_{\rho_{\rmT_1}, \cdots, \rho_{\rmT_n}, m, \cdots, m}$.
\vskip .2cm \noindent
By \cite[Proposition 2.9]{C08}, the homotopy inverse limit of the diagonal of the right-hand-side identifies with
\vskip .2cm
$(\holimDm (\sigma_{\le m, \cdots, \le m} ({\mathcal T}^{\bullet, \cdots, \bullet}_{\bK(\rmS, \rmT_1), \cdots, \bK(\rmS, \rmT_n)}(\bK(\rmS, \rmH), \bK(\rmS))))) {\overset L {\underset {\bK(\rmS, \rmH)} \wedge}} \bG(\rm{\rm X}, \rmH)$
\vskip .2cm \noindent
By Remark ~\ref{der.compl.0}, this identifies with $\bK(\rmE\rmH^{gm,m}, \rmH) {\overset L {\underset {\bK(\rmS, \rmH)} \wedge}} \bG(\rm{\rm X}, \rmH)$. (Since $\rmH = {\mathbb G}_m^n= \Pi_{i=1}^n \rmT_i$, $\rmE\rmH^{gm,m}= \Pi_{i=1}^n \rmE\rmT_i^{gm,m}$.) By the Kunneth formula
in the equivariant setting proven in Proposition ~\ref{key.prop}, this identifies with $\bG(\rmE\rmH^{gm,m}{\underset {\rmH}  \times}\rm{\rm X}, \rmH)$. Summarizing, we
obtain the weak-equivalences:
\be \begin{equation}
     \label{key.weak.eq.1}
\holimDm \sigma _{\le m, \cdots, \le m} ({\mathcal T}^{\bullet, \cdots, \bullet}_{\bK(\rmS, \rmT_1), \cdots, \bK(\rmS, \rmT_n)}(\bG({\rm X}, \rmH), \bK(\rmS))) \simeq \bG(E\rmH^{gm,m}{ \times}\rm{\rm X}, \rmH) \simeq 
\bG(\rmE\rmH^{gm,m}{\underset {\rmH}  \times}\rmX)
\end{equation} \ee
\vskip .2cm \noindent
which are compatible as $m \ra \infty$. Therefore, one takes the homotopy inverse limit of both sides as $m \ra \infty$ to obtain
the weak-equivalence claimed in the theorem for $G$-theory. Clearly the arguments extend to the $\rho_{\ell} \circ \rho_{H}$-completed 
$G$-theory. When $\rmX$ is regular, the property ~\ref{PD} shows that one may replace equivariant $G$-theory everywhere by 
the corresponding equivariant K-theory. The naturality of this weak-equivalence
follows from the functoriality of the derived completion: see \cite[Proposition 3.2, 1]{C08}.
\vskip .2cm \noindent
{\it Second proof}. For this proof we will  restrict to the case when $\rmH$ is a one dimensional torus, 
though the proof may be extended to the general case with a little effort. Let ${\mathcal I}_{\rmH}$ denote the homotopy fiber of the map 
$\bK(\rmS, \rmH) \ra \bK(\rmS)$ induced
by the restriction map. One may readily see that it is module spectrum over $\bK(\rmS, \rmH)$. Then, we will replace this by a cofibrant object in
$Mod(\bK(\rmS, \rmH))$ and denote the resulting object by the same symbol. Therefore, one may now form the derived smash-product
\vskip .2cm
${\mathcal I}_{\rmH}^{m+1}= {\overset {m+1} {\overbrace {{{\mathcal I}_{\rmH}} {\underset  {\bK(\rmS, \rmH)} \wedge}  
{{\mathcal I}_{\rmH}} \cdots {{\underset {\bK(\rmS, \rmH)} \wedge } {{\mathcal I}_{\rmH}}}}}}$.
\vskip .2cm \noindent
\cite[Corollary 6.7]{C08} then shows that 
\vskip .2cm
${\mathcal I}_{\rmH}^{m+1} \ra \bK(\rmS, \rmH) \ra \bK(\rmS, \rmH)\compl_{I_{\rmH}, m} \simeq \bK(\rmE\rmH^{gm,m}, \rmH)$
\vskip .2cm \noindent
is a stable fibration sequence of spectra, and therefore also a stable cofibration sequence of spectra. Let ${\widetilde {\bG({\rm X}, \rmH)}}$ denote
a cofibrant replacement of $\bG({\rm X}, \rmH)$ in the model category $Mod(\bK(\rmS, \rmH))$. Now taking smash-product over $\bK(\rmS, \rmH)$ with the spectrum ${\widetilde {\bG(\rm{\rm X}, \rmH)}}$
provides the stable cofibration (or equivalently stable fibration) sequence:
\vskip .2cm
${\mathcal I}_{\rmH}^{m+1} {\underset {\bK(\rmS, \rmH)} \wedge} {\widetilde {\bG({\rm X}, \rmH)}} \ra {\widetilde {\bG({\rm X}, \rmH)}} \ra
 \bK(\rmE\rmH^{gm,m}, \rmH) {\underset {\bK(\rmS, \rmH)} \wedge} {\widetilde {\bG({\rm X}, \rmH)}} \simeq \bG(\rmE\rmH^{gm,m}{\underset {\rmH}  \times}\rmX)$.
\vskip .2cm \noindent
The last weak-equivalence is provided by Proposition ~\ref{key.prop}.
One also obtains the stable cofibration (or equivalently stable fibration) sequence (by smashing the stable cofiber sequence
${\mathcal I}_{\rmH}^{m+1} \ra \bK(\rmS, \rmH) \ra \bK(\rmS, \rmH)/{\mathcal I}_H^{m+1}$ with ${\widetilde {\bG({\rm X}, \rmH)}}$ over $\bK(\rmS, \rmH)$):
\vskip .2cm
${\mathcal I}_{\rmH}^{m+1} {\underset {\bK(\rmS, \rmH)} \wedge} {\widetilde {\bG({\rm X}, \rmH)}} \ra  {\widetilde  {\bG({\rm X}, \rmH)}} \ra
\bK(\rmS, \rmH)/{\mathcal I}_{\rmH}^{m+1} {\underset {\bK(\rmS, \rmH)} \wedge} {\widetilde {\bG(\rm{\rm X}, \rmH)}} \simeq 
{\widetilde {\bG({\rm X}, \rmH)}}/({\mathcal I}_{\rmH}^{m+1}{\underset {\bK(\rmS, \rmH)} \wedge} {\widetilde {\bG({\rm X}, \rmH)}})$.
\vskip .2cm \noindent
By \cite[Corollary 6.7]{C08} again, one sees that last term above identifies with $\bG({\rm X}, \rmH)_{\rho_{\rmH}, m}$. Therefore, one obtains the
weak-equivalence
\be \begin{equation}
     \label{key.weak.eq.2}
\bG({\rm X}, \rmH)_{\rho_{\rmH}, m} \simeq \bG(\rmE\rmH^{gm,m}{\underset {\rmH}  \times}\rmX). 
\end{equation} \ee
\vskip .2cm \noindent
Taking the homotopy inverse limit as $m \ra \infty$, one obtains the conclusion of the theorem once more.
\end{proof}
We conclude this section with the proofs of Theorem ~\ref{main.thm.1} and  Theorem ~\ref{comparison}.
\vskip .2cm \noindent
{\bf Proof of Theorem ~\ref{main.thm.1}}. Observe that Theorem ~\ref{key.thm.1} proves Theorem ~\ref{main.thm.1} when the group-scheme is 
a split torus and Theorem ~\ref{main.thm.3} then proves Theorem ~\ref{main.thm.1} when the group-scheme is a finite product of $\GL_n$s.
Next let $\rmG$ denote a split reductive group over $k$ satisfying the Standing Hypotheses ~\ref{stand.hyp.1}. Then, we may imbed
$\rmG$ as a closed sub-group-scheme in a finite product $\GL_{n_1} \times \cdots \times \GL_{n_q}$ so that the restriction map
$\rmR(\GL_{n_1} \times \cdots \times \GL_{n_q}) \ra \rmR(\rmG)$ is surjective. Therefore we obtain the weak-equivalences:
\[\bG(\rm{\rm X}, \rmG)\compl_{\rho_{\rmG}} \simeq \bG(\rm{\rm X}, \rmG) \compl_{\rho_{\GL_{n_1} \times \cdots \times \GL_{n_q}}} \simeq \bG(\rmX{\underset {\rmG} \times} \GL_{n_1} \times \cdots \times \GL_{n_q}, \GL_{n_1} \times \cdots \times \GL_{n_q})\compl_{\rho_{\GL_{n_1} \times \cdots \times \GL_{n_q}}}\]
where the first weak-equivalence is by Theorem ~\ref{comparison} and the next weak-equivalence is by ~\ref{KG.props}(iv). 
Since $\bG(\rmE\rmG^{gm} \times _{\rmG}X) \simeq \bG(\rmE\GL_{n_1} \times \cdots \times \rmE\GL_{n_q} \times_{\GL_{n_1} \times \cdots \times \GL_{n_q}} (\GL_{n_1} \times \cdots \times \GL_{n_q}\times_{\rmG}X))$
this completes the proof of Theorem ~\ref{main.thm.1}. \qed
\vskip .2cm \noindent
{\bf Proof of Theorem ~\ref{comparison}}. One observes that the hypotheses of \cite[Corollary 7.11]{C08} are satisfied with
the ring spectrum $A= \bK(\rmS, \rmG)$, the ring spectrum $B= \bK(\rmS, \rmH)$ and the ring spectrum $C= \bK(\rmS)$.\qed

\section{\bf Examples and Applications}
\label{sec.egs}
\subsection{Comparison with Thomason's Theorem}
In this section we will first compare our results with prior results on Atiyah-Segal type completion theorems in equivariant G-theory, the earliest  of which
are those of Thomason. In \cite[Theorem 3.2]{Th86}, Thomason proves the following theorem:
\begin{theorem}\label{Thomason}
Let $\rmG$ denote a linear algebraic group scheme over a separably closed field $k$ so that it is the product of a group scheme smooth over $k$ and
an infinitesimal group scheme. Let $\rmX$ denote a separated algebraic space on which $\rmG$ acts. Let $\ell$ denote a prime
that is invertible in $k$, let $\beta$ denote the Bott element and let $I_{\rmG}$ denote the kernel of rank map $R(\rmG) \ra Z$.
Then the map
\[\pi_*(\bG/\ell^{\nu}(\rmX, \rmG)[\beta^{-1}]) \compl_{I_{\rmG}} {\overset {\simeq} \ra} \pi_*\bG/\ell^{\nu}(E\rmG^{gm}{\underset {\rmG} \times}\rmX)[\beta^{-1}]) \]
is an isomorphism, where the completion on the left denotes the completion of the homotopy groups
$\pi_*(\bG/\ell^{\nu}(\rmX, \rmG)[\beta^{-1}])$  at the ideal $I_{\rmG}$ in the usual sense. 
\end{theorem}
\vskip .2cm \noindent
 In the above theorem of Thomason, in fact the strategy is to show
that a spectral sequence whose $E_2$-terms are the cohomology groups computed on a site called the {\it isovariant \'etale site} associated to the
$\rmG$-scheme $\rmX$ with ${\mathbb Z}/\ell^{\nu}$-coefficients converges strongly to the above homotopy groups. (The above spectral sequence is an
equivariant analogue of the spectral sequence relating \'etale cohomology to mod$-\ell^{\nu}$- Algebraic K-theory with the Bott-element inverted as
in \cite{Th85}.) {\it All the stringent hypotheses stated there seem necessary 
to obtain strong convergence of this spectral sequence and thereby to conclude that  the homotopy groups $\pi_*(\bG/\ell^{\nu}(\rmX, \rmG)[\beta^{-1}])$
are finite modules over the representation ring $R(\rmG)$.} Nevertheless, the following simple counter-example shows that for lots varieties with
 group actions, the Bott-element inverted form of equivariant algebraic K-theory with finite coefficients is distinct from the
 corresponding equivariant algebraic K-theory with finite coefficients. Thus Thomason's theorem does {\it not provide}
 an Atiyah-Segal type completion theorem for equivariant Algebraic K-theory, but only for its topological variant, i.e. with
 the Bott element inverted.
\subsubsection{Equivariant Algebraic K-theory and the Equivariant Algebraic K-theory with the Bott element inverted are not 
  isomorphic in general} \label{counter.eg}
  \vskip .2cm 
  Let $k$ denote an algebraically closed field and let $\ell$ denote a prime different from $char(k)$. 
  Let $s\ge 1$ be an integer and let $\rmX = {\mathbb G}_m^s$. Then we make the following observation:
  \begin{proposition}
   \label{comp.Bott.in} The map $\pi_i(\bK({\mathbb G}_m^s)/\ell^{\nu}) \ra \pi_i(\bK({\mathbb G}_m^s)/\ell^{\nu}[\beta^{-1}])$
   is an isomorphism for all $i\ge 0$ if $s=1$ and for $s \ge 2$, it is {\it not}
   surjective for $i<s-1$. In particular,  for all $i < s-1$,  the above map is only an injection. 
  \end{proposition}
\begin{proof} For $s=1$, this follows readily from Quillen's calculation of the $K$-groups of 
${\mathbb G}_m$: i.e. $\pi_i(\bK({\mathbb G}_m)) = \pi_i(\bK(Spec\, k)) \oplus \pi_{i-1}(\bK(Spec \, k))$,
for $i \ge 0$. (See \cite[Corollary to Theorem 8]{Qu}.) Therefore, let $s \ge 2$. Then, observing that ${\mathbb G}_m$ is a linear scheme, one invokes the derived Kunneth formula
in \cite[Theorem 4.2]{J01}. Then one observes that the corresponding spectral sequences degenerate (in view of the
computation of the $K$-groups of ${\mathbb G}_m$ above), thereby providing
the isomorphisms
\begin{align}
\pi_*(\bK({\mathbb G}_m^s)/\ell^{\nu}) &\cong \pi_*(\bK({\mathbb G}_m^{s-1})/\ell^{\nu}) {\underset {\pi_*(\bK(Spec \, k)/\ell^n)} \otimes} \pi_*(\bK({\mathbb G}_m)/\ell^{\nu}),\\
\pi_*(\bK({\mathbb G}_m^s)/\ell^{\nu}[\beta^{-1}]) &\cong \pi_*(\bK({\mathbb G}_m^{s-1})/\ell^{\nu}[\beta^{-1}]) {\underset {\pi_*(\bK(Spec \, k)/\ell^{\nu}[\beta^{-1}])} \otimes} \pi_*(\bK({\mathbb G}_m)/\ell^{\nu}[\beta^{-1}]) \notag\\
&\cong \pi_*(\bK({\mathbb G}_m^{s-1})/\ell^{\nu}[\beta^{-1}) {\underset {\pi_*(\bK(Spec \, k)/\ell^{\nu})} \otimes} \pi_*(\bK({\mathbb G}_m)/\ell^{\nu}). \notag\end{align}
Now an induction on $s$, will show
that the induced map of the right-hand-sides is an injection on homotopy groups for all $i\ge 0$,
and that it fails to be a surjection for $i<s-1$, since the Bott-element inverted K-theory of ${\mathbb G}_m^{s-1}$ and $Spec \, k$
will be non-trivial in infinitely many negative degrees.
 \end{proof}
 \vskip .2cm
Now let $s \ge 3$, $\rmG= {\mathbb G}_m$ acting by translation on the first factor in $\rmX ={\mathbb G}_m^s$
 and trivially on the other factors.  Observing that the classifying space for ${\mathbb G}_m$ is ${\mathbb P}^{\infty}$, 
 one can readily see that $\rmE\rmG{\underset {\rmG} \times}\rmX \simeq \rmE\rmG \times {\mathbb G}_m^{s-1} \simeq {\mathbb G}_m^{s-1}$. Therefore, it follows 
 that the homotopy groups of the  mod-$\ell^{\nu}$ algebraic K-theory spectrum of $\rmE\rmG{\underset {\rmG} \times}\rmX $
  and the homotopy groups of the corresponding K-theory spectrum with the Bott-element inverted
 are not isomorphic in degrees $<s-2$.  The case where ${\mathbb G}_m$ acts diagonally on ${\mathbb G}_m^s$ may be reduced
 to this case by viewing ${\mathbb G}_m^s\cong \Delta({\mathbb G}_m) \times ({\mathbb G}_m^{s}/ \Delta({\mathbb G}_m))$. Moreover, the above counter-example shows that
  for varieties provided with an action by a torus ${\mathbb G}_m$, and provided with strata that are products of affine spaces and
   split tori (as is the case with all linear varieties), the ${\mathbb G}_m$-equivariant Algebraic K-theory is likely to be distinct
   from the corresponding Bott-element inverted variant.
 \vskip .2cm
{\it Our main result, Theorem ~\ref{main.thm.1}, is a huge improvement over the above theorem of Thomason, since we do not need to work modulo a prime nor with 
the Bott element inverted, nor with any other serious restriction such as the base field being separably closed. This is made 
possible by the use of derived completions.}
\vskip .2cm
\subsection{Comparison with the completion theorem in \cite{K}}
Next we proceed to compare our results in detail with the most recent attempt in \cite{K}, at proving Atiyah-Segal completion theorems, 
for equivariant algebraic K-theory integrally and without inverting the Bott element, making use of the classical completion of the homotopy groups of the equivariant 
K-theory spectra at the augmentation ideal of the representation ring.
\vskip .2cm
\subsubsection{The range of allowed group actions}
\vskip .2cm
We may assume the base scheme $\rmS$ is a regular Noetherian affine scheme of finite type over a field, in general, though,
we may restrict to the case where $\rmS$ is the spectrum of a field to keep the discussion simple enough.
We allow any smooth linear algebraic group scheme acting on any quasi-projective variety (scheme) that admits a $\rmG$-equivariant closed
immersion into a regular quasi-projective variety (scheme) provided with a $\rmG$-action. This puts very little restriction on
the groups that are allowed: {\it clearly all finite groups may be imbedded as closed subgroups of a suitable ${\rm GL}_n$
and therefore are allowed in Theorem ~\ref{main.thm.1}.  In contrast, \cite[Theorem 1.2]{K} has to restrict to connected groups, thereby disallowing even 
actions by finite abelian groups.}
\vskip .1cm
In addition, as shown in \cite[Theorem 1.5 and section 10]{K}, there are  counter-examples that show that \cite[Theorem 1.2]{K}, {\it fails even for smooth varieties such as homogeneous spaces for a linear algebraic
	group, which are not projective}. The source of the counter-example is that the first map in \cite[(10.3)]{K} is 
not surjective. On using derived completions, the corresponding map from the partial derived completion to the
corresponding term of the Borel-style equivariant G-theory spectrum is a weak-equivalence as shown in ~\eqref{key.weak.eq.2}. Therefore,
this counter-example does not arise on using derived completions.
\vskip .2cm
\begin{remark}The kind of dimension shifting that occurs in \cite[Theorem 1.5]{K} is quite natural from our point of view.  
	In our case, the derived completion when applied to rings and modules, creates higher dimensional homotopy.  
	For example, the derived completion at the ring homomorphism ${\mathbb Z} \rightarrow {\mathbb F}_{\ell}$ of 
	the module ${\mathbb Q}/{\mathbb Z}$ has $\pi _0 \cong 0$ and $\pi_1 \cong {\mathbb Z}_{\ell}$.  
	In \cite[Theorem 1.5]{K}, this is accounted for by the fact that the derived completion of 
	the module ${\mathbb Q}/{\mathbb Z} \otimes {\mathbb Z}[{\mathbb Z}/2{\mathbb Z}]$ at the augmentation 
	homomorphism ${\mathbb Z}[{\mathbb Z}/2{\mathbb Z}] \rightarrow {\mathbb Z}$ has $\pi _0 \cong {\mathbb Z}$ and $\pi _1 \cong {\mathbb Z}_2$.  
\end{remark}
{\it The rest of this section will focus on various applications and good features of the derived completion developed in the earlier sections.}
\vskip .1cm
\subsection{Equivariant G-Theory and K-theory of $\rmG$-quasi-projective varieties and schemes: Toric varieties, Spherical varieties and Linear varieties}
\label{G.quasiproj}
We will henceforth restrict to the case the base field is a field.
 For a given linear algebraic group $\rmG$, a convenient class of schemes we consider
 are what are called {\it $\rmG$-quasi-projective schemes}: a quasi-projective scheme $\rmX$ with a given $\rmG$-action is
 $\rmG$-quasi-projective, if it admits a $\rmG$-equivariant locally closed immersion into a projective space,  ${\rm Proj(V)}$,
 for a representation ${\rm V}$ of $\rmG$. Then a well-known theorem of Sumihiro, (see \cite[Theorem 2.5]{Sum}) shows that any normal quasi-projective variety provided
 with the action of a connected linear algebraic group $\rmG$ is $\rmG$-quasi-projective. This may be extended to the 
 case where $\rmG$ is not necessarily connected by invoking the above result for the action of $\rmG^o$ and then
 considering the closed $\rmG$-equivariant closed immersion of the given $\rmG$-scheme into 
 $\Pi_{g \eps \rmG/\rmG^o} {\rm Proj(V)} $ followed by a $\rmG$-equivariant closed immersion of the latter into
 ${\rm Proj}(\otimes_{g \eps \rmG/\rmG^o}{\rm V})$. (See for example. \cite[p. 629]{Th88}.)
 \vskip .2cm
 Now the important observation is that {\it quasi-projective toric varieties} are defined to be {\it normal} varieties on
 which a split torus $\rmT$ acts with finitely many orbits. Therefore, such toric varieties are $\rmT$-quasi-projective.
 This already includes many familiar varieties: see \cite{Oda}, for example.
 \vskip .2cm
 Given a connected reductive group $\rmG$, a {\it quasi-projective $\rmG$-spherical variety} is defined to be a normal $\rmG$-variety 
 on which $\rmG$, as well as a Borel subgroup of $\rmG$, have finitely many orbits. (See \cite{Tim}.) This class of varieties clearly includes all 
 toric varieties (when the group is a split torus), but also includes many other varieties. The discussion in
 the first paragraph above, now shows these are also $\rmG$-quasi-projective varieties.
 \vskip .2cm
 Finally we recall that linear varieties are those varieties that have a stratification where each stratum is a product of an
  affine space and a split torus. We will restrict to quasi-projective linear varieties with an action by the
  linear algebraic group $\rmG$, where each stratum is also $\rmG$-stable. It is clear the $\rmG$-equivariant $\bG$-theory
  of such varieties can be readily understood in terms of the corresponding $\rmG$-equivariant $\bG$-theory of 
  the strata. Examples of such varieties are normal quasi-projective varieties on which a connected split solvable group
  acts with finitely many orbits. (See \cite[p. 119]{Rosen}.)
  \vskip .2cm
  {\it It is important to observe that all the main theorems in this paper, Theorems ~\ref{main.thm.1} and ~\ref{comparison}
  apply to {\it all} the above classes of schemes, without any restrictions, i.e. to all toric varieties, all spherical varieties and
  all $\rmG$-quasi-projective varieties, irrespective of whether they are smooth or projective. See also ~\ref{comput.homotopy.groups} below. }
  \vskip .2cm
  The counter-example in ~\ref{counter.eg} shows that for large classes of toric and spherical varieties,
  the equivariant algebraic K-theory with finite coefficients away from the characteristic is not isomorphic to
  the corresponding equivariant K-theory with the Bott element inverted, so that Thomason's theorem does not provide
  a completion theorem for equivariant K-theory unless the Bott element is inverted. Neither does \cite[Theorem 1.2]{K} unless these varieties are stratified by affine spaces as shown in the following paragraph.
  \vskip .2cm
  If one imposes the restriction that the schemes have to be also projective and smooth, then that cuts down the 
  class of varieties above drastically: since all toric and spherical varieties come equipped with the action of a (maximal) 
  split torus which has only finitely many fixed points, it would mean that the resulting varieties are those
  that have a stratification by affine spaces alone. {\it Thus the only toric or spherical varieties for which the 
  completion theorem of \cite[Theorem 1.2]{K} apply are the varieties stratified by affine spaces, which is clearly an 
  extremely narrow class of varieties.}
  \vskip .2cm
  Moreover, for projective smooth toric and spherical varieties over separably closed fields, in all non-negative degrees, the
  left-hand-side (right-hand-side) terms in Thomason's theorem stated above, are isomorphic to the 
  homotopy groups of the corresponding equivariant algebraic K-theory spectrum without inverting the Bott-element. (To see this for the 
  classifying space, one may first reduce to the case the group is a split torus and then observe that the  classifying space is a finite product of infinite projective 
  spaces.)
  {\it This means that for all toric and spherical varieties over separably closed fields to which \cite[Theorem 1.2]{K} 
  (with finite coefficients prime to the characteristic) applies, the conclusion of
  \cite[Theorem 1.2]{K} is already implied by Thomason's Theorem.}
 
\vskip .2cm

  \subsection{Computation of the homotopy groups of the derived completion}
  \label{comput.homotopy.groups}
  Assume that $\rmX$ is a $\rmG$-scheme as in ~\ref{G.quasiproj}, with $\rmG= \rmT$ a split torus. Then the approximation $\rmE\rmG^{gm,m}{\underset {\rmG} \times}\rmX$ is always a scheme,  and therefore, for a fixed value of $-s-t$, the homotopy groups of its $\bG$-theory may be
  computed using the spectral sequence with $m$ sufficiently large:
  \[E_2^{s,t} = \H_{BM,M}^{s-t}(\rmE\rmG^{gm,m}{\underset {\rmG} \times}\rmX, {\mathbb Z}(-t)) \Ra \pi_{-s-t}(\bG(\rmE\rmG^{gm,m}{\underset {\rmG} \times}\rmX)) \cong \pi_{-s-t}(\bG(\rmX, \rmG)\compl_{\rho_{\tilde \rmG, m}}) \]
The last isomorphism is from ~\eqref{key.weak.eq.2} and $\H_{BM,M}$ denotes Borel-Moore motivic cohomology, which identifies with the higher equivariant Chow groups. (This identification shows that the $E_2^{s,t}=0$ for all but finitely many values of $s$ and $t$, once $-s-t$ is fixed.  Therefore, 
 these $E_2$-terms, and hence the abutment, will be independent of $m$, if $m>>0$.) In addition, rationally, the homotopy groups 
$\pi_{-s-t}(\bG(\rmX, \rmG)\compl_{\rho_{\tilde \rmG, m}})$ identify with the rational $\rmG$-equivariant higher Chow  groups. For example, if $\rmX$ has only finitely many $\rmG$-orbits, 
 then the homotopy groups of the derived completion may be computed from the higher Chow-groups of the classifying spaces of
  the stabilizer groups. Such computations will hold if $\rmX$ is {\it any} toric variety and $\rmG$ denotes the corresponding torus. 
  For more general groups other than a split torus, there is still
  a similar spectral sequence, but only for the full $\rmE\rmG^{gm}{\underset {\rmG} \times}\rmX$. This spectral sequence converges only conditionally, and computes the homotopy groups of the full derived completion.
  \subsection{Actions on Algebraic spaces} Recall that Theorem ~\ref{main.thm.1}(iii) applies to actions by split tori
on algebraic spaces. Here we will mention some well-known cases where one is forced to consider algebraic spaces.
Consider the action of a finite group on a scheme. If each orbit is not contained in an affine subscheme, then the quotient
for the group action may not exist as a scheme, but only as an algebraic space. A well-known example is the non-projective
proper three-fold due to Hironaka (see \cite{Hir} or \cite[p. 443]{Hart}) provided with the action of the cyclic group of order two. In this case,
the quotient exists only as an algebraic space. 
\vskip .2cm
Thus taking geometric quotients of schemes by actions of finite and reductive groups on schemes often produce algebraic spaces
that are not schemes. {\it Therefore, derived completion theorems for algebraic spaces provided with actions by split tori, such as in Theorem ~\ref{main.thm.1}(iii),
greatly increases the breadth and utility  of our theorems.}
\vskip .2cm
\subsection{Higher Equivariant Riemann-Roch theorems}
An immediate consequence of our main theorems is a strong form of Riemann-Roch theorem relating higher equivariant
$\bG$-theory with other forms of Borel-style equivariant homology theories for actions of linear algebraic groups on
schemes (and also algebraic spaces, if the group actions are by split tori.) We consider this briefly as follows. Let $f:\rmX \ra \rmY$ denote
a proper $\rmG$-equivariant map of finite cohomological dimension between two $\rmG$-quasi-projective schemes. Let $\tilde \rmG$ denote
either a ${\rm GL}_n$ or a finite product of groups of the form ${\rm GL}_n$ containing $\rmG$ as a closed subgroup-scheme. Then 
the squares
\[ \xymatrix{ {{\bf G}(\rmX, \rmG)} \ar@<1ex>[r] \ar@<1ex>[d]^{Rf_*} & {{\bf G}(\rmX, \rmG)\compl_{\rho_{{\tilde \rmG}, \alpha(m)}}} \ar@<1ex>[r] \ar@<1ex>[d]^{Rf_*} & {{\bf G}({\rm E}{\tilde \rmG}^{gm,m}\times \rmX, \rmG)} \ar@<1ex>[d]^{Rf_*} \ar@<1ex>[r]^{\simeq} & {{\bf G}({\rm E}{\tilde \rmG}^{gm,m}{\underset {\rmG} \times}\rmX)} \ar@<1ex>[d]^{Rf_*}\\
              {{\bf G}(\rmY, \rmG)} \ar@<1ex>[r]  & {{\bf G}(\rmY, \rmG)\compl_{\rho_{{\tilde \rmG}, \alpha(m)}}} \ar@<1ex>[r]  &  {{\bf G}({\rm E}{\tilde \rmG}^{gm,m}\times \rmY, \rmG)} \ar@<1ex>[r]^{\simeq} & {{\bf G}({\rm E}{\tilde \rmG}^{gm,m}{\underset {\rmG} \times}\rmY)}}
\]
homotopy commute, for any fixed integer $m \ge 0$, with $\alpha(m)$ as in Definition ~\ref{funct.alpha}. (The commutativity of the left-square follows from 
Proposition ~\ref{factoring.through.part.compl}(ii). The composition of the left two horizontal maps is given by sending a complex of $\rmG$-equivariant
coherent sheaves on $\rmX$ ($\rmY$) to its pull-back on ${\rm E}{\tilde \rmG}^{gm,m}{\underset {\rmG} \times}\rmX$ (${\rm E}{\tilde \rmG}^{gm,m}{\underset {\rmG} \times}\rmY)$), \res. Therefore, the homotopy commutativity of the of the two left-most squares follows from Proposition ~\ref{factoring.through.part.compl}.)
One may compose with the usual
Riemann-Roch transformation into any equivariant Borel-Moore homology theory such as equivariant higher Chow groups, to obtain
higher equivariant Riemann-Roch theorems that hold in general. When the group $\rmG$ is a split torus, one may also assume 
in the above theorem that $\rmX$ and $\rmY$ are separated algebraic spaces of finite type (over the given base field). Such equivariant Riemann-Roch theorems will be pursued separately in forthcoming work.
\vskip .2cm
It may be important to point out that, in view of the difficulties with 
usual Atiyah-Segal type completion, (for example as in \cite{K}), such Riemann-Roch theorems are currently known only for
Grothendieck groups, unless one assumes that both $\rmX$ and $\rmY$ are projective smooth schemes. 
\subsection{Actions on projective smooth schemes over a field}
Finally we proceed to show  that for projective smooth schemes over field provided with an action by a connected split reductive group, the homotopy groups of the 
derived
completions in fact reduce to the usual completions of the homotopy groups of equivariant K-theory at the augmentation ideal and thereby complete the
proof of Theorem ~\ref{comp.2}. This shows that the results of \cite[Theorem 1.2]{K} all follow as a corollary to our more general derived completion theorems.
\begin{proposition}
\label{triv.action}
 Let $\rmX$ denote a scheme with a trivial action by a split torus $\rmT$, all defined over a field $k$. Then 
$\pi_*(\bK(\rm{\rm X}, \rmT))\compl_{I_{\rmT}}  \cong \pi_*(\bK(\rm{\rm X}, \rmT) \compl_{\rho_{\rmT}})$  where the completion on the left denotes the
completion at the augmentation ideal $I_{\rmT}$ in the usual sense and the completion on the right denotes the derived completion along
$\rho_{\rmT}$.
\end{proposition}
\begin{proof} Observe that $\pi_*(\bK({\rm X}, \rmT)) \cong \pi_*(\bK(\rmX)) {\underset {\mathbb Z} \otimes} \rmR(\rmT)$. Observe also 
 that $\rmR(\rmT)$ and $\rmR(\rmT)/\rmI_{\rmT}^n$ are all flat as modules over ${\mathbb Z}$. Therefore, we obtain:
\[\pi_*(\bK(\rmX)) {\underset {\mathbb Z} \otimes} \rmR(\rmT) {\overset L{\underset {\rm {R(T)}} \otimes }} \rmR(\rmT)/\rmI_{\rmT}^n \simeq \pi_*(\bK(\rmX)) {\underset {\mathbb Z} \otimes} \rmR(\rmT)/\rmI_{\rmT}^n.\]
At this point, Proposition ~\ref{derived.comp.for.rings} 
shows that the derived completion of the $\rmR(\rmT)$-module, $\pi_*(\bK(\rmX)) {\underset {\mathbb Z} \otimes} \rmR(\rmT)$ at the ideal $\rmI_{\rmT}$
identifies with the usual completion, namely $\pi_*\bK(\rmX) \otimes \rmR(\rmT)\compl_{\rmI_{\rmT}}$. (This is possible because,
for each fixed integer $n$, the $n$-th homotopy group of $\holimr \{\H(\pi_n(\bK(\rmX){\underset {\mathbb Z} \otimes} \rmR(\rmT)/\rmI_{\rmT}^r)|r\}$ identifies
 with $\limr\{\pi_n(\bK(\rmX){\underset {\mathbb Z} \otimes} \rmR(\rmT)/\rmI_{\rmT}^r)|r\}$ and the other homotopy groups of $\holimr\{\H(\pi_n(\bK(\rmX){\underset {\mathbb Z} \otimes} \rmR(\rmT)/\rmI_{\rmT}^r)|r\}$ are trivial.)
Finally, the first spectral sequence in ~\eqref{der.compl.props}
degenerates identifying the homotopy groups $\pi_*(\bK(\rm{\rm X}, \rmT)\compl_{\rho_{\rmT}})$ with the derived completion of $\pi_*\bK(\rmX) \otimes \rmR(\rmT)$ at $\rmI_{\rmT}$, which by the above arguments identify with 
$\pi_*\bK(\rmX) \otimes \rmR(\rmT)\compl_{\rmI_{\rmT}}  \cong \pi_*(\bK(\rm{\rm X}, \rmT))\compl_{\rmI_{\rmT}}$.
\end{proof}
\vskip .2cm
An immediate  consequence of these results is Theorem ~\ref{comp.2} stated in the introduction.
\vskip .2cm \noindent
{\bf Proof of Theorem ~\ref{comp.2}}. 
 First, Theorem ~\ref{T.comp.vs.G.comp} and a standard argument as in the proof of Theorem ~\ref{main.thm.3}, making use of the proper smooth $\rmG$-equivariant map
$\rmG{\underset {\rmB} \times}\rmX \ra {\rmX}$, reduce  this to actions by maximal tori. 
 Next, making use of Bialynicki-Birula decomposition, for the  action of a maximal torus $\rmT$ on the scheme $\rmX$, one  shows
that the theorem holds for the action of $\rmT$ on $\rmX$ by reducing it to the trivial action of $\rmT$ on the fixed point scheme $\rmX^{\rmT}$, at which point 
Proposition ~\ref{triv.action} applies. (A key observation here is the following: there is an ordering of the connected components of the fixed point 
scheme $\rmX^{\rmT}= \sqcup_{i=0}^n \rmZ_i$ so that there exists a corresponding filtration 
$\{\phi =\rmX_{-1} \subsetneq \rmX_0 \subsetneq \rmX_1 \subsetneq \cdots \subsetneq \rmX_n = \rmX\}$ so that each $\rmX_i-\rmX_{i-1} \ra \rmZ_i$ is a $\rmT$-equivariant
affine space bundle. Therefore one may show $\bK(\rm{\rm X}, \rmT) \simeq \vee_i \bK(\rmZ_i, \rmT)$.)  This proves (i). 
Theorem ~\ref{main.thm.1}(i) with $\rmG= {\rm GL}_n$ and $\rmH = \rmG$ proves that $\pi_*(\bK(\rmX, \rmG)\compl_{\rho_{{\rm GL}_n}}) \cong 
\pi_*(\bK(\rmE{ \rm GL}_n{\underset {\rmG} \times}\rmX)) \cong \pi_*(\bK(\rmE{ \rm G}{\underset {\rmG} \times}\rmX))$, since ${\rm E}{\rm GL}_n$ also satisfies the properties
required of ${\rm E}{\rm G}$: see ~\eqref{uniqueness.pro.1}. But by \cite[Theorem 1.2]{K}, this identifies with $\pi_*(\bK(\rmX, \rmG))\compl_{\rmI _{\rmG}}$. This proves (ii).
\qed
\subsection{Various Extensions of our current results}
The key reason we have restricted to equivariant $\bG$-theory in this paper is because of the readily available localization sequences.
However, already by
the late 1980s (see \cite{ThTr}), Thomason had established localization sequences for the K-theory of perfect complexes and recent work, for example \cite{Sch}, considers
an abstract framework in which localization sequences for K-theory holds: this framework applies to the equivariant situation by \cite[section 5]{Sch}. 
 On considering the associated homotopy invariant version, one also recovers
the homotopy property.
\vskip .2cm
One of our immediate goals therefore is to extend the derived completion theorems proven here to {\it equivariant (homotopy invariant) Algebraic K-Theory of perfect complexes} and
also to {\it actions by families of linear algebraic groups}. We expect these to have important consequences, one of these being to
the vanishing of equivariant K-theory in negative degrees (i.e. the equivariant analogue of Weibel's conjecture: see \cite{CHSW}.)
\section{Appendix A: Chain complexes vs. abelian group spectra}
In this section we recall certain well-known relations between the above two categories, the basic references
being  \cite{Ship07} and \cite{RS}. Since this correspondence is discussed in detail
in the above references, we will simply summarize the key points for the convenience of the reader.
\vskip .2cm
A chain complex will mean one with
differentials of degree $-1$. Let $R$ denote a commutative  ring with $1$. Then $\H(R)$ will denote
the Eilenberg-Maclane spectrum defined as follows. The space in degree $n$, $\H(R)_n$ is the 
underlying simplicial set of  $R\otimes S^n$ which in  degree $k$ is the free $R$-module on 
the non-base point $k$-simplices of $S^n$ (and
with the base point identified with $0$).
 The symmetric group $\Sigma_n$ acts on $S^n \cong (S^1)^{\wedge n}$ and therefore on $\H(R)_n$.
Clearly there is a map
from the simplicial suspension $S^1 \wedge \H(R)_n \ra \H(R)_{n+1}$ (compatible with the action of $\Sigma_{n+1}$), so that this defines a symmetric 
spectrum. This is the {\it Eilenberg-Maclane spectrum associated to $R$}. 
One may observe this is an $\Omega$-spectrum and that it is also a symmetric ring spectrum.
\vskip .2cm
Next one may define a functor $\H: Mod(R) \ra Mod(\H(R))$, where $Mod(R)$ ($Mod(\H(R)))$) 
denotes the category of modules over $R$ (module spectra over $\H(R)$, 
\res) by $\H(M) = M \otimes _R \H(R)$: the $n$-th space $\H(M)_n$ is now $M \otimes _R (R \otimes S^n)$.
\vskip .2cm
Though the normalizing functor $N$ sending a simplicial abelian group to its associated normalized chain complex
is symmetric monoidal, its familiar inverse $\Gamma$ (defined by the classical Dold-Kan equivalence) is
{\it not} symmetric monoidal. This is the main difficulty in extending the functor $\H$ to a monoidal
functor defined on all unbounded chain complexes of $R$-modules with the usual tensor product as the
monoidal structure. Nevertheless, solutions to this problem have been worked out in detail
in \cite{Ship07} and \cite{RS}. The main result then is the following. 
\vskip .2cm
Let $C(Mod(R))$ denote the category of
unbounded chain complexes (i.e. with differentials of degree $-1$) of $R$-modules, with the tensor product of
chain complexes as the monoidal product and provided with the projective model structure where the
fibrations are degree-wise surjections, weak-equivalences are maps that induce isomorphism on homology
and the cofibrations are defined by the lifting property.  Let $Mod (\H(R))$ denotes the category of 
symmetric module spectra
over $\H(R)$, provided with $\quad \wedge_{\H(R)} \quad $ as the monoidal product and with the model structure
defined on it as in \cite[section 4]{SS}. Then these two model categories are related by  functors
\be \begin{equation}
     \label{H.functor}
\H: C(Mod(R)) \ra Mod(\H(R)), \Theta: Mod(\H(R)) \ra C(Mod(R))
\end{equation} \ee
so that they are part of a Quillen equivalence between the corresponding derived categories. (See \cite[p. 6]{Ship07}.)
Since the two functors $\H$ and $\Theta$ are weak-monoidal functors, one also obtains a weak-equivalence 
\be \begin{equation}
     \label{monoidal.prop.H}
\H(M){\overset L { \underset {\H(R)} \wedge}} \H(N) \simeq \H(M {\overset L{\underset R \otimes}} N), \ M, N \eps C(Mod(R)).
    \end{equation} \ee
that is natural in the arguments $M$ and $N$. Here the left-derived functor of ${\underset {\H(R)} \wedge}$ 
(${\underset R \otimes}$) is defined by replacing one of the arguments
by a cofibrant replacement.
The above weak-equivalence is not induced by a map from the left-hand-side to the right-hand-side: rather,
the left-hand-side and the right-hand-side are related by maps which are weak-equivalences into intermediate spectra.
Nevertheless, this suffices for us to reduce computations of the left-hand-side to computing $M{\overset L {\underset R \otimes}}N$.
\vskip .2cm
One may make use of the discussion on model structures for $C(Mod(R))$ given below to find suitable cofibrant 
replacements of a given
$K \eps C(Mod(R))$. However, if $K=S$ is actually a commutative algebra over $R$, one may make use of the following device to produce
a particularly nice cofibrant replacement of $S$ in the category of commutative dg-algebras over $R$. This is discussed in \cite[proof of Lemma 5.2]{C08}.
\vskip .2cm
\subsubsection{\bf Resolutions via the free commutative algebra functor}
\label{res.free.comm.alg}
Let $R$ denote a commutative ring with $1$ and let $S$ denote a commutative algebra over $R$. Given a set $\rmX$, one first forms the free $R$-module on $\rmX$:
this will be denoted $F_R(X)$. Now one takes the $R$-symmetric algebra on $F_R(X)$: this will be denoted $\tilde F_R(X)$. Observe that $\tilde F_R(X)$
is a graded $R$-module which is free in each degree. Let $(R-algebras)$ denote the category of all commutative $R$-algebras and let 
(sets) denote the category of all (small) sets. Let $\tilde U: (R-algebras) \ra (sets)$  be the forgetful functor. Then the functor
$\tilde F_R$ defined above will be left-adjoint to $\tilde U$. Together, these two functors define a triple and we may apply them in the usual
manner to $S$ to produce a simplicial object in the category of $R$-algebras, $\tilde S_{\bullet}$ together with a quasi-isomorphism to $S$.
Since the normalizing functor sending a simplicial abelian group to the  associated  
normalized chain complex is strictly monoidal,
one observes that $\tilde S = N(\tilde S_{\bullet})$  is a  dg-algebra, trivial in negative degrees, together with a quasi-isomorphism
to $S$. (One may also observe that if each  $\tilde S_n$ is a free $R$-module, then so is 
$N(\tilde S_{\bullet})_n$.) In view of the explicit construction above, one may verify the following isomorphism:
\be \begin{equation}
     \label{tens.struct}
\H(\tilde S){\underset {\H(R)} \wedge} \H(M) \cong \H(\tilde S {\underset R \otimes} M), \, M \eps C(Mod(R))
    \end{equation} \ee
\vskip .1cm \noindent
where we let $\H(\tilde S) = \H(\tilde S_{\bullet}) =\hocolimD \{\H(\tilde S_n)|n\}$ and 
$\H(\tilde S{\underset R \otimes}M) = \H(\tilde S_{\bullet}{\underset R \otimes}M) =\hocolimD 
\{\H(\tilde S_{n}{\underset R \otimes}M)|n\}$. 
\vskip .2cm
Finally it is important to observe that  $\H(\tilde S_n)$ for each $n \ge 0$ and therefore 
$\hocolimD \{\H(\tilde S_n)|n\}$ is commutative algebra spectrum over $\H(R)$. Moreover, 
each $\H(\tilde S_n)$ is a cofibrant object in $Mod(\H(R))$ and hence so is $\hocolimD \{\H(\tilde S_n)|n\}$. Therefore
we will make use of this cofibrant replacement, along with Remark ~\ref{alternate.form}. 
\subsubsection{Model structure for the category of (unbounded) chain complexes}
Let $R$ denote a commutative ring and let $\C (R)$ denote the category of all (i.e. possibly unbounded) 
chain-complexes of 
$R$-modules  with a differential of degree $-1$. This category may be given the cofibrantly generated
 model category structure described in \cite[Definition 2.3.3]{Hov00} and also \cite{Fausk}. Recall the 
generating cofibrations and generating trivial fibrations are defined as follows. For each
integer $n$, one lets $S^n$ to be the complex concentrated in degree $n$ where it is $R$ and
trivial elsewhere. $D^n$ will be the complex concentrated in degrees $n$ and $n-1$ where they
are both $R$ with the differential $D^n_n \ra D^n_{n-1}$ being the identity. Then it is shown
in \cite[Proposition 4.2.13]{Hov00} that this model category is a cofibrantly generated symmetric
monoidal model category with the generating cofibrations $I$, being maps of the form
$\{S^{n-1} \ra D^n|n\}$ and with the generating trivial cofibrations $J$, being maps of
the form $\{0 \ra D^n|n \}$. The monoidal structure is given by the tensor product of chain 
complexes and weak-equivalences are quasi-isomorphisms of chain complexes.
It is straight-forward to verify that this model category satisfies the monoidal axiom: : see \cite[Lemma 4.6]{Fausk} for more details.
\subsubsection{Homotopy inverse limits of cosimplicial objects of chain complexes}
\label{holimD.ab}
 Let $R$ denote a commutative ring and let $\C(R)$ denote the category of chain complexes of possibly
 unbounded complexes of 
$R$-modules with differentials of degree $-1$. We provide $\C(R)$ with the model structure discussed above. 
In order to define homotopy inverse limits of
diagrams in $\C(R)$, it suffices to first observe that the model category $\C(R)$ is {\it
quasi-simplicial} (see \cite[Definition 5.1]{Fausk}, though not simplicial. Then the usual definition of the homotopy inverse
limit as an end carries over and therefore the homotopy inverse limit will have the usual properties.

\section{Appendix B: Motivic slices and independence on the choice of the geometric classifying space}
The work of \cite{VV}, \cite{LK} and \cite{Pel} show that now one can define a sequence of functors $f_n: \Spt_{S^1}(k) \ra \Spt_{S^1}(k)$
(where $\Spt_{S^1}(k)$ denotes the ${\mathbb A}^1$-localized category of $S^1$-spectra on the Nisnevich site of $Spec \, k$) 
so that the following properties are true for any  spectrum $E \eps \Spt_{S^1}(k)$ which we will assume is $-1$-connected and 
${\mathbb A}^1$-homotopy invariant:
\subsubsection{}
\begin{enumerate}[\rm(i)]
\label{props.mot.Post}
 \item One obtains a map $f_n E \ra E$ that is natural in the spectrum $E$ and universal for maps $F \ra E$, with 
$F \eps \Sigma_{{\P^1}}^n\Spt_{S^1}(k)$. 
\item One also obtains a tower of maps $\cdots \ra f_{n+1}E \ra f_n E \ra \cdots  \ra f_0E = E$. Let $s_pE$ denote the
canonical homotopy cofiber of the map $f_{p+1}E \ra f_pE$ and let $s_{\le q-1}E$ be defined as the canonical homotopy cofiber of the map $f_qE \ra E$. 
Then one may show readily (see, for example, \cite[Proposition 3.1.19]{Pel}) that the canonical homotopy fiber of 
the induced map $s_{\le q}E \ra s_{\le q-1}E$ identifies also with $s_q(E)$. One then also obtains a tower: $\cdots \ra s_{\le q}E \ra s_{\le q-1}E \ra \cdots$. 
\item Let $\rmY $ be a smooth scheme of finite type over $k$ and let $\rmW \subseteq \rmY$ denote a closed not necessarily smooth subscheme so that
$codim_{\rmY}(\rmW) \ge q$ for some $q \ge 0$. Then the map $f_qE \ra E$ induces a weak-equivalence (see \cite[Lemma 7.3.2]{Lev} and also 
\cite[Lemma 2.3.2]{LK}) where $\Map$ denotes the mapping  spectrum:
\[\Map(\Sigma_{S^1}(\rmY/\rmY-\rmW)_+, f_qE) \ra \Map(\Sigma_{S^1}(\rmY/\rmY- \rmW)_+, E), E \eps \Spt_{S^1}(k).\]
\item It follows that, then, 
\[\Map(\Sigma_{S^1}(\rmY/\rmY- \rmW)_+, s_{\le q-1}E) \simeq *\]
\end{enumerate}
\vskip .5cm \noindent
Next assume that $\rmG$ is a linear algebraic group and $(\rmW, \rmU)$,  $(\bar \rmW,\bar \rmU)$ are both good pairs  for $\rmG$. Let $\{(\rmW_m, \rmU_m)|m \ge 1\}$
 and $\{(\bar \rmW_m, \bar \rmU_m)|m \ge 1\}$ denote the associated admissible gadgets. Then, since $\rmG$ acts freely on both $\rmU_m$ and $\bar \rmU_m$, it 
is easy to see  that $(\rmW \times \bar \rmW, \rmU \times \bar \rmW \cup \rmW \times \bar \rmU)$ is also a good pair for $\rmG$ with
respect to the diagonal action on $\rmW \times \bar \rmW$. Moreover, under the same hypotheses,  one may readily check shows that $\{\rmW_m \times \bar \rmW_m, \rmU_m \times \bar \rmW_m \cup \rmW_m \times \bar \rmU_m)|m \ge 1\}$
 is also an admissible gadget for $\rmG$ for the diagonal action on $\rmW \times \bar \rmW$. Let $\rmX$ denote a smooth scheme of finite type over $k$
on which $\rmG$ acts.
\vskip .2cm
Since $\rmG$ acts freely on both $\rmU_m$ and $\bar \rmU_m$, it follows that $\rmG$ has a free action on $\rmU_m \times \bar \rmW_m$ and also on
$\rmW_m \times \bar \rmU_m$. We will let $\widetilde \rmU_m = \rmU_m \times \bar \rmW_m \cup \rmW_m \times \bar \rmU_m$ for the following discussion.
One may now compute the codimensions 
\be \begin{align}
\label{codims}
codim_{\widetilde \rmU_m{\underset {\rmG} \times}\rmX}({\widetilde \rmU}_m {\underset {\rmG} \times}\rmX- \rmU_m \times \bar \rmW_m{\underset {\rmG} \times}\rmX) &= codim_{\rmW_m}(\rmW_m -\rmU_m), \mbox{ and }  \\
codim_{\widetilde \rmU_m{\underset {\rmG} \times}\rmX}({\widetilde \rmU}_m{\underset {\rmG} \times}\rmX - \rmW_m \times \bar \rmU_m{\underset {\rmG} \times}\rmX) &= codim_{\bar \rmW_m}(\bar \rmW_m -\bar \rmU_m) 
\end{align} \ee
\vskip .2cm \noindent
In view of ~\ref{props.mot.Post}(iv), it follows that the induced maps
\be \begin{multline}
     \begin{split}
\label{uniqueness.pro.1}
\Map(\Sigma_{S^1}\rmU_m \times _{\rmG}\rmX_+, s_{\le q-1}E) \simeq \Map(\Sigma_{S^1}(\rmU_m \times \bar \rmW_m) \times_{\rmG}\rmX_+, s_{\le q-1}E)\\
 \leftarrow \Map(\Sigma_{S^1}{\widetilde \rmU}_m\times _{\rmG} \rmX_+, s_{\le q-1}E) \mbox{ and }\\
 \Map(\Sigma_{S^1}\bar \rmU_m \times _{\rmG}\rmX_+, s_{\le q-1}E) \simeq \Map(\Sigma_{S^1}( \rmW_m \times \bar \rmU_m ) \times_{\rmG}\rmX_+, s_{\le q-1}E) \\
\leftarrow \Map(\Sigma_{S^1}{\widetilde \rmU}_m\times _{\rmG} \rmX_+, s_{\le q-1}E) 
\end{split}  \end{multline} \ee
are both weak-equivalences if $codim_{\rmW_m}(\rmW_m -\rmU_m)$ and $codim_{\bar \rmW_m}(\bar \rmW_m -\bar \rmU_m) $ are both 
greater than or equal to $q$. The first weak-equivalences in ~\eqref{uniqueness.pro.1} are provided by the homotopy property for the 
spectrum $E$, which is inherited by the slices.  
Therefore, both maps in 
~\eqref{uniqueness.pro.1} induce a  weak-equivalence on taking the homotopy inverse limit as $m, q \ra \infty$. Finally, one may make 
use of \cite[Proposition 2.1.3]{Lev} 
to conclude that, since  $E$ is also assumed to be ${\mathbb A}^1$-homotopy invariant, 
$\holimq \Map(\Sigma_{S^1} \rmU_m \times _{\rmG}\rmX_+, s_{\le q-1}E) \simeq \Map(\Sigma_{S^1}\rmU_m \times _{\rmG}\rmX_+, E) $
and $\holimq \Map(\Sigma_{S^1}\bar \rmU_m \times _{\rmG}\rmX_+, s_{\le q-1}E) \cong \Map(\Sigma_{S^1}\bar \rmU_m \times _{\rmG}\rmX_+, E) $.

\section{Appendix C:  Additional results on Derived completion}
In this section we collect together a few technical results that have been used in the body of the paper.
A  notion that we have found useful is that of {\it ideals} in ring spectra, which we use in Lemma ~\ref{res.1}, 
Proposition ~\ref{factoring.through.part.compl} and also appears 
in \cite[Corollary 6.7]{C08}. (See also \cite{Hov14} for related results.)
\subsection{Ideals in ring spectra}
\label{ideals}
\begin{definition}
Let $E$ be a commutative ring spectrum, which we will assume is cofibrant in the category of 
algebra spectra over the sphere spectrum. An {\it ideal} in $E$ is an $E$-module spectrum, $I$, which is {\it cofibrant
as an $E$-module spectrum}, together
 with a specified map $i:I \ra E$ of $E$-module-spectra so that 
the canonical homotopy cofiber of $i$ is weakly-equivalent to  a commutative ring spectrum {\it under} $E$ (i.e. with a map of commutative ring spectra from $E$). Moreover, we require that
these can be chosen functorially in $i$.
\end{definition}
\begin{lemma}
\label{ideals}
Let $f: E' \ra E$ denote a map of commutative ring-spectra. 
\vskip .2cm
(i) If $i:I \ra E$ is an ideal in 
$E$, then the canonical homotopy pull-back of $i$ by $f$ is an ideal in $E'$. 
\vskip .2cm
(ii) If $j:J \ra E'$ is an ideal in
$E'$, and $E$ is cofibrant as an $E'$-module spectrum, then $i=j{\underset {id_{E'}} \wedge} id_E: J{\underset {E'} \wedge} E \ra E$ is an ideal in $E$.
\end{lemma}
\begin{proof}  Observe first that, $E'$, $E$ and $I$ are all module-spectra over $E'$ and that the given maps
 between them are also maps of $E'$-module-spectra. 
\vskip .2cm
Next we consider (i). One may functorially replace $i$ by a map which
is a fibration in the category of $E$-module spectra. Let $j:J \ra E'$ denote the pull-back of $I$ over $E'$ by $f$. Therefore, it  follows that $J$ is an $E'$-module spectrum. The homotopy fiber of the induced map $J \ra E'$, which
will be denoted $hofib(j)$ maps  to the homotopy fiber of 
the map $i:I \ra E$ (denoted $hofib(i)$) and the above map is a weak-equivalence. 
By our hypotheses, $hocofib(i)$ (i.e. the homotopy cofiber of $i$, which is weakly equivalent to the suspension of $hofib(i)$, i.e. $\Sigma hofib(i)$) 
maps  by a weak-equivalence to a commutative ring-spectrum $E/I$ under $E$. 
Since $f$ is map of ring-spectra, $E/I$ is a commutative
a ring-spectrum under $E'$ as well. Moreover the choice of $E/I$ is functorial in $i$ and hence in $i'$.
\vskip .2cm
(ii) Clearly 
$J{\underset {E'} \wedge} E$ is an $E$-module. The map $i$ corresponds to the map 
$J{\underset {E'} \wedge} E \ra E'{\underset {E'} \wedge} E =E$. Since $E$ is assumed to be cofibrant
as an $E'$-module, 
\[hofib(j) {\underset {E'} \wedge} E \ra J{\underset {E'} \wedge} E \ra E'{\underset {E'} \wedge} E  =E \]
 is a fiber sequence. Let $\Sigma hofibf(j) \ra hocofib(j) \simeq E'/J$ denote a weak-equivalence to  a commutative ring spectrum under $E'$ given by the ideal-structure of $j$.
 Since $E$ is a commutative ring spectrum under $E'$ (and cofibrant as an $E'$-module), it follows that,
$hocofib(j) {\underset {E'} \wedge} E \simeq E'/J {\underset {E'} \wedge} E$. The latter is  clearly a commutative ring spectrum under $E$.
This proves (ii). 
\end{proof}

\begin{proposition}
 \label{key.compare.compl}
(i) Let $R$ denote a commutative Noetherian ring with unit and let $I$, $J$ denote two ideals in $R$ with $J \subseteq I$ and where there exists a positive integer
$N$ so that $I^N \subseteq J$. If $M$ is any chain complex (i.e. with differentials of degree $-1$) 
of $R$-modules trivial in sufficiently small degrees, then $M\compl_I \simeq M\compl_J$, where the completion denotes the derived completion.
\vskip .2cm
(ii) Let $E$ denote a $-1$-connected commutative ring spectrum so that $\pi_0(E)$ is a commutative Noetherian ring. Let $I, J$ denote two ideals in $E$ so that 
$\pi_0(J) \subseteq \pi_0(I)$ and there exists a positive integer $N$ so that $\pi_0(I)^N \subseteq \pi_0(J)$. If $M$ is any $t$-connected
$E$-module spectrum (for some integer $t$), then one obtains a weak-equivalence $M\compl_{I} \simeq M \compl_J$.
\end{proposition}
\begin{proof} We may reduce (ii) to (i) by making use of the 
arguments in ~\ref{der.compl.props}. Now (i) follows
 by invoking Proposition ~\ref{derived.comp.for.rings}. (See for example, the proof of ~\eqref{pro.isom}.)
\end{proof}
\vskip .2cm
Next we will prove that indeed the derived completion of modules over a commutative Noetherian ring at an ideal depends only the 
radical of the ideal.
\vskip .2cm
First we show how to obtain functorial flat resolutions  that are compatible with pairings. Let $R$ denote a commutative Noetherian ring
with $1$. Let $Mod(R)$ denote the category of all $R$-modules and let $(sets)$ denote the category of all small sets.
Then we obtain the following functors: $S:Mod(R) \ra (sets)$ and $F: (sets) \ra Mod(R)$ by letting $S(M) = Hom_{R}(R, M)$ and
$F(T) =$ the free $R$-module on the set $T$. Here $Hom_R$ denotes $Hom$ in the category $Mod(R)$. Then $F$ is left-adjoint to $S$ and they provide a triple:
\vskip .2cm
$F\circ S: Mod(R) \ra Mod(R)$.
\vskip .2cm \noindent
By iterating this triple, we obtain a simplicial resolution of any $R$-module $M$, which consists of flat-modules in each degree.
One may then apply the normalization functor that sends a simplicial abelian group to a chain complex (i.e. where the differentials
are of degree $-1$) and then to the associated co-chain complex that is trivial in positive degrees and with differentials of degree $+1$.
We will denote this functorial flat  resolution of an $R$-module $M$ by $\tilde M$ (or $\tilde M^{\bullet}$). This construction has the
following properties.
\vskip .2cm
(i) It is functorial in $M$. i.e if $f: M \ra N$ and $g: N \ra P$ are maps in $Mod(R)$, then one obtains induced maps
$\tilde f: \tilde M \ra \tilde N$ and $\tilde g: \tilde N \ra \tilde P$ so that $\widetilde {(g \circ f)} = \tilde g \circ \tilde f$.
\vskip .2cm
(ii) It is compatible with pairings. i.e. if $M \otimes_R N \ra P$ is a map of $R$-modules, one obtains an induced map
$\tilde M \otimes _R \tilde N \ra \widetilde {(M \otimes _R N)} \ra \tilde P$.
\vskip .2cm
\begin{theorem} Let $R$ denote a commutative Noetherian ring with $1$ and let $J \subseteq I$ denote two ideals in $R$ so that
 $\sqrt{I} = \sqrt{J}$. Then if $M$ is an $R$-module, not necessarily finitely generated, 
$M \compl_{I} \simeq M \compl_{J}$ where $M \compl_I$ ($M \compl_J$) denotes the derived completion of $M$ with respect to
 $I$ ($J$, \res).
\end{theorem}
\begin{proof} First observe that since $\sqrt{J} = \sqrt{I} \supseteq I$ and $I$ is finitely generated, it follows that $I^{n_0} \subseteq J$
for some $n_0 >>0$, i.e., now we obtain the chain of containments: $J^{n_0} \subseteq I^{n_0} \subseteq J \subseteq I$. Therefore, 
the conclusion follows from Proposition ~\ref{key.compare.compl}.
\end{proof}

%%%%%%%%%%%%%%%%%%%%%%%%%%%%%%%%%%%%%%%%%%%%%%%%%%%%%%%%%%%%%%%%%%%%%
%%%%%%%%%%%%%%%%%%%%%%%%%%%%%%%%%%%%%%%%%%%%%%%%%%%%%%%%%%%%%%%%%%%%%
%% REFERENCES
%%%%%%%%%%%%%%%%%%%%%%%%%%%%%%%%%%%%%%%%%%%%%%%%%%%%%%%%%%%%%%%%%%%%%

\vfill \eject
\end{document}